\documentclass{amsart}
\usepackage{amsmath}
\usepackage{amstext}
\usepackage{amssymb}
\usepackage{amsthm}
\swapnumbers
\theoremstyle{plain}%default
\newtheorem{thm}{Theorem}[section]
\newtheorem{lem}[thm]{Lemma}
\newtheorem{prop}[thm]{Proposition}
\newtheorem{cor}[thm]{Corollary}

\theoremstyle{definition}
\newtheorem{defi}[thm]{Definition}
\newtheorem{defs}[thm]{Definitions}

\newtheorem{ex}[thm]{Example}

\newtheorem{ntn}[thm]{Notation}
\newtheorem{rmd}[thm]{Reminder}
\newtheorem{rmds}[thm]{Reminders}
\theoremstyle{remark}
\newtheorem*{note}{Note}
\newtheorem{rmk}[thm]{Remark}
\newtheorem{rmks}[thm]{Remarks}

 \DeclareMathOperator{\height}{ht}
 \DeclareMathOperator{\nd}{end}
\DeclareMathOperator{\Ext}{Ext} \DeclareMathOperator{\Hom}{Hom}
 \DeclareMathOperator{\reg}{reg}
\DeclareMathOperator{\Ker}{Ker} \DeclareMathOperator{\Ima}{Im}
\DeclareMathOperator{\depth}{depth}
\DeclareMathOperator{\grade}{grade}
\DeclareMathOperator{\rank}{rank} \DeclareMathOperator{\anch}{anch}

\DeclareMathOperator{\Ass}{Ass} 
\DeclareMathOperator{\dir}{dir} \DeclareMathOperator{\red}{red}
 \DeclareMathOperator{\ara}{ara}
 \DeclareMathOperator{\Var}{Var}
 \DeclareMathOperator{\Spec}{Spec}
 \DeclareMathOperator{\Proj}{Proj}

\DeclareMathOperator{\Max}{Max}
\DeclareMathOperator{\Min}{Min} \DeclareMathOperator{\bnd}{bnd}
\def\Z{\mathbb Z}
\def\N{\mathbb N}

\def\X{\mathbb X}
\def\Y{\mathbb Y}
\def\K{\mathbb K}

\def\fa{{\mathfrak{a}}}

\def\fb{{\mathfrak{b}}}
\def\fB{{\mathfrak{B}}}
\def\fc{{\mathfrak{c}}}

\def\fd{{\mathfrak{d}}}

\def\fm{{\mathfrak{m}}}

\def\p{{\mathfrak{p}}}
\def\fp{{\mathfrak{p}}}

\def\fq{{\mathfrak{q}}}
\def\q{{\mathfrak{q}}}

\def\R{R_{\red}}
\def\nn{\relax\ifmmode{\mathbb N_{0}}\else$\mathbb N_{0}$\fi}
\def\lra{\longrightarrow}

\begin{document}

\title[SUPPORTING DEGREES OF MULTI-GRADED LOCAL COHOMOLOGY MODULES]{SUPPORTING DEGREES OF MULTI-GRADED
LOCAL COHOMOLOGY MODULES}
\author{MARKUS P. BRODMANN}
\address[Brodmann]{Institut f\"{u}r Mathematik, Universit\"{a}t Z\"{u}rich,
Winterthurerstrasse
190, 8057 Z\"{u}rich, Switzerland\\
{\it Fax number}: 0041-44-635-5706} \email{Brodmann@math.unizh.ch}
\author{RODNEY Y. SHARP}
\address[Sharp]{Department of Pure Mathematics,
University of Sheffield, Hicks Building, Sheffield S3 7RH, United Kingdom\\
{\it Fax number}: 0044-114-222-3769}
\email{R.Y.Sharp@sheffield.ac.uk}
\thanks{The second author was partially supported by the Swiss National
Foundation (Grant Number 20-103491/1), and the Engineering and
Physical Sciences Research Council of the United Kingdom (Grant
Number EP/C538803/1).}


\date{\today}


\begin{abstract}
For a finitely generated graded module $M$ over a positively-graded
commutative Noetherian ring $R$, the second author established in
1999 some restrictions, which can be formulated in terms of the
Castelnuovo regularity of $M$ or the so-called $a^{{\rm
*}}$-invariant of $M$, on the supporting degrees of a
graded-indecomposable graded-injective direct summand, with
associated prime ideal containing the irrelevant ideal of $R$, of
any term in the minimal graded-injective resolution of $M$. Earlier,
in 1995, T. Marley had established connections between finitely
graded local cohomology modules of $M$ and local behaviour of $M$
across $\Proj(R)$.

The purpose of this paper is to present some multi-graded analogues
of the above-mentioned work.
\end{abstract}

\maketitle

\setcounter{section}{-1}
\section{\sc Introduction}
Very briefly, the purpose of this paper is to explore multi-graded
analogues of some results in the algebra of modules, and
particularly local cohomology modules, over a commutative Noetherian
ring that is graded by the additive semigroup $\nn$ of non-negative
integers.

To describe the results that we plan to generalize, let $R =
\bigoplus_{n\in\nn}R_n$ be such a `positively-graded' commutative
Noetherian ring. Any unexplained notation in this Introduction will
be as in Chapters 12 and 13 of our book \cite{LC}. In particular,
the $\mbox{\rm *}$injective envelope of a graded $R$-module $M$ will
be denoted by $\mbox{\rm *}E(M)$ (see \cite[\S13.2]{LC}), and, for
$t \in \Z$, the $t$th shift functor (on the category $\mbox{\rm
*}\mathcal{C}(R)$ of all graded $R$-modules and homogeneous
$R$-homomorphisms) will be denoted by $ (\: {\scriptscriptstyle
\bullet} \:)(t)$ (see \cite[\S12.1]{LC}).

Let $\N$ denote the set of positive integers; set $R_+ :=
\bigoplus_{n \in \N}R_n$, the irrelevant ideal of $R$. For a graded
$R$-module $M$ and $\fp \in \mbox{\rm *}\Spec (R)$ (the set of
homogeneous prime ideals of $R$), we use $M_{(\fp)}$ to denote the
homogeneous localization of $M$ at $\fp$. For $i \in \nn$, the
ordinary Bass number $\mu^i(\fp,M)$ is equal to the rank of the
homogeneous localization $\left(\mbox{\rm
*}\Ext^i_{R}(R/\p,M)\right)_{({\mathfrak{p}})}$ as a (free) module
over $R_{({\mathfrak{p}})}/\p R_{({\mathfrak{p}})}$ (see R. Fossum
and H.-B. Foxby \cite[Corollary 4.9]{Fos-Fox74}).

Let $i \in \nn$, and consider a direct decomposition given by a
homogeneous isomorphism $$\mbox{\rm *}E^{i}(M)
\stackrel{\cong}{\longrightarrow} \bigoplus_{\alpha \in \Lambda_i}
\mbox{\rm *}E(R/{\mathfrak{p}}_{\alpha})(-n_{\alpha}),$$ for an
appropriate family $(\mathfrak{p}_{\alpha})_{\alpha \in \Lambda_i}$
of graded prime ideals of $R$ and an appropriate family
$(n_{\alpha})_{\alpha \in \Lambda_i}$ of integers. (See
\cite[\S13.2]{LC}.)

Suppose that the graded prime ideal $\fp$ contains the irrelevant
ideal $R_+$. In this case, the graded ring $R_{({\mathfrak{p}})}/\p
R_{({\mathfrak{p}})}$ is concentrated in degree $0$, and its $0$th
component is a field isomorphic to $k_{R_0}(\fp_0)$, the residue
field of the local ring $(R_0)_{\fp_0}$. Thus,
$$
\mu^i(\fp,M) = \dim_{k_{R_0}(\fp_0)} \left(\mbox{\rm
*}\Ext^i_{R}(R/\p,M)\right)_{({\mathfrak{p}})} = \sum_{t \in \Z}
\dim_{k_{R_0}(\fp_0)} \big(\big(\mbox{\rm
*}\Ext^i_{R}(R/\p,M)\big)_{({\mathfrak{p}})}\big)_t.
$$
In \cite{70}, it was shown that the graded $R_{({\mathfrak{p}})}/\p
R_{({\mathfrak{p}})}$-module $\left(\mbox{\rm
*}\Ext^i_{R}(R/\p,M)\right)_{({\mathfrak{p}})}$ carries information
about the shifts `$-n_{\alpha}$' for those $\alpha \in \Lambda _i$
for which $\fp_{\alpha} = \fp$. One has
\[
\mbox{\rm *}E(R/{\mathfrak{p}})(n) \not\cong \mbox{\rm
*}E(R/{\mathfrak{p}})(m) \mbox{~in~} \mbox{\rm *}\mathcal{C}(R)
\quad \mbox{~for~} m,n \in \Z \mbox{~with~} m \neq n,
\]
and, for a given $t \in \Z$, the cardinality of the set $\left\{
\alpha \in \Lambda_i : {\mathfrak{p}}_{\alpha} = {\mathfrak{p}}
\mbox{~and~} n_{\alpha} = t \right\}$ is equal to
$$\dim_{k_{R_0}(\fp_0)} \big(\big(\mbox{\rm
*}\Ext^i_{R}(R/\p,M)\big)_{({\mathfrak{p}})}\big)_t,$$ the
dimension of the $t$th component of $\big(\mbox{\rm
*}\Ext^i_{R}(R/\p,M)\big)_{({\mathfrak{p}})}$.

Let $\mbox{\rm *}\Var (R_+) := \{ \fq \in \mbox{\rm *}\Spec (R) :
\fq \supseteq R_+\}$. Let $\fp \in \mbox{\rm *}\Var (R_+)$, let $i
\in \nn$ and let $t \in \Z$. We say that $t$ is an {\em $i$th level
anchor point of $\fp$ for $M$\/} if
$$\big(\big(\mbox{\rm *}\Ext^i_{R}(R/\p,M)\big)_{({\mathfrak{p}})}\big)_t
\neq 0\mbox{;}$$ the set of all $i$th level anchor points of $\fp$
for $M$ is denoted by $\anch^i(\fp,M)$; also, we write
$$
\anch(\fp,M) = \bigcup_{j \in \nn}\anch^j(\fp,M),
$$ and refer to this as the set of {\em anchor points of $\fp$ for $M$}.
Thus $\anch^i(\fp,M)$ is the set of integers $h$ for which, when we
decompose
$$\mbox{\rm *}E^{i}(M) \stackrel{\cong}{\longrightarrow}
\bigoplus_{\alpha \in \Lambda_i} \mbox{\rm
*}E(R/{\mathfrak{p}}_{\alpha})(-n_{\alpha})$$ by means of a
homogeneous isomorphism, there exists $\alpha \in \Lambda_i$ with
$\fp_{\alpha} = \fp$ and $n_{\alpha} = h$. Note that $\anch^i(\fp,M)
= \emptyset$ if $\mu^i(\fp,M) = 0$, and that $\anch^i(\fp,M)$ is a
finite set when $M$ is finitely generated.

It was also shown in \cite{70} that, when the graded $R$-module $M$
is non-zero and finitely generated, the Castelnuovo regularity
$\reg(M)$ of $M$ is an upper bound for the set
$$\bigcup_{\fp \in \mbox{\rm
*}\Var (R_+)} \anch(\fp,M)$$ of all anchor points of $M$. Consequently, for each $i \geq 0$,
every $\mbox{\rm *}$indecomposable $\mbox{\rm *}$injective direct
summand $F$ of $\mbox{\rm *}E^i(M)$ with associated prime containing
$R_+$ must have $F_j = 0$ for all $j > \reg (M)$.

In \S\S\ref{mu},\ref{el} we shall present an analogue of this theory
for a standard multi-graded commutative Noetherian ring $S =
\bigoplus_{\mathbf{n} \in \nn^r} S_{\mathbf{n}}$ (where $r \in \N$
with $r \geq 2$). There is a satisfactory generalization of anchor
point theory to the multi-graded case, but we must stress now that
we have not uncovered any links between our multi-graded anchor
point theory and the fast-developing theory of multi-graded
Castelnuovo regularity (see, for example, Huy T\`ai H\`a \cite{ha}
and D. Maclagan and G. G. Smith \cite{MacSmi05}). This may be
because our multi-graded anchor point theory only yields information
about multi-graded local cohomology modules with respect to
$\nn^r$-graded ideals of $S$ that contain one of the components
$S_{(0,\ldots,0,1,0,\ldots,0)}$, whereas the ideal $S_+ :=
\bigoplus_{\mathbf{n} \in \N^r} S_{\mathbf{n}}$, which is relevant
to multi-graded Castelnuovo regularity, normally does not have that
property.

The short \S\ref{va} provides some motivation for our work in
\S\ref{tm}, where we provide multi-graded analogues of work of T.
Marley \cite{Marle95} about finitely graded local cohomology
modules. We say that a graded $R$-module $L = \bigoplus_{n\in\Z}
L_n$ is {\em finitely graded\/} precisely when $L_n \neq 0$ for only
finitely many $n \in \Z$. In \cite{Marle95}, Marley defined, for a
finitely generated graded $R$-module $M$,
$$
g_{\fa}(M) := \sup\left\{k \in \nn :  H^i_{\fa}(M) \mbox{~is
finitely graded for all~} i < k\right\},
$$
and he modified ideas of N. V. Trung and S. Ikeda in \cite[Lemma
2.2]{TruIke89} to prove that
$$
g_{\fa}(M) := \sup\left\{k \in \nn : R_+ \subseteq
\sqrt{\left(0:_RH^i_{\fa}(M)\right)} \mbox{~for all~}i <
k\right\}\mbox{;}
$$
he then used Faltings' Annihilator Theorem for local cohomology (see
\cite{Falti78} and \cite[Theorem 9.5.1]{LC}). In \S \ref{tm} below,
we shall obtain some multi-graded analogues of some of Marley's
results in this area.

\section{\sc Background results in multi-graded commutative algebra}
\label{in}

Let $R = \bigoplus_{g\in G} R_g$ be a commutative Noetherian ring
graded by a finitely generated, additively-written, torsion-free
Abelian group $G$. Some aspects of the $G$-graded analogue of the
theory of Bass numbers have been developed by S. Goto and K.-i.
Watanabe \cite[\S\S 1.2, 1.3]{GotWat78}, and it is appropriate for
us to review some of those here.

We shall denote by $\mbox{\rm *}\mathcal{C}^G(R)$ (or sometimes by
$\mbox{\rm *}\mathcal{C}(R)$ when the grading group $G$ is clear)
the category of all $G$-graded $R$-modules and $G$-homogeneous
$R$-homomorphisms of degree $0_G$ between them.  Projective
(respectively injective) objects in the category $\mbox{\rm
*}\mathcal{C}^G (R)$ will be referred to as $\mbox{\rm
*}$projective (respectively $\mbox{\rm *}$injective) $G$-graded
$R$-modules. Similarly, the attachment of `$\mbox{\rm *}$' to
other concepts indicates that they refer to the obvious
interpretations of those concepts in the category $\mbox{\rm
*}\mathcal{C}^G(R)$, although we shall sometimes use `$G$' instead
of `$\mbox{\rm *}$' in order to emphasize the grading group.
However, the following comments about $\mbox{\rm *}\Hom_R$ and the
$\mbox{\rm *}\Ext^i_R~(i \geq 0)$ may be helpful.

\begin{rmds}
\label{in.1}
Let
$M = \bigoplus_{g \in G} M_g$
and $N = \bigoplus_{g \in G} N_g$ be $G$-graded $R$-modules.

\begin{enumerate}
\item Let $a \in G$\@. We say that an $R$-homomorphism $f : M
\longrightarrow N$ is {\it $G$-homogeneous of degree $a$} precisely
when $f(M_g) \subseteq N_{g+a}$ for all $g \in G$\@. Such a
$G$-homogeneous homomorphism of degree $0_G$ is simply called {\em
$G$-homogeneous}. We denote by $\mbox{\rm
*}\Hom_R(M,N)_{a}$ the $R_{0_G}$-submodule of $\Hom_R(M,N)$
consisting of all $G$-homogeneous $R$-homomorphisms from $M$ to
$N$ of degree $a$\@. Then the sum $\sum_{a \in G} \mbox{\rm
*}\Hom_R(M,N)_{a}$ is direct, and we set
\[
\mbox{\rm *}\Hom_R(M,N) := \sum_{a \in G} \mbox{\rm *}\Hom_R(M,N)_{a}
= \bigoplus_{a \in G} \mbox{\rm *}\Hom_R(M,N)_{a}.
\]
This is an $R$-submodule of $\Hom_R(M,N)$, and the above direct
decomposition provides it with a structure as $G$-graded
$R$-module. It is straightforward to check that
$$\mbox{\rm *}\Hom_R (\: {\scriptscriptstyle \bullet} \:,\:
{\scriptscriptstyle \bullet} \:) : \mbox{\rm *}\mathcal{C}^G (R)
\times \mbox{\rm *}\mathcal{C}^G (R) \lra \mbox{\rm
*}\mathcal{C}^G (R)$$ is a left exact, additive functor.

\item If $M$ is finitely generated, then $\Hom_R(M,N)$ is
actually equal to $\mbox{\rm *}\Hom_R(M,N)$ with its $G$-grading
forgotten.

\item For $i \in \nn$, the functor $\mbox{\rm *}\Ext^i_R$ is the
$i$th right derived functor in $\mbox{\rm *}\mathcal{C}^G (R)$ of
$\mbox{\rm *}\Hom_R$. We make two comments here about the case where
$M$ is finitely generated. In that case $\Ext^i_R(M,N)$ is actually
equal to $\mbox{\rm *}\Ext^i_R(M,N)$ with its $G$-grading forgotten,
and, second, one can calculate the $\mbox{\rm
*}\Ext^i_R(M,N)$ by applying the functor $\mbox{\rm *}\Hom_R(M,\:
{\scriptscriptstyle \bullet} \:)$ to a (deleted) $\mbox{\rm
*}$injective resolution of $N$ in the category $\mbox{\rm
*}\mathcal{C}^G (R)$ and then taking cohomology of the resulting
complex.
\end{enumerate}
\end{rmds}

For $a \in G$, we shall denote the {\em $a$th shift functor\/} by $
(\: {\scriptscriptstyle \bullet} \:)(a) : \mbox{\rm
*}\mathcal{C}^G (R) \longrightarrow \mbox{\rm *}\mathcal{C}^G (R)
$: thus, for a $G$-graded $R$-module $M = \bigoplus _{g \in G}
M_{g}$, we have $(M(a))_{g} = M_{g+a}$ for all $g \in G$; also,
$f(a)\lceil\: _{(M(a))_{g}} = f\lceil\: _{M_{g+a}}$ for each
morphism $f$ in $\mbox{\rm *}\mathcal{C}^G (R)$ and all $g \in
G$\@.

\begin{thm} [S. Goto and K.-i. Watanabe {\cite[\S 1.3]{GotWat78}}]
\label{in.2}
Let $M$ be $G$-graded $R$-module, and denote by $\mbox{\rm *}\Spec (R)$ the set
of $G$-graded prime ideals of $R$.
We denote by $\mbox{\rm *}E(M)$ or $\mbox{\rm
*}E_R(M)$ `the' $\mbox{\rm *}$injective envelope of $M$, and
by $\mbox{\rm *}E^i(M)$ or $\mbox{\rm *}E^i_R(M)$ `the' $i$th term
in `the' minimal $\mbox{\rm *}$injective resolution of $M$ (for each
$i \geq 0$).
\begin{enumerate}
\item $\Ass_R\mbox{\rm *}E_R(M) = \Ass_RM$. \item We have
that $M$ is a $\mbox{\rm *}$indecomposable $\mbox{\rm *}$injective
$G$-graded $R$-module if and only if $M$ is isomorphic (in the
category $\mbox{\rm *}\mathcal{C}^G(R)$) to $\mbox{\rm
*}E(R/\mathfrak{q})(a)$ for some $\mathfrak{q} \in \mbox{\rm
*}\Spec (R)$ and $a \in G$. In this case, $\Ass_R M = \{\fq\}$ and
$\fq$ is uniquely determined by $M$. \item Let $(M_{\lambda})
_{\lambda \in \Lambda}$ be a non-empty family of $G$-graded
$R$-modules. Then $\bigoplus _{\lambda \in \Lambda} M_{\lambda}$
is $\mbox{\rm *}$injective if and only if $M_{\lambda}$ is
$\mbox{\rm *}$injective for all $\lambda \in \Lambda$\@. \item
Each $\mbox{\rm *}$injective $G$-graded $R$-module $M$ is a direct
sum of $\mbox{\rm *}$indecomposable $\mbox{\rm *}$injective
$G$-graded submodules, and this decomposition is uniquely
determined by $M$ up to isomorphisms. \item Let $i$ be a
non-negative integer. In view of part\/ {\rm (iv)}
above, there is a family $(\mathfrak{p}_{\alpha})_{\alpha \in
\Lambda_i}$ of $G$-graded prime ideals of $R$ and a family
$(g_{\alpha})_{\alpha \in \Lambda_i}$ of elements of $G$ for which
there is a $G$-homogeneous isomorphism
$$\mbox{\rm *}E^i(M)
\stackrel{\cong}{\longrightarrow} \bigoplus_{\alpha \in \Lambda_i}
\mbox{\rm *}E(R/\mathfrak{p}_{\alpha})(-g_{\alpha}).$$
Let $\mathfrak{p} \in \mbox{\rm *}\Spec (R)$\@. Then the
cardinality of the set $\left\{ \alpha \in \Lambda_i :
\mathfrak{p}_{\alpha} = \mathfrak{p} \right\}$ is equal to the
ordinary Bass number $\mu^i (\mathfrak{p},M)$ (that is, to
$\dim_{k(\fp)}\Ext^i_R(R_{\fp}/\fp R_{\fp},M_{\fp})$, where
$k(\fp)$ denotes the residue field of the local ring $R_{\fp}$).
\end{enumerate}
\end{thm}

A significant part of \S \ref{mu} of this paper is concerned with
the shifts `$-g_{\alpha}$' in the statement of part (v) of Theorem
\ref{in.2}. (The minus signs are inserted for notational
convenience.) In \cite{70}, the second author obtained some results
about such shifts in the special case in which $R$ is graded by the
semigroup $\nn$ of non-negative integers, and in \S \ref{mu} below,
we shall establish some multi-graded analogues.

We shall employ the following device used by Huy T\`ai H\`a
\cite[\S 2]{ha}.

\begin{defi}
\label{nt.0} Let $\phi : G \lra H$ be a homomorphism of finitely
generated torsion-free Abelian groups, and let $R =
\bigoplus_{g\in G} R_g$ be a $G$-graded commutative Noetherian
ring.

For each $h \in H$, set $R^{\phi}_h := \bigoplus_{g\in
\phi^{-1}(\{h\})} R_g$; then
$$
R^{\phi} := \bigoplus_{h\in H} R^{\phi}_h = \bigoplus_{h\in H}
\left( \bigoplus_{g\in \phi^{-1}(\{h\})} R_g\right)
$$
provides an $H$-grading on $R$, and we denote $R$ by $R^{\phi}$
when considering it as an $H$-graded ring in this way.

Furthermore, for each $G$-graded $R$-module $M = \bigoplus_{g\in G}
M_g$, set $M^{\phi}_h := \bigoplus_{g\in \phi^{-1}(\{h\})} M_g$ and
$M^{\phi} := \bigoplus_{h\in H} M^{\phi}_h$; then $M^{\phi}$ is an
$H$-graded $R^{\phi}$-module. Also, if $f : M \lra N$ is a
$G$-homogeneous homomorphism of $G$-graded $R$-modules, then the
same map $f$ becomes an $H$-homogeneous homomorphism of $H$-graded
$R^{\phi}$-modules $f^{\phi} : M^{\phi} \lra N^{\phi}$.

In this way, $(\: {\scriptscriptstyle \bullet} \:)^{\phi}$ becomes
an exact additive covariant functor from $\mbox{\rm *}\mathcal{C}^G(R)$ to
$\mbox{\rm *}\mathcal{C}^H(R)$.
\end{defi}

\begin{ntn}
\label{nt.1} We shall use $\N$ and $\nn$ to denote the sets of
positive and non-negative integers, respectively, and $r$ will
denote a fixed positive integer. Throughout the remainder of the
paper, $R := \bigoplus_{\mathbf{n}\in \Z^r} R_{\mathbf{n}}$ will
denote a commutative Noetherian ring, graded by the
additively-written finitely generated free Abelian group $\Z^r$
(with its usual addition). For $\mathbf{n} = (n_1, \ldots, n_r),~
\mathbf{m} = (m_1, \ldots, m_r) \in \Z^r$, we shall write
$$
\mathbf{n} \leq \mathbf{m} \quad \mbox{~if and only if~} \quad n_i \leq m_i
\mbox{~for all~} i = 1, \ldots,r\mbox{;}
$$
furthermore, $\mathbf{n} < \mathbf{m}$ will mean that $\mathbf{n}
\leq \mathbf{m}$ and $\mathbf{n} \neq \mathbf{m}$. The zero element
of $\Z^r$ will be denoted by $\mathbf{0}$, and, for each $i = 1,
\ldots, r$, we shall use $\mathbf{e}_i$ to denote the element of
$\Z^r$ which has $1$ in the $i$th spot and all other components
zero. Also, $\mathbf{1}$ will denote $(1, \ldots, 1) \in \Z^r$. Thus
$\mathbf{1} = \sum_{i = 1}^r \mathbf{e}_i$, and
$R_{\mathbf{e}_1}R_{\mathbf{e}_2}\ldots R_{\mathbf{e}_r} \subseteq
R_{\mathbf{1}}$.

We shall sometimes denote the $i$th component of a general member
$\mathbf{w}$ of $\Z^r$ by $w_i$ without additional explanation.

Comments made above that apply to the category $\mbox{\rm
*}\mathcal{C}^{\Z^r}(R)$ will be used without further comment. For
example, we shall say that a graded ideal of $R$ is {\em $\mbox{\rm
*}$maximal\/} if it is maximal among the set of proper $\Z^r$-graded
ideals of $R$, and that $R$ is {\em $\mbox{\rm
*}$local\/} if it has a unique $\mbox{\rm *}$maximal ideal. We shall
use $\mbox{\rm *}\Max(R)$ to denote the set of $\mbox{\rm *}$maximal
ideals of $R$.

We shall use $\mbox{\rm *}\Spec (R)$ to denote the set of
$\Z^r$-graded prime ideals of $R$; for a $\Z^r$-graded ideal $\fa$
of $R$, we shall set
$
\mbox{\rm *}\Var (\fa) := \left\{ \fp \in \mbox{\rm *}\Spec (R) : \fp \supseteq \fa
\right\}.
$
\end{ntn}

The next three lemmas are multi-graded analogues of preparatory results
in \cite[\S 1]{70}.

\begin{lem}\label{mi.0r}
Let $\mathfrak{p} \in \mbox{\rm *}\Spec (R)$ and let $a$ be an
$\Z^r$-homogeneous element of degree $\mathbf{n}$ in $R \setminus
\mathfrak{p}$. Then multiplication by $a$ provides a
$\Z^r$-homogeneous automorphism of degree $\mathbf{n}$ of $\mbox{\rm
*}E(R/{\mathfrak{p}})$\@. Also, each element of $\mbox{\rm
*}E(R/{\mathfrak{p}})$ is annihilated by some power of
$\mathfrak{p}$.

Consequently, if $S$ is a multiplicatively closed subset of
$\nn^r$-homogeneous elements of $R$ such that $S \cap \fp \neq
\emptyset$, then $S^{-1}\left(\mbox{\rm *}E(R/{\mathfrak{p}})
\right) = 0$.
\end{lem}

\begin{proof} Multiplication by $a$ provides a $\Z^r$-homogeneous
$R$-homomorphism
$$
\mu_a: \mbox{\rm *}E(R/{\mathfrak{p}}) \lra
\mbox{\rm *}E(R/{\mathfrak{p}})(\mathbf{n}).
$$
Since $\Ker \mu_a$ has zero intersection with $R/\mathfrak{p}$, it
follows that $\mu_a$ is injective. In view of Theorem
\ref{in.2}(ii), $\Ima \mu_a$ is a non-zero $\mbox{\rm *}$injective
$\Z^r$-graded submodule of the $\mbox{\rm *}$indecomposable
$\mbox{\rm
*}$injective $\Z^r$-graded $R$-module $\mbox{\rm
*}E(R/{\mathfrak{p}})(\mathbf{n})$. Hence $\mu_a$ is surjective.

The fact that each element of $\mbox{\rm *}E(R/{\mathfrak{p}})$ is
annihilated by some power of $\mathfrak{p}$ follows from Theorem
\ref{in.2}(i), which shows that $\fp$ is the only associated
prime ideal of each non-zero cyclic submodule of $\mbox{\rm
*}E(R/{\mathfrak{p}})$. The final claim is then immediate.
\end{proof}

The next two lemmas below can be proved by making obvious
modifications to the proofs of the (well-known) `ungraded' analogues.

\begin{lem}\label{ld.31}
Let $f : L \longrightarrow M$ be a $\Z^r$-homogeneous homomorphism
of $\Z^r$-graded $R$-modules such that $M$ is a $\mbox{\rm
*}$essential extension of\/ $\Ima f$. Let $S$ be a
multiplicatively closed subset of $\Z^r$-homogeneous elements of
$R$\@. Then $S^{-1}M$ is a $\mbox{\rm *}$essential extension of its
$\Z^r$-graded submodule $\Ima (S^{-1}f)$.
\end{lem}

\begin{proof} Modify the proof of \cite[11.1.5]{LC} in the obvious way.
\end{proof}

\begin{lem}\label{du.3b}
Let $S$ be a multiplicatively closed subset of $\Z^r$-homogeneous
elements of $R$, and let ${\mathfrak{p}} \in \mbox{\rm *}\Spec (R)$
be such that ${\mathfrak{p}} \cap S = \emptyset$\@. Then

\begin{enumerate}
\item the natural map $\mbox{\rm *}E_R(R/{\mathfrak{p}}) \longrightarrow S^{-1}
(\mbox{\rm *}E_R(R/{\mathfrak{p}}))$ is a $\Z^r$-homogeneous
$R$-isomor\-p\-h\-ism, so that $\mbox{\rm *}E_R(R/{\mathfrak{p}})$
has a natural structure as a $\Z^r$-graded $S^{-1}R$-module;
\item there is a $\Z^r$-homogeneous isomorphism (in $\mbox{\rm *}\mathcal{C}(S^{-1}R)$)
$$\mbox{\rm *}E_R(R/{\mathfrak{p}}) \cong
\mbox{\rm *}E_{S^{-1}R}(S^{-1}R/S^{-1}{\mathfrak{p}})\mbox{;}$$
\item $\mbox{\rm *}E_{S^{-1}R}(S^{-1}R/S^{-1}{\mathfrak{p}})$,
when considered as a $\Z^r$-graded $R$-module by means of the
natural homomorphism $R \longrightarrow S^{-1}R$, is
$\Z^r$-homogeneously isomorphic to $\mbox{\rm
*}E_R(R/{\mathfrak{p}})$;
\item for each $\mathbf{n} \in \Z^r$,
there is a $\Z^r$-homogeneous isomorphism (in $\mbox{\rm
*}\mathcal{C}(S^{-1}R)$)
$$
S^{-1}\big(\mbox{\rm *}E_R(R/{\mathfrak{p}})(\mathbf{n})\big)
\cong
\mbox{\rm *}E_{S^{-1}R}(S^{-1}R/S^{-1}{\mathfrak{p}})(\mathbf{n})\mbox{;}
$$
\item
if $I$ is a $\mbox{\rm *}$injective $\Z^r$-graded $R$-module, then
the $\Z^r$-graded $S^{-1}R$-module $S^{-1}I$ is $\mbox{\rm
*}$injective.
\end{enumerate}
\end{lem}

\begin{proof} (i) This is immediate from \ref{mi.0r}.

(ii) One can make the obvious modifications to the proof of
\cite[10.1.11]{LC} to see that, as a $\Z^r$-graded
$S^{-1}R$-module, $\mbox{\rm *}E_R(R/{\mathfrak{p}})$ is
$\mbox{\rm *}$injective; it is also $\Z^r$-homogeneously
isomorphic, as a $\Z^r$-graded  $S^{-1}R$-module,
to $S^{-1}(\mbox{\rm *}E_R(R/{\mathfrak{p}}))$. One can
use \ref{ld.31} to see that $S^{-1}(\mbox{\rm
*}E_R(R/{\mathfrak{p}}))$ is a $\mbox{\rm *}$essential extension
of $S^{-1}R/S^{-1}{\mathfrak{p}}$. The claim follows.

(iii), (iv) These are now easy.

(v) This can now be proved by making the obvious modifications
to the proof of \cite[10.1.13(ii)]{LC}.
\end{proof}

\section{\sc A multi-graded analogue of anchor point theory}
\label{mu}

\begin{defi}
\label{mu.1} We shall say that $R$ is {\em positively graded}
precisely when $R_{\mathbf{n}} = 0$ for all $\mathbf{n} \not\geq
\mathbf{0}$. When that is the case, we say that $R$ (as in
\ref{nt.1}) is {\em standard\/} precisely when $R =
R_{\mathbf{0}}[R_{\mathbf{e}_1}, \ldots,R_{\mathbf{e}_r}]$.
\end{defi}

The main results of this paper will concern the case where $R$ is
positively graded and standard.

\begin{lem}
\label{mu.2} Suppose that $R := \bigoplus_{\mathbf{n}\in \nn^r}
R_{\mathbf{n}}$ is positively graded and standard. If $\fa$ is an
$\nn^r$-graded ideal of $R$ such that $\fa \supseteq
R_{\mathbf{t}}$ for some $\mathbf{t} \in \nn^r$, then $\fa
\supseteq R_{\mathbf{n}}$ for each $\mathbf{n} \in \nn^r$ with
$\mathbf{n} \geq \mathbf{t}$.
\end{lem}

\begin{proof} Since $R$ is standard, $R_{\mathbf{n}} =
R_{\mathbf{t}}R_{\mathbf{n-t}}$, and so is contained in $\fa$.
\end{proof}

\begin{defi}
\label{mu.3} Suppose that $R := \bigoplus_{\mathbf{n}\in \nn^r}
R_{\mathbf{n}}$ is positively graded and standard. Let $\fp \in
\mbox{\rm *}\Spec (R)$. The set $\left\{j \in \{1, \ldots, r\} :
R_{\mathbf{e}_j} \subseteq \fp\right\}$ will be called {\em the set
of $\fp$-directions\/} and will be denoted by $\dir (\fp)$.

Observe that, if $i \in \dir (\fp)$, then $\fp \supseteq
R_{\mathbf{1}}$ by \ref{mu.2}. Conversely, if $\fp \supseteq
R_{\mathbf{1}}$, then, since $R_{\mathbf{1}} =
R_{\mathbf{e}_1}\ldots R_{\mathbf{e}_r}$, there exists $i \in \{1,
\ldots,r\}$ such that $R_{\mathbf{e}_i} \subseteq \fp$, and $i \in
\dir(\fp)$. Thus $\dir(\fp) \neq \emptyset$ if and only if $\fp
\supseteq R_{\mathbf{1}}$.

More generally, let $\fb$ be an $\nn^r$-graded ideal of $R$. We
define {\em the set of $\fb$-directions\/} to be $$\dir(\fb) :=
\left\{j \in \{1, \ldots, r\} : R_{\mathbf{e}_j} \subseteq
\sqrt{\fb}\right\}.$$ The members of the set $\{1,\ldots,r\}
\setminus \dir(\fb)$ are called the {\em non-$\fb$-directions\/.} It
is easy to see that $\dir(\fb) = \bigcap_{\fp \in
\Min(\fb)}\dir(\fp)$, where $\Min(\fb)$ denotes the set of minimal
prime ideals of $\fb$.
\end{defi}

\begin{rmk}
\label{mu.4a} It follows from Lemma \ref{mu.2} that, in the
situation of Definition \ref{mu.3}, each $\nn^r$-homogeneous element
of $R \setminus \fp$ has degree with $i$th component $0$ for all $i
\in \dir(\fp)$.
\end{rmk}

\begin{prop}
\label{mu.4} Suppose that $R := \bigoplus_{\mathbf{n}\in \nn^r}
R_{\mathbf{n}}$ is positively graded and standard. Let $\fp \in
\mbox{\rm *}\Var (R_{\mathbf{1}}R)$. For notational convenience,
suppose that $\dir (\fp) = \{1, \ldots, m\}$, where $0 < m \leq r$.
For each $i \in \{1, \ldots, r\} \setminus \dir (\fp) =
\{m+1,\ldots,r\}$, select $u_i \in R_{\mathbf{e}_i} \setminus \fp$.

Let $\mathbf{a} = (a_{1}, \ldots, a_m) \in \Z^{m}$. For $\mathbf{c}
= (c_{m+1}, \ldots, c_r) \in \Z^{r-m}$, we shall denote by
$\mathbf{a}| \mathbf{c}$ the element $(a_{1}, \ldots, a_m,c_{m+1},
\ldots, c_r)$ of $\Z^r$ obtained by juxtaposition.

\begin{enumerate}
\item For all choices of $\mathbf{c}, \mathbf{d} \in \Z^{r-m}$,
there is an isomorphism
of $R_{\mathbf{0}}$-modules
$$
(\mbox{\rm *} E_R(R/\p))_{\mathbf{a}|\mathbf{c}} \cong (\mbox{\rm
*} E_R(R/\p))_{\mathbf{a}|\mathbf{d}}.
$$
(Note that this does not say anything of interest if $m = r$.)

\item If $(\mbox{\rm *} E_R(R/\p))_{\mathbf{a}|\mathbf{c}} \neq 0$ for any
$\mathbf{c} \in \Z^{r-m}$, then $\mathbf{a} \leq \mathbf{0}$.

\item Let $T := R_{(\fp)}/\fp R_{(\fp)}$, where $R_{(\fp)}$ is the
$\Z^r$-homogeneous localization of $R$ at $\fp$. Then
\begin{enumerate}
\item $T$ is a simple $\Z^r$-graded ring in
the sense of\/ {\rm \cite[Definition 1.1.1]{GotWat78}};
\item $T_{\mathbf{0}}$ is a field;
\item for each $\mathbf{c} = (c_{m+1}, \ldots, c_r) \in \Z^{r-m}$,
$$
T_{\mathbf{a}|\mathbf{c}} = \begin{cases}
0 & \text{if $\mathbf{a} \neq \mathbf{0}$},\\
T_{\mathbf{0}}(\overline{u_{m+1}/1})^{c_{m+1}}\ldots(\overline{u_r/1})^{c_r}
& \text{if $\mathbf{a} = \mathbf{0}$} \end{cases}
$$
(where `$\overline{\phantom{u_1/1}}$' is used to denote natural
images of elements of $R_{(\fp)}$ in $T$); and
\item every $\Z^r$-graded $T$-module is free.
\end{enumerate}

\item We have $(0:_{\mbox{\rm
*}E_{R_{(\fp)}}(R_{(\fp)}/{\mathfrak{p}}R_{(\fp)})} \p R_{(\fp)})
= R_{(\fp)}/{\mathfrak{p}}R_{(\fp)}$.

\item If $\mathbf{a}, \mathbf{b} \in \Z^m$ and $\mathbf{c},
\mathbf{d} \in \Z^{r-m}$, and there is a $\Z^r$-homogeneous
isomorphism
$$
(\mbox{\rm *} E_R(R/\p))(\mathbf{a}|\mathbf{c}) \cong (\mbox{\rm *}
E_R(R/\p))(\mathbf{b}|\mathbf{d}),
$$
then $\mathbf{a} = \mathbf{b}$.
\end{enumerate}
\end{prop}

\begin{note} The obvious interpretation of the above statement is
to be made in the case where $m = r$.
\end{note}

\begin{proof} It will be convenient to write $\mathbf{v}$ for a general
member of $\Z^m$ and $\mathbf{w}$ for a general member of
$\Z^{r-m}$, and to use $\mathbf{v}|\mathbf{w}$ to indicate the
element of $\Z^r$ obtained by juxtaposition.

(i) By Lemma \ref{mi.0r}, for each $i = m+1, \ldots, r$,
multiplication by $u_i$ provides a $\Z^r$-homogeneous automorphism
of $\mbox{\rm *} E_R(R/\p)$ of degree $\mathbf{e}_i$; the claim
follows from this.

(ii) Set $\Delta := \left\{ \mathbf{v} \in \Z^m : v_i > 0 \mbox{~for
some~} i \in \{1, \ldots,m\}\right\}$. Since $R_{\mathbf{e}_i}
\subseteq \fp$ for all $i = 1, \ldots, m$, it follows from Lemma
\ref{mu.2} that the $\Z^r$-graded $R$-module $R/{\mathfrak{p}}$ has
$(R/{\mathfrak{p}})_{\mathbf{v}|\mathbf{w}} = 0$ for all choices of
$\mathbf{v}|\mathbf{w} \in \Z^r$ with $\mathbf{v} \in \Delta$.
Therefore the $\Z^r$-graded submodule $$
\bigoplus_{\stackrel{{\scriptstyle \mathbf{v} \in
\Delta}}{\mathbf{w} \in
\Z^{r-m}}}(R/{\mathfrak{p}})_{\mathbf{v}|\mathbf{w}} $$ of $R/\fp$
is zero. Since $\mbox{\rm *}E_R(R/{\mathfrak{p}})$ is a $\mbox{\rm
*}$essential extension of $R/{\mathfrak{p}}$, it follows that
\[
\bigoplus_{\stackrel{{\scriptstyle \mathbf{v} \in
\Delta}}{\mathbf{w} \in \Z^{r-m}}}(\mbox{\rm
*}E_R(R/{\mathfrak{p}}))_{\mathbf{v}|\mathbf{w}} = 0.
\]

(iii) By Remark \ref{mu.4a}, each $\nn^r$-homogeneous element of $R
\setminus \fp$ has degree $\mathbf{v}|\mathbf{w}$ with $\mathbf{v} =
\mathbf{0}$. Also, $(R/{\mathfrak{p}})_{\mathbf{v}|\mathbf{w}} = 0$
for all $\mathbf{v} \in \Z^{m}$ with $\mathbf{v} > \mathbf{0}$. Now
every non-zero $\Z^r$-homogeneous element of $T$ is a unit of $T$,
so that $T$ is a simple $\Z^r$-graded ring. Furthermore, the
subgroup
$$
G := \left\{ \mathbf{n} \in \Z^r : T_{\mathbf{n}} \mbox{~contains a
unit of~} T \right\}
$$
is equal to $\left\{ (n_1, \ldots, n_m, n_{m+1}, \ldots, n_r) \in
\Z^r: n_{1} = \cdots = n_m = 0 \right\}. $ The claims in parts (b),
(c) and (d) now follow from \cite[Lemma 1.1.2, Corollary 1.1.3 and
Theorem 1.1.4]{GotWat78}.

(iv) Recall that $T = R_{(\fp)}/{\mathfrak{p}}R_{(\fp)}$. Now the
$\Z^r$-graded $T$-module
 $(0:_{\mbox{\rm
*}E_{R_{(\fp)}}(R_{(\fp)}/{\mathfrak{p}}R_{(\fp)})} \p R_{(\fp)})$
contains its $\Z^r$-graded $T$-submodule
$R_{(\fp)}/{\mathfrak{p}}R_{(\fp)}$, and cannot be strictly larger,
by $\mbox{\rm *}$essentiality and the fact (see part (iii)) that
every $\Z^r$-graded $T$-module is free.

(v) By Lemma \ref{du.3b}(iv), there is a $\Z^r$-homogeneous isomorphism of
$\Z^r$-graded $R_{(\fp)}$-modules
$$
\left(\mbox{\rm
*}E_{R_{(\fp)}}(R_{(\fp)}/{\mathfrak{p}}R_{(\fp)})\right)(\mathbf{a}|\mathbf{c})
\cong \left(\mbox{\rm
*}E_{R_{(\fp)}}(R_{(\fp)}/{\mathfrak{p}}R_{(\fp)})\right)(\mathbf{b}|\mathbf{d}).
$$
Abbreviate $\mbox{\rm
*}E_{R_{(\fp)}}(R_{(\fp)}/{\mathfrak{p}}R_{(\fp)})$ by $F$. It
follows from part (iv) that
\begin{align*}
T(\mathbf{a}|\mathbf{c}) & = (0:_F \fp
R_{(\fp)})(\mathbf{a}|\mathbf{c}) =
(0:_{F(\mathbf{a}|\mathbf{c})} \fp R_{(\fp)})\\
& \cong (0:_{F(\mathbf{b}|\mathbf{d})} \fp R_{(\fp)}) = (0:_F \fp
R_{(\fp)})(\mathbf{b}|\mathbf{d})\\ & = T(\mathbf{b}|\mathbf{d}),
\end{align*}
where the isomorphism is $\Z^r$-homogeneous.
But, for $\mathbf{n} = (n_1, \ldots, n_m,n_{m+1}, \ldots, n_r) \in \Z^r$, we have
$$
T(\mathbf{a}|\mathbf{c})_{\mathbf{n}} \neq 0 \quad \mbox{~if and
only if~} \quad (n_{1}, \ldots, n_m) = -\mathbf{a}
$$
(by part (iii)). Therefore $\mathbf{a} = \mathbf{b}$.
\end{proof}

\begin{rmk}
\label{mu.5} Suppose that $R := \bigoplus_{\mathbf{n}\in \nn^r}
R_{\mathbf{n}}$ is positively graded and standard, and let $\fb$ be
an $\nn^r$-graded ideal of $R$ for which $\dir(\fb) \neq \emptyset$.

Write $\dir (\fb) = \{i_1, \ldots, i_{m}\}$, where $0 < m \leq r$
and $i_1 < \cdots < i_{m}$. Let $\phi(\fb) : \Z^r \lra \Z^{m}$ be
the epimorphism of Abelian groups defined by
$$
\phi(\fb)((n_1, \ldots, n_r)) = (n_{i_1}, \ldots, n_{i_{m}}) \quad
\mbox{~for all~} (n_1, \ldots, n_r) \in \Z^r.
$$
We can think of $\phi(\fb) : \Z^r \lra \Z^{m}$ as the homomorphism
which `forgets the co-ordinates in the non-$\fb$-directions'.

Now let $\fp \in \mbox{\rm *}\Var (R_{\mathbf{1}}R)$. The above
defines an Abelian group homomorphism $\phi(\fp) : \Z^r \lra
\Z^{\#\dir(\fp)}$. (For a finite set $Y$, the notation $\#Y$ denotes
the cardinality of the set $Y$.) In the case where $\fb \subseteq
\fp$, we have $\dir(\fb) \subseteq \dir(\fp)$, and we define the
Abelian group homomorphism $\phi(\fp;\fb) : \Z^{\#\dir(\fp)} \lra
\Z^{\#\dir(\fb)}$ to be the unique $\Z$-homomorphism such that
$\phi(\fp;\fb)\circ\phi(\fp) = \phi(\fb)$.

Now let $\fp \in \mbox{\rm *}\Var (R_{\mathbf{1}}R)$ and
$\#\dir(\fp) = m$; we use the notation of\/ {\rm \ref{nt.0}}. Let $T
:= R_{(\fp)}/\fp R_{(\fp)}$, and let $L$ be a $\Z^r$-graded
$T$-module.

\begin{enumerate}
\item By Proposition\/ {\rm \ref{mu.4}(iii)}, for each $\mathbf{a} \in \Z^m$ and
each $\mathbf{n} \in \Z^r$,
$$
(T(-\mathbf{n})^{\phi(\fp)})_{\mathbf{a}} = \begin{cases}
0 & \text{if $\phi(\fp)(\mathbf{n}) \neq \mathbf{a}$},\\
(T^{\phi(\fp)})_{\mathbf{0}} & \text{if $\phi(\fp)(\mathbf{n}) =
\mathbf{a}$}. \end{cases}
$$
In particular, the $\Z^{m}$-graded ring $T^{\phi(\fp)}$ is
concentrated in degree $\mathbf{0} \in \Z^{m}$.

\item Each component of the $\Z^{m}$-graded
$T^{\phi(\fp)}$-module $L^{\phi(\fp)}$ is a free
$(T^{\phi(\fp)})_{\mathbf{0}}$-submodule of $L^{\phi(\fp)}$.

\item If $L$ is finitely generated, then
$$
\rank_{T^{\phi(\fp)}}L^{\phi(\fp)} = \sum_{\mathbf{a} \in
\Z^{m}}\rank_{(T^{\phi(\fp)})_{\mathbf{0}}}
\big(L^{\phi(\fp)}\big)_{\mathbf{a}}\mbox{;}
$$
since the left-hand side of the above equation is finite,
all except finitely many of the terms on the right-hand side are zero.
\end{enumerate}
\end{rmk}

\begin{thm}\label{mi.1} Suppose that
$R := \bigoplus_{\mathbf{n}\in \nn^r} R_{\mathbf{n}}$
is positively graded and standard.
Let $M$ be a $\Z^r$-graded $R$-module, and let
\[
I^{\scriptscriptstyle \bullet} :
0
\longrightarrow \mbox{\rm *}E^0(M)
\stackrel{d^0}{\longrightarrow} \mbox{\rm *}E^1(M)
\longrightarrow \cdots \longrightarrow
\mbox{\rm *}E^i(M )\stackrel{d^i}{\longrightarrow} \mbox{\rm *}E^{i+1}(M)
\longrightarrow \cdots
\]
be the minimal $\mbox{\rm *}$injective resolution of $M$.
For each $i \in \nn$, let
$$
\theta_i : \mbox{\rm *}E^{i}(M) \stackrel{\cong}{\longrightarrow}
\bigoplus_{\alpha \in \Lambda_i}
\mbox{\rm *}E(R/{\mathfrak{p}}_{\alpha})(- \mathbf{n}_{\alpha})
$$
be a $\Z^r$-homogeneous isomorphism, where ${\mathfrak{p}}_{\alpha} \in
\mbox{\rm *}\Spec (R)$ and $\mathbf{n}_{\alpha} \in \Z^r$ for all
$\alpha \in \Lambda_i$.

Let $\fp \in \mbox{\rm *}\Var (R_{\mathbf{1}}R)$ and use the
notation $\phi(\fp) : \Z^r \lra \Z^{m}$ and $T := R_{(\fp)}/\fp
R_{(\fp)}$ of Remark\/ {\rm \ref{mu.5}}, where $m$ is the number of
$\fp$-directions.

Let $i \in \nn$ and let $\mathbf{a} \in \Z^{m}$.
Then the cardinality of the set $\left\{ \alpha \in \Lambda_i :
{\mathfrak{p}}_{\alpha} = {\mathfrak{p}} \mbox{~and~}
\phi(\fp)(\mathbf{n}_{\alpha}) = \mathbf{a} \right\}$ is equal to
$$
\rank_{(T^{\phi(\fp)})_{\mathbf{0}}} \left(\big(\big(\mbox{\rm
*}\Ext^i_{R_{({\mathfrak{p}})}} (R_{({\mathfrak{p}})}/\p
R_{({\mathfrak{p}})},M_{({\mathfrak{p}})})\big)
^{\phi(\fp)}\big)_{\mathbf{a}}\right).$$
\end{thm}

\begin{proof} By Lemmas \ref{mi.0r}, \ref{ld.31} and \ref{du.3b},
there are $\Z^r$-homogeneous isomorphisms of graded
$R_{(\fp)}$-modules
$$ \mbox{\rm *}E^{i}_{R_{(\fp)}}(M_{(\fp)}) \cong
\left(\mbox{\rm *}E^{i}_{R}(M)\right)_{(\fp)} \cong
\bigoplus_{\stackrel{\scriptstyle \alpha \in
\Lambda_i}{\fp_{\alpha} \subseteq \fp}}
\mbox{\rm *}E(R_{(\fp)}/{\mathfrak{p}}_{\alpha}R_{(\fp)})(-\mathbf{n}_{\alpha}).
$$

One can calculate $\mbox{\rm *}\Ext^i_{R_{({\mathfrak{p}})}}
(R_{({\mathfrak{p}})}/\p R_{({\mathfrak{p}})},M_{({\mathfrak{p}})})$
(up to isomorphism in the category $\mbox{\rm
*}\mathcal{C}^{\Z^r}(R_{(\fp)})$) by taking the $i$th cohomology
module of the complex $(0:_{(I^{\scriptscriptstyle
\bullet})_{(\fp)}} \fp R_{(\fp)}).$ Note that, by Lemma \ref{ld.31},
for each $j \in \nn$, the inclusion $\Ker (d^j_{(\fp)}) \subseteq
\mbox{\rm *}E^j(M)_{(\fp)}$ is $\mbox{\rm *}$essential, so that the
inclusion
$$\Ker (d^j_{(\fp)})\cap\left(0:_{\mbox{\rm *}E^j(M)_{(\fp)}} \fp
R_{(\fp)}\right) \subseteq \left(0:_{\mbox{\rm *}E^j(M)_{(\fp)}} \fp
R_{(\fp)}\right)$$ is also $\mbox{\rm *}$essential. Because, by
Proposition \ref{mu.4}(iii)(d), each $\Z^r$-graded $T$-module is
free, it follows that all the `differentiation' maps in the complex
$(0:_{(I^{\scriptscriptstyle \bullet})_{(\fp)}} \fp R_{(\fp)})$ are
zero. Hence
$$
\mbox{\rm *}\Ext^i_{R_{({\mathfrak{p}})}} (R_{({\mathfrak{p}})}/\p
R_{({\mathfrak{p}})},M_{({\mathfrak{p}})}) \cong
\bigoplus_{\stackrel{\scriptstyle \alpha \in
\Lambda_i}{\fp_{\alpha} \subseteq \fp}} \left(0: _{\mbox{\rm
*}E(R_{(\fp)}/{\mathfrak{p}}_{\alpha}R_{(\fp)})(-\mathbf{n}_{\alpha})}
\p R_{({\mathfrak{p}})} \right) \quad \mbox{~in~} \mbox{\rm
*}\mathcal{C}^{\Z^r}(R_{(\fp)}).
$$

For $\alpha \in \Lambda_i$ such that $\p_{\alpha}\subset \fp$ (the
symbol `$\subset$' is reserved to denote strict inclusion), there
exists an $\nn^r$-homogeneous element $u \in \fp \setminus
\p_{\alpha}$, and the fact (see Lemma \ref{mi.0r}) that
multiplication by $u/1 \in R_{(\fp)}$ provides an automorphism of
$\mbox{\rm
*}E(R_{(\fp)}/{\mathfrak{p}}_{\alpha}R_{(\fp)})$ ensures that
\[
\big(0:_{\mbox{\rm *}E(R_{(\fp)}/{\mathfrak{p}}_{\alpha}R_{(\fp)})
(-\mathbf{n}_{\alpha})} \fp R_{(\fp)}\big)
= 0.
\]
If $\alpha \in \Lambda_i$ is such that $\p_{\alpha} = \fp$, then, by
Proposition \ref{mu.4}(iv),
\[
(0:_{\mbox{\rm
*}E_{R_{(\fp)}}(R_{(\fp)}/{\mathfrak{p}}R_{(\fp)})(-\mathbf{n}_{\alpha})}
\p R_{(\fp)})
= \left(R_{(\fp)}/{\mathfrak{p}}R_{(\fp)}\right)(-\mathbf{n}_{\alpha})
\]
and, by Proposition \ref{mu.4}(iii)(d), this is a free $\Z^r$-graded
$T$-module.

Therefore there is a $\Z^r$-homogeneous isomorphism of $\Z^r$-graded
$T$-modules
$$
\mbox{\rm *}\Ext^i_{R_{({\mathfrak{p}})}}
(R_{({\mathfrak{p}})}/\p R_{({\mathfrak{p}})},M_{({\mathfrak{p}})}) \cong
\bigoplus_{\stackrel{\scriptstyle \alpha \in
\Lambda_i}{\fp_{\alpha} = \fp}}
\left(R_{(\fp)}/{\mathfrak{p}}R_{(\fp)}\right)(-\mathbf{n}_{\alpha}).
$$
Now apply the functor $(\: {\scriptscriptstyle \bullet}
\:)^{\phi(\fp)}$ to obtain a $\Z^m$-homogeneous isomorphism of
$\Z^m$-graded $T^{\phi(\fp)}$-modules
$$
\big(\mbox{\rm *}\Ext^i_{R_{({\mathfrak{p}})}}
(R_{({\mathfrak{p}})}/\p
R_{({\mathfrak{p}})},M_{({\mathfrak{p}})})\big) ^{\phi(\fp)} \cong
\bigoplus_{\stackrel{\scriptstyle \alpha \in \Lambda_i}{\fp_{\alpha}
= \fp}}
\left(\left(R_{(\fp)}/{\mathfrak{p}}R_{(\fp)}\right)(-\mathbf{n}_{\alpha})
\right)^{\phi(\fp)}.
$$
But, by Remark \ref{mu.5}(i), for an $\alpha \in \Lambda_i$,
$$
\big(\big(T(-\mathbf{n}_{\alpha})\big)^{\phi(\fp)}\big)_{\mathbf{a}}
=
\begin{cases} 0 & \text{if $\phi(\fp)(\mathbf{n}_{\alpha}) \neq \mathbf{a}$},\\
\big(T^{\phi(\fp)}\big)_{\mathbf{0}} & \text{if
$\phi(\fp)(\mathbf{n}_{\alpha}) = \mathbf{a}$}.
\end{cases}
$$
The desired result now follows from Remark \ref{mu.5}(iii)
\end{proof}

\begin{defs}\label{mi.2}
Let the situation and notation be as in Theorem \ref{mi.1}, so that,
in particular, $\fp \in \mbox{\rm *}\Var(R_{\mathbf{1}}R)$ and $m$
denotes the number of $\fp$-directions.

Let $i \in \nn$. We say that $\mathbf{a} \in \Z^m$ is an {\em $i$th
level anchor point of $\fp$ for $M$\/} if
$$\big(\big(\mbox{\rm *}\Ext^i_{R_{({\mathfrak{p}})}}
(R_{({\mathfrak{p}})}/\p
R_{({\mathfrak{p}})},M_{({\mathfrak{p}})})\big)
^{\phi(\fp)}\big)_{\mathbf{a}} \neq 0\mbox{;}$$ the set of all $i$th
level anchor points of $\fp$ for $M$ is denoted by $\anch^i(\fp,M)$;
also, we write
$$
\anch(\fp,M) = \bigcup_{j \in \nn}\anch^j(\fp,M),
$$ and refer to this as the set of {\em anchor points of $\fp$ for $M$}.

Thus $\anch^i(\fp,M)$ is the set of $m$-tuples
$\mathbf{a} \in \Z^m$ for which, when we decompose
$$\mbox{\rm *}E^{i}(M) \stackrel{\cong}{\longrightarrow}
\bigoplus_{\alpha \in \Lambda_i} \mbox{\rm
*}E(R/{\mathfrak{p}}_{\alpha})(-\mathbf{n}_{\alpha})$$ by means of
a $\Z^r$-homogeneous isomorphism, there exists $\alpha \in
\Lambda_i$ with $\fp_{\alpha} = \fp$ and
$\phi(\fp)(\mathbf{n}_{\alpha}) = \mathbf{a}$. Note that
$\anch^i(\fp,M) = \emptyset$ if $\mu^i(\fp,M) = 0$, and that, if $M$
is finitely generated, then $\anch^i(\fp,M)$ is a finite set, by
Remark \ref{mu.5}(iii).
\end{defs}

The details in our present multi-graded situation are more
complicated (and therefore more interesting!)\ than in the
singly-graded situation studied in \cite{70} because there might
exist a $\fp \in \mbox{\rm *}\Var(R_{\mathbf{1}}R)$ for which the
set of $\fp$-directions is a proper subset of $\{1, \ldots,r\}$.
This cannot happen when $r = 1$. It is worthwhile for us to draw
attention to the simplifications that occur in the above theory when
$\dir(\fp) = \{1, \ldots,r\}$, for that case provides a more-or-less
exact analogue of the anchor point theory for the singly-graded case
developed in \cite{70}.

\begin{ex}
\label{mu.6}
Suppose that
$R := \bigoplus_{\mathbf{n}\in \nn^r} R_{\mathbf{n}}$
is positively graded and standard.
Let $M$ be a $\Z^r$-graded $R$-module, and let
\[
I^{\scriptscriptstyle \bullet} :
0
\longrightarrow \mbox{\rm *}E^0(M)
\stackrel{d^0}{\longrightarrow} \mbox{\rm *}E^1(M)
\longrightarrow \cdots \longrightarrow
\mbox{\rm *}E^i(M )\stackrel{d^i}{\longrightarrow} \mbox{\rm *}E^{i+1}(M)
\longrightarrow \cdots
\]
be the minimal $\mbox{\rm *}$injective resolution of $M$.
For each $i \in \nn$, let
$$
\theta_i : \mbox{\rm *}E^{i}(M) \stackrel{\cong}{\longrightarrow}
\bigoplus_{\alpha \in \Lambda_i}
\mbox{\rm *}E(R/{\mathfrak{p}}_{\alpha})(- \mathbf{n}_{\alpha})
$$
be a $\Z^r$-homogeneous isomorphism, where ${\mathfrak{p}}_{\alpha} \in
\mbox{\rm *}\Spec (R)$ and $\mathbf{n}_{\alpha} \in \Z^r$ for all
$\alpha \in \Lambda_i$.

Let $\fp \in \mbox{\rm *}\Spec(R)$ be such that $\fp \supseteq
R_{\mathbf{n}}$ for all $\mathbf{n} > \mathbf{0}$, so that
$\dir(\fp) = \{1, \ldots,r\}$. In this case, $T:= R_{(\fp)}/\fp
R_{(\fp)}$ is concentrated in degree $\mathbf{0}$, and
$T_{\mathbf{0}}$ is a field isomorphic to
$k_{R_{\mathbf{0}}}(\fp_{\mathbf{0}})$.

Let $i \in \nn$. Then $\anch^i(\fp,M)$ is the
set of $r$-tuples $\mathbf{a} \in \Z^{r}$ for which
there exists $\alpha \in \Lambda_i$
with $\fp_{\alpha} = \fp$ and $\mathbf{n}_{\alpha} = \mathbf{a}$. The
cardinality of the set of such $\alpha$s is
$$\dim_{k_{R_{\mathbf{0}}}(\fp_{\mathbf{0}})}
\left(\big(\mbox{\rm *}\Ext^i_{R_{({\mathfrak{p}})}}
(R_{({\mathfrak{p}})}/\p
R_{({\mathfrak{p}})},M_{({\mathfrak{p}})})
\big)_{\mathbf{a}}\right),
$$
and we have
$$
\sum_{\mathbf{a} \in \Z^r}
\dim_{k_{R_{\mathbf{0}}}(\fp_{\mathbf{0}})}
\left(\big(\mbox{\rm *}\Ext^i_{R_{({\mathfrak{p}})}}
(R_{({\mathfrak{p}})}/\p R
_{({\mathfrak{p}})},M_{({\mathfrak{p}})})
\big)_{\mathbf{a}}\right) = \mu^i(\fp,M).
$$
In particular, if $M$ is finitely generated, then there are only
finitely many $i$th level anchor points of $\fp$ for $M$.

This reflects rather well the singly-graded anchor point theory
studied in \cite{70}.
\end{ex}

Our next aim is to extend (in a sense) the final result in Example
\ref{mu.6} (namely that, when $M$ (as in the example) is a finitely
generated $\Z^r$-graded $R$-module and $\fp \in \mbox{\rm
*}\Spec(R)$ is such that $\fp \supseteq R_{\mathbf{n}}$ for all
$\mathbf{n} > \mathbf{0}$, then, for each $i \in \nn$, there are
only finitely many $i$th level anchor points of $\fp$ for $M$) to
{\em all\/} $\nn^r$-graded primes of $R$ that contain
$R_{\mathbf{1}}$.

\begin{rmk}
\label{mu.7} Let $S$ be a multiplicatively closed set of $\Z^r$-homogeneous
elements of $R$, and let $M$, $N$ be $\Z^r$-graded $R$-modules with $M$
finitely generated.  Then, for each $i \in \nn$, there is a $\Z^r$-homogeneous
$S^{-1}R$-isomorphism
$$
S^{-1}\left(\mbox{\rm *}\Ext^i_R(M,N)\right) \cong \mbox{\rm
*}\Ext^i_{S^{-1}R}(S^{-1}M,S^{-1}N).
$$
\end{rmk}

\begin{thm}\label{mu.10}
Assume that $R = \bigoplus_{\mathbf{n} \in \nn^r} R_\mathbf{n}$ is
positively graded and standard, and let $M$ be a $\Z^r$-graded
$R$-module. Let $i \in \nn$, and let $\fp\in \mbox{\rm *}\Var
(R_{\mathbf{1}}R)$. Then
$$
\anch^i(\fp,M) = \anch^i(\fp^{\phi(\fp)},M^{\phi(\fp)}),
$$
and so is finite if $M$ is finitely generated.
\end{thm}

\begin{proof} Suppose, for ease of notation, that $\dir (\fp) = \{1,\ldots,m\}$,
where $0 < m \leq r$. Note
that $\fp^{\phi(\fp)}$ is a $\Z^m$-graded prime ideal of the
$\Z^m$-graded ring $R^{\phi(\fp)}$, and that $\dir(\fp^{\phi(\fp)}) =
\{1, \ldots, m\}$ (by Lemma \ref{mu.2}).

Set $E := \mbox{\rm *}\Ext^i_{R}(R/\fp,M)$. Let $\mathbf{a} \in
\Z^m$. In view of \ref{mu.7}, the $m$-tuple $\mathbf{a}$ is an $i$th
level anchor point of $\fp$ for $M$ if and only if
$((E_{(\fp)})^{\phi(\fp)})_{\mathbf{a}} \neq 0$. Our initial task in
this proof is to show that this is the case if and only if
$$
\big((\mbox{\rm *}\Ext^i_{R^{\phi(\fp)}}(R^{\phi(\fp)}/\fp^{\phi(\fp)},
M^{\phi(\fp)}))_{(\fp^{\phi(\fp)})}\big)_{\mathbf{a}} \neq 0.
$$

Now the $\Z^r$-graded $R$-module $E$ can be constructed by
application of the functor $\mbox{\rm *}\Hom_R(\:
{\scriptscriptstyle \bullet} \:,M)$ to a (deleted) $\mbox{\rm
*}$free resolution of $R/\fp$ by finitely generated $\mbox{\rm
*}$free $\Z^r$-graded modules
in the category $\mbox{\rm
*}\mathcal{C}^{\Z^r} (R)$ and then taking cohomology of the resulting
complex. It follows that there is a $\Z^m$-homogeneous isomorphism of $\Z^m$-graded
$R^{\phi(\fp)}$-modules
$$
E^{\phi(\fp)} \cong \mbox{\rm *}\Ext^i_{R^{\phi(\fp)}}(R^{\phi(\fp)}/
\fp^{\phi(\fp)}, M^{\phi(\fp)}).
$$

Suppose that $((E_{(\fp)})^{\phi(\fp)})_{\mathbf{a}} \neq 0$. Thus
there exists $\mathbf{n} \in \Z^r$ such that $\phi(\fp)(\mathbf{n})
= \mathbf{a}$ and $\xi \in (E_{(\fp)})_{\mathbf{n}}$ such that $\xi
\neq 0$. By Remark \ref{mu.4a}, there exists $\mathbf{n}' \in \Z^r$
such that $\phi(\fp)(\mathbf{n}') = \mathbf{a}$ and $e \in
E_{\mathbf{n}'}$ which is not annihilated by any $\Z^r$-homogeneous
element of $R \setminus \fp$. Now any $\Z^m$-homogeneous element of
$R^{\phi(\fp)} \setminus \fp^{\phi(\fp)}$ will, when written as a
sum of $\Z^r$-homogeneous elements of $R$, have at least one
component outside $\fp$, and so $0 \neq e/1 \in
(E^{\phi(\fp)})_{(\fp^{\phi(\fp)})}$. Hence
$\big((E^{\phi(\fp)})_{(\fp^{\phi(\fp)})}\big)_{\mathbf{a}} \neq 0$,
so that
$$
\big((\mbox{\rm *}\Ext^i_{R^{\phi(\fp)}}(R^{\phi(\fp)}/\fp^{\phi(\fp)},
M^{\phi(\fp)}))_{(\fp^{\phi(\fp)})}\big)_{\mathbf{a}} \neq 0.
$$

Now suppose that $ \big((\mbox{\rm
*}\Ext^i_{R^{\phi(\fp)}}(R^{\phi(\fp)}/\fp^{\phi(\fp)},
M^{\phi(\fp)}))_{(\fp^{\phi(\fp)})}\big)_{\mathbf{a}} \neq 0.
$
Then $\big((E^{\phi(\fp)})_{(\fp^{\phi(\fp)})}\big)_{\mathbf{a}} \neq 0$. Since
every $\Z^m$-homogeneous element of $R^{\phi(\fp)} \setminus \fp^{\phi(\fp)}$
has degree $\mathbf{0} \in \Z^m$, it follows that there exists
$e \in (E^{\phi(\fp)})_{\mathbf{a}}$
that is not annihilated by any
$\Z^m$-homogeneous element of $R^{\phi(\fp)} \setminus \fp^{\phi(\fp)}$.
In particular, $e$ is not annihilated by any
$\Z^r$-homogeneous element of $R \setminus \fp$.
Therefore $0 \neq e/1 \in ((E_{(\fp)})^{\phi(\fp)})_{\mathbf{a}}$.

This proves that $\anch^i(\fp,M) = \anch^i(\fp^{\phi(\fp)},M^{\phi(\fp)})$.
Finally, since $\dir(\fp^{\phi(\fp)}) =
\{1, \ldots, m\}$, it follows from Example \ref{mu.6} that
$\anch^i(\fp^{\phi(\fp)},M^{\phi(\fp)})$ is finite when $M$ is finitely generated.
\end{proof}

The aim of the remainder of this section is to establish a multi-graded
analogue of a result of Bass \cite[Lemma 3.1]{Bass63}. However, there are
some subtleties which mean that our generalization of \cite[Lemma 1.8]{70}
is not completely straightforward.

\begin{thm}\label{mi.3}
Assume that $R = \bigoplus_{\mathbf{n} \in \nn^r} R_\mathbf{n}$ is
positively graded and standard, and let $M$ be a finitely generated
$\Z^r$-graded $R$-module. Let $\fp, \fq \in \mbox{\rm *}\Spec (R)$
be such that $R_{\mathbf{1}}R \subseteq \fp \subset \fq$ (we reserve
the symbol `$\subset$' to denote strict inclusion) and that there is
no $\Z^r$-graded prime ideal strictly between $\fp$ and $\fq$. Note
that $\dir(\fp) \subseteq \dir(\fq)$: suppose, for ease of notation,
that $\dir(\fp) = \{1,\ldots,m\}$ and $\dir(\fq) = \{1,\ldots,m,m+1,
\ldots, h\}$, where $0 < m \leq h \leq r$.

Let $i \in \nn$. Then, for each $\mathbf{a} = (a_1, \ldots, a_m)
\in \anch^i(\fp,M)$, there exists $$\mathbf{b} = (b_1, \ldots,
b_m,b_{m+1}, \ldots, b_h) \in \anch^{i+1}(\fq,M)$$ such that
$
(b_1, \ldots, b_m) = (a_1, \ldots, a_m) = \mathbf{a}.
$
\end{thm}

\begin{proof} There exists an $\nn^r$-homogeneous element $b \in \fq \setminus \fp$.
By Remark \ref{mu.4a}, each $\nn^r$-homogeneous element of $R
\setminus \fp$ has degree with first $m$ components $0$. In
particular, $\deg (b) = \mathbf{0}|\mathbf{v} \in \Z^m \times
\Z^{r-m}$ for some $\mathbf{v} \in \Z^{r-m}$.

Since $\mathbf{a} \in \anch^i(\fp,M)$, there exists $\mathbf{w} \in
\Z^{r-m}$ such that
$
\big(\mbox{\rm *}\Ext^i_{R_{({\mathfrak{p}})}}
(R_{({\mathfrak{p}})}/\p
R_{({\mathfrak{p}})},M_{({\mathfrak{p}})})\big)_ {\mathbf{a}|
\mathbf{w}} \neq 0.
$
Set $E:= \mbox{\rm *}\Ext^i_{R}(R/\fp,M)$. In view of Remark
\ref{mu.7}, we must have $(E_{(\fp)})_{\mathbf{a}| \mathbf{w}} \neq
0$. Since each $\nn^r$-homogeneous element of $R \setminus \fp$ has
degree with first $m$ components $0$, this means that there exists a
homogeneous element $e \in E$, with $\deg(e) = \mathbf{a}|
\mathbf{w}'$ for some $\mathbf{w}' \in \Z^{r-m}$, that is not
annihilated by any $\nn^r$-homogeneous element of $R \setminus \fp$.
But $R \setminus \fq \subseteq R \setminus \fp$, and so it follows
that $(E_{(\fq)})_{\mathbf{a}| \mathbf{w}'} \neq 0$. By Remark
\ref{mu.7} again,
$
\big(\mbox{\rm *}\Ext^i_{R_{({\mathfrak{q}})}}
(R_{({\mathfrak{q}})}/\p
R_{({\mathfrak{q}})},M_{({\mathfrak{q}})})\big)_ {\mathbf{a}|
\mathbf{w}'} \neq 0.
$
Write $F := \mbox{\rm *}\Ext^i_{R_{({\mathfrak{q}})}}
(R_{({\mathfrak{q}})}/\p R_{({\mathfrak{q}})},M_{({\mathfrak{q}})})$.

There is an exact sequence
$$
0 \lra \left(R_{(\fq)}/\fp
R_{(\fq)}\right)(-(\mathbf{0}|\mathbf{v})) \stackrel{b/1}{\lra}
R_{(\fq)}/\fp R_{(\fq)} \lra R_{(\fq)}/\left(\fp
R_{(\fq)}+(b/1)R_{(\fq)}\right) \lra 0
$$
in $\mbox{\rm *}\mathcal{C}^{\Z^r} (R_{(\fq)})$, and this induces an exact sequence
$$
F \stackrel{b/1}{\lra} F(\mathbf{0}|\mathbf{v}) \lra \mbox{\rm
*}\Ext^{i+1}_{R_{({\mathfrak{q}})}} (R_{({\mathfrak{q}})}/(\p
R_{({\mathfrak{q}})} + (b/1)R_{({\mathfrak{q}})} ),
M_{({\mathfrak{q}})}).
$$
Recall that $\deg (b) = \mathbf{0}|\mathbf{v}$. We claim that there
exists $\mathbf{y} \in \Z^{r-m}$ such that $F_{\mathbf{a}|
\mathbf{y}} \neq (b/1)F_{\mathbf{a}| \mathbf{y}-\mathbf{v}}$. To see
this, note that $b/1 \in \fq R_{({\mathfrak{q}})}$, the unique
$\mbox{\rm *}$maximal ideal of the homogeneous localization
$R_{({\mathfrak{q}})}$, and if we had $F_{\mathbf{a}| \mathbf{y}} =
(b/1)F_{\mathbf{a}| \mathbf{y}-\mathbf{v}}$ for every $\mathbf{y}
\in \Z^{r-m}$, then we should have $F_{\mathbf{a}| \mathbf{w}'}
\subseteq \bigcap_{n \in \N} (b/1)^nF$, which is zero by the
multi-graded version of Krull's Intersection Theorem. (One can show
that $G := \bigcap_{n \in \N} (b/1)^nF$ satisfies $G = (b/1)G$, and
then use the multi-graded version of Nakayama's Lemma.) Thus there
exists $\mathbf{y} \in \Z^{r-m}$ such that $F_{\mathbf{a}|
\mathbf{y}} \neq (b/1)F_{\mathbf{a}| \mathbf{y}-\mathbf{v}}$, and
therefore, in view of the last exact sequence,
$$
\left(\mbox{\rm *}\Ext^{i+1}_{R_{({\mathfrak{q}})}}
(R_{({\mathfrak{q}})}/(\p R_{({\mathfrak{q}})} +
(b/1)R_{({\mathfrak{q}})} ),
M_{({\mathfrak{q}})})\right)_{\mathbf{a}| \mathbf{y}} \neq 0.
$$

Now $R_{({\mathfrak{q}})}/\left(\p R_{({\mathfrak{q}})} +
(b/1)R_{({\mathfrak{q}})} \right)$ is concentrated in
$\Z^r$-degrees whose first $m$ components are all zero. Therefore
all its $\Z^r$-graded $R$-homomorphic images and all its
$\Z^r$-graded submodules are also concentrated in $\Z^r$-degrees
whose first $m$ components are all zero.

The only $\Z^r$-graded prime ideal of $R_{({\mathfrak{q}})}$ that
contains the ideal $\p R_{({\mathfrak{q}})} +
(b/1)R_{({\mathfrak{q}})}$ is $\fq R_{({\mathfrak{q}})}$, and so $\p
R_{({\mathfrak{q}})} + (b/1)R_{({\mathfrak{q}})}$ is $\fq
R_{({\mathfrak{q}})}$-primary. It follows that there is a chain of
$\Z^r$-graded ideals of $R_{({\mathfrak{q}})}$ from $\fq
R_{({\mathfrak{q}})}$ to $\p R_{({\mathfrak{q}})} +
(b/1)R_{({\mathfrak{q}})}$ with the property that each subquotient
is $R_{({\mathfrak{q}})}$-isomorphic to $\left(R_{({\mathfrak{q}})}/
\fq R_{({\mathfrak{q}})}\right)(\mathbf{0}|\mathbf{z})$ for some
$\mathbf{z} \in \Z^{r-m}$. It therefore follows from the
half-exactness of $\mbox{\rm *}\Ext^{i+1}_{R_{({\mathfrak{q}})}}$
that there exists $\mathbf{y}' \in \Z^{r-m}$ such that
$$
\big(\mbox{\rm *}\Ext^{i+1}_{R_{({\mathfrak{q}})}}
(R_{({\mathfrak{q}})}/\q
R_{({\mathfrak{q}})},M_{({\mathfrak{q}})})\big)_ {\mathbf{a}|
\mathbf{y}'} \neq 0.
$$
The claim then follows from Theorem \ref{mi.1}.
\end{proof}

\begin{cor}
\label{mu.8} Assume that $R = \bigoplus_{\mathbf{n} \in \nn^r}
R_\mathbf{n}$ is positively graded and standard, and let $M$ be a
finitely generated $\Z^r$-graded $R$-module. Let $\fp \in \mbox{\rm
*}\Var (R_{\mathbf{1}}R)$, and suppose, for ease of notation, that
$\dir(\fp) = \{1,\ldots,m\}$.

Let $\mathbf{a} \in \anch(\fp, M)$. Then there exists $\fq \in \mbox{\rm *}\Spec (R)$
such that $\fq \supseteq R_{\mathbf{n}}$ for all $\mathbf{n} \in \nn^r$ with
$\mathbf{n} > \mathbf{0}$ and $\mathbf{b} = (b_1, \ldots, b_m,b_{m+1},\ldots,b_r)
\in \anch(\fq, M)$ such that
$
\mathbf{a} = (b_1, \ldots, b_m).
$
\end{cor}

\begin{proof} There exists a saturated chain $\fp = \fp_0 \subset \fp_1 \subset
\cdots \subset \fp_t = \fq$ of $\Z^r$-graded prime ideals of $R$ such that
$\fq$ is $\mbox{\rm *}$maximal. Since $\fq$ is contained in the
$\Z^r$-graded prime ideal
$$
(\fq \cap R_{\mathbf{0}}) \bigoplus \bigoplus_{\mathbf{n} > \mathbf{0}}
R_{\mathbf{n}},
$$
these two $\Z^r$-graded prime ideals must be the same; we therefore see that
$\fq \supseteq R_{\mathbf{n}}$ for all $\mathbf{n} \in \nn^r$ with
$\mathbf{n} > \mathbf{0}$. The claim is now immediate from Theorem \ref{mi.3}.
\end{proof}

\section{\sc The ends of certain multi-graded local cohomology modules}
\label{el}

We begin with a combinatorial lemma.

\begin{lem}
\label{el.1} Let $\mathbf{a} := (a_1, \ldots,a_r) \in \Z^r$ and let
$\Sigma$ be a non-empty subset of\/ $\Z^r$ such that $\mathbf{n}
\leq \mathbf{a}$ for all $\mathbf{n} \in \Sigma$. Then $\Sigma$ has
only finitely many maximal elements.
\end{lem}

\begin{note} We are grateful to the referee for drawing our attention to the
following proof, which is shorter than our original.
\end{note}

\begin{proof} The set $\Delta := \mathbf{a} - \Sigma := \{ \mathbf{a} -
\mathbf{n} : \mathbf{n} \in \Sigma\}$ is a non-empty subset of
$\nn^r$. Now $\nn^r$ is a Noetherian monoid with respect to
addition, by \cite[Proposition 1.3.5]{KreRob00}, for example. (All
terminology concerning monoids in this proof is as in \cite[Chapter
1]{KreRob00}.) Therefore the monoideal $(\Delta)$ of $\nn^r$
generated by $\Delta$ can be generated by finitely many elements of
$\Delta$, say by $\mathbf{m}^{(1)}, \ldots, \mathbf{m}^{(s)} \in
\Delta$. Therefore
$$
\Delta \subseteq (\Delta) = \left(\mathbf{m}^{(1)} + \nn^r\right)
\cup \cdots \cup \left(\mathbf{m}^{(s)} + \nn^r\right),
$$
from which it follows that any minimal member of $\Delta$ must
belong to the set $\{\mathbf{m}^{(1)}, \ldots, \mathbf{m}^{(s)}\}$.
Therefore any maximal member of $\Sigma$ must belong to the set
$\{\mathbf{a}-\mathbf{m}^{(1)}, \ldots,
\mathbf{a}-\mathbf{m}^{(s)}\}$.
\end{proof}

\begin{ntn}
\label{el.2} Let $\Sigma, \Delta \subseteq \Z^{r}$. We shall denote
by $\max(\Sigma)$ the set of maximal members of $\Sigma$. (If
$\Sigma$ has no maximal member, then we interpret $\max(\Sigma)$ as
the empty set.)

We shall write $\Sigma \preccurlyeq \Delta$ to indicate that, for
each $\mathbf{n} \in \Sigma$, there exists $\mathbf{m} \in \Delta$
such that $\mathbf{n} \leq \mathbf{m}$; moreover, we shall describe
this situation by the terminology `$\Delta$ {\em dominates\/}
$\Sigma$'. We shall use obvious variants of this terminology.
Observe that, if $\Sigma \preccurlyeq \Delta$ and $\Delta
\preccurlyeq \Sigma$, then $\max(\Sigma) = \max(\Delta)$, and
$\Sigma \preccurlyeq \max(\Sigma)$ if and only if $\Delta
\preccurlyeq \max(\Delta)$.
\end{ntn}

\begin{rmk} [Huy T\`ai H\`a {\cite[\S 2]{ha}}]
\label{lc.1} Let $\phi : \Z^r \lra \Z^m$, where $m$ is a positive
integer, be a homomorphism of Abelian groups. We use the notation
$R^{\phi}$, {\it etcetera}, of Definition \ref{nt.0}.  Let $\fa$ be
a $\Z^r$-graded ideal of $R$. Then $\left((H^i_{\fa} (\:
{\scriptscriptstyle \bullet} \:))^{\phi})\right)_{i \in \nn}$ and
$\left((H^i_{\fa^{\phi}} (\: {\scriptscriptstyle \bullet}
\:^{\phi}))\right)_{i \in \nn}$ are both negative strongly connected
sequences of covariant functors from $\mbox{\rm
*}\mathcal{C}^{\Z^r}(R)$ to $\mbox{\rm
*}\mathcal{C}^{\Z^m}(R^{\phi})$; moreover, the $0$th members of
these two connected sequences are the same functor, and, whenever,
$I$ is a $\mbox{\rm *}$-injective $\Z^r$-graded $R$-module and $i >
0$, we have $H^i_{\fa} (I) = 0$ when all gradings are forgotten, so
that $(H^i_{\fa} (I))^{\phi} = 0$ and $H^i_{\fa^{\phi}}(I^{\phi}) =
0$. Consequently, the two above-mentioned connected sequences are
isomorphic. Hence, for each $\Z^r$-graded $R$-module $M$, there is a
$\Z^m$-homogeneous isomorphism of $\Z^m$-graded $R^{\phi}$-modules
$$
(H^i_{\fa}(M))^{\phi} \cong H^i_{\fa^{\phi}}(M^{\phi}) \quad \mbox{~for each~}
i \in \nn.
$$
\end{rmk}

\begin{ntn}
\label{el.3} Throughout this section, we shall be concerned with the
situation where $$R = \bigoplus_{\mathbf{n} \in \nn^r}
R_\mathbf{n}$$ is positively graded; we shall only assume that $R$
is standard when this is explicitly stated.

We shall be greatly concerned with the $\nn^r$-graded ideal $$\fc :=
\fc(R) := \bigoplus_{\stackrel{\scriptstyle \mathbf{n} \in
\nn^r}{\mathbf{n}
> \mathbf{0}}} R_\mathbf{n}.
$$
We shall accord $R_+$ its usual meaning (see E. Hyry \cite[p.\
2215]{Hyry99}), so that
$$
R_+ := \bigoplus_{\stackrel{\scriptstyle \mathbf{n} \in
\nn^r}{\mathbf{n} \geq \mathbf{1}}} R_\mathbf{n} =
\bigoplus_{\mathbf{n} \in \N^r} R_\mathbf{n}.
$$
Observe that, when $r= 1$, we have $\fc = R_+$. However, in general
this is not the case when $r > 1$.
\end{ntn}

\begin{defi}
\label{el.4} Suppose that $R = \bigoplus_{\mathbf{n} \in \nn^r}
R_\mathbf{n}$ is positively graded and standard; let $M =
\bigoplus_{\mathbf{n} \in \Z^r} M_\mathbf{n}$ be a finitely
generated $\Z^r$-graded $R$-module, and let $j \in \nn$.

Let $\fb$ be an $\nn^r$-graded ideal such that $\dir(\fb) \neq
\emptyset$, and let $i \in \dir(\fb)$; consider the Abelian group
homomorphism $\phi_i:\Z^r \lra \Z$ for which $\phi_i((n_1, \ldots,
n_r)) = n_i$ for all $(n_1, \ldots, n_r)\in \Z^r$, which is just the
$i$th coordinate function.

By Lemma \ref{mu.2}, since $R_{\mathbf{e}_i} \subseteq \sqrt{\fb}$,
we have
$$
(R^{\phi_i})_+ = \bigoplus_{\stackrel{\scriptstyle{\mathbf{n} \in
\nn^r}}{n_i > 0}}R_\mathbf{n} \subseteq \sqrt{\fb}^{\phi_i}.
$$
It therefore follows from \cite[Corollary 2.5]{70}, with the
notation of that paper, that the $\nn$-graded $R^{\phi_i}$-module
$(H^j_{\fb}(M))^{\phi_i} \cong H^j_{\fb^{\phi_i}}(M^{\phi_i})$, if
non-zero, has finite end satisfying
$$ \nd ((H^j_{\fb}(M))^{\phi_i}) \leq a^*(M^{\phi_i}) =
\sup\{\nd(H^k_{(R^{\phi_i})_+}(M^{\phi_i})): k \in \nn\} =
\sup\{a^k_{(R^{\phi_i})_+}(M^{\phi_i}): k \in \nn\}.
$$
(Note that, in these circumstances, the invariant $a^*(M^{\phi_i})$
is an integer.) Thus, if $\mathbf{n} := (n_1, \ldots,n_r) \in \Z^r$
is such that $H^j_{\fb}(M)_{\mathbf{n}} \neq 0$, then $n_i \leq
a^*(M^{\phi_i})$. Thus there exists $\mathbf{a} \in
\Z^{\#\dir(\fb)}$ such that, for all $\mathbf{n} := (n_1,
\ldots,n_r) \in \Z^r$ with $H^j_{\fb}(M)_{\mathbf{n}} \neq 0$, we
have $\phi(\fb)(\mathbf{n}) \leq \mathbf{a}$. We define the {\em end
of\/} $H^j_{\fb}(M)$ by
$$
\nd(H^j_{\fb}(M)) := \max\left\{\phi(\fb)(\mathbf{n}) : \mathbf{n}
\in \Z^r \mbox{~and~}H^j_{\fb}(M)_{\mathbf{n}} \neq 0\right\}.
$$
By Lemma \ref{el.1}, if $H^j_{\fb}(M) \neq 0$ and $\dir(\fb) \neq
\emptyset$, then this end is a non-empty finite set of points of
$\Z^{\#\dir(\fb)}$. Note that the end of $H^j_{\fb}(M)$ dominates
$\phi(\fb)(\mathbf{n})$ for every $\mathbf{n} \in \Z^r$ for which
$H^j_{\fb}(M)_{\mathbf{n}} \neq 0$.

We draw the reader's attention to the fact that, when $r > 1$ and
$R_{\mathbf{e}_i} \neq 0$ for all $i \in \{1, \ldots, r\}$, the
ideal $R_+ = \bigoplus_{\stackrel{\scriptstyle \mathbf{n} \in
\nn^r}{\mathbf{n} \geq \mathbf{1}}} R_\mathbf{n}$ has empty set of
directions; consequently, we have not defined the end of the $i$th
local cohomology module $H^i_{R_+}(M)$ of $M$ with respect to $R_+$.
Thus we are not, in this paper, making any contribution to the
theory of multi-graded Castelnuovo regularity, and, in particular,
we are not proposing an alternative definition of $a$-invariant or
$a^*$-invariant (see \cite[Definitions 3.1.1 and 3.1.2]{ha}).
\end{defi}

With this definition of the ends of (certain) multi-graded local
cohomology modules, we can now establish multi-graded analogues of
some results in \cite[\S 2]{70}.

\begin{thm}
\label{el.5} Suppose that $R := \bigoplus_{\mathbf{n}\in \nn^r}
R_{\mathbf{n}}$ is positively graded and standard. Let $M$ be a
finitely generated $\Z^r$-graded $R$-module, and let
\[
I^{\scriptscriptstyle \bullet} : 0 \longrightarrow \mbox{\rm
*}E^0(M) \stackrel{d^0}{\longrightarrow} \mbox{\rm *}E^1(M)
\longrightarrow \cdots \longrightarrow \mbox{\rm *}E^i(M
)\stackrel{d^i}{\longrightarrow} \mbox{\rm *}E^{i+1}(M)
\longrightarrow \cdots
\]
be the minimal $\mbox{\rm *}$injective resolution of $M$.

Let $\fb$ be an $\nn^r$-graded ideal such that $\dir(\fb) \neq
\emptyset$, and let $j \in \nn$. Then
\begin{align*}
\max \left( \bigcup_{i=0}^j \nd(H^j_{\fb}(M))\right) & = \max
\left\{ \phi(\fb)(\mathbf{n}) : \mathbf{n} \in \Z^r \mbox{~and~}
(\Gamma_{\fb}(\mbox{\rm *}E^i(M)))_{\mathbf{n}} \neq 0 \mbox{~for
some~}i \in \{0,\ldots,j\} \right\}\\ & = \max\left(\bigcup_{i=0}^j
\bigcup_{\fp \in \mbox{\rm
*}\Var(\fb)}\phi(\fp;\fb)(\anch^i(\fp,M))\right).
\end{align*}

\end{thm}

\begin{proof} Let $i \in \nn$ and set
$$\Delta_i := \left\{\phi(\fb)(\mathbf{n}) : \mathbf{n} \in \Z^r
\mbox{~and~}H^i_{\fb}(M)_{\mathbf{n}} \neq 0 \right\}, \quad
\Sigma_i := \left\{ \phi(\fb)(\mathbf{n}) : \mathbf{n} \in \Z^r
\mbox{~and~} (\Gamma_{\fb}(\mbox{\rm *}E^i(M)))_{\mathbf{n}} \neq 0
 \right\}
$$
and
$$
\Phi_i := \bigcup_{\fp \in \mbox{\rm
*}\Var(\fb)}\phi(\fp;\fb)(\anch^i(\fp,M)).
$$
Also, let
$$
\theta_i : \mbox{\rm *}E^{i}(M) \stackrel{\cong}{\longrightarrow}
\bigoplus_{\alpha \in \Lambda_i} \mbox{\rm
*}E(R/{\mathfrak{p}}_{\alpha})(- \mathbf{n}_{\alpha})
$$
be a $\Z^r$-homogeneous isomorphism, where ${\mathfrak{p}}_{\alpha}
\in \mbox{\rm *}\Spec (R)$ and $\mathbf{n}_{\alpha} \in \Z^r$ for
all $\alpha \in \Lambda_i$.

We shall first show that $\Delta_i \preccurlyeq \Sigma_i
\preccurlyeq \Phi_i$. Now $H^i_{\fb}(M)$ is a homomorphic image, by
a $\Z^r$-homogeneous epimorphism, of
$$
\Ker \left(\Gamma_{\fb}(d^i) : \Gamma_{\fb}(\mbox{\rm *}E^i(M)) \lra
\Gamma_{\fb}(\mbox{\rm
*}E^{i+1}(M))\right).
$$
Therefore, if $\mathbf{n} \in \Z^r$ is such that
$H^i_{\fb}(M)_{\mathbf{n}} \neq 0$, then $(\Gamma_{\fb}(\mbox{\rm
*}E^i(M)))_{\mathbf{n}} \neq 0$.  This proves that $\Delta_i
\subseteq \Sigma_i$, so that $\Delta_i \preccurlyeq \Sigma_i$.

Furthermore, given $\mathbf{n} \in \Z^r$ such that
$(\Gamma_{\fb}(\mbox{\rm
*}E^i(M)))_{\mathbf{n}} \neq 0$, we can see from the isomorphism
$\theta_i$ that there must exist $\alpha \in \Lambda_i$ such that
$\fb \subseteq \fp_{\alpha}$ and $(\mbox{\rm
*}E(R/{\mathfrak{p}}_{\alpha})(- \mathbf{n}_{\alpha}))_{\mathbf{n}}
\neq 0$. It now follows from Proposition \ref{mu.4}(ii) that
$\phi(\fp_{\alpha})(\mathbf{n}) \leq
\phi(\fp_{\alpha})(\mathbf{n}_{\alpha})$, so that
$$
\phi(\fp_{\alpha};\fb)(\phi(\fp_{\alpha})(\mathbf{n})) \leq
\phi(\fp_{\alpha};\fb)(\phi(\fp_{\alpha})(\mathbf{n}_{\alpha})).
$$
Now $\phi(\fp_{\alpha})(\mathbf{n}_{\alpha})$ is an $i$th level
anchor point of $\fp_{\alpha}$ for $M$, and
$\phi(\fp_{\alpha};\fb)\circ\phi(\fp_{\alpha}) = \phi(\fb)$. This is
enough to prove that $\Sigma_i \preccurlyeq \Phi_i$.

In particular, we have proved that $\Delta_0 \preccurlyeq \Sigma_0
\preccurlyeq \Phi_0$. We shall prove the desired result by induction
on $j$. We show next that $\Phi_0 \preccurlyeq \Delta_0$, and this,
together with the above, will prove the claim in the case where $j =
0$. Let $\mathbf{m} \in \Phi_0$. Thus $\mathbf{m} \in
\Z^{\#\dir(\fb)}$ and there exists $\alpha \in \Lambda_0$ such that
$\fp_{\alpha} \in \mbox{\rm *}\Var(\fb)$ and $\mathbf{m} =
\phi(\fp_{\alpha};\fb)(\phi(\fp_{\alpha})(\mathbf{n}_{\alpha}))$.
Now the image of
$$
\bigoplus_{\stackrel{\scriptstyle \mathbf{n} \in
\Z^{r}}{\phi(\fp_{\alpha})(\mathbf{n}) \geq
\phi(\fp_{\alpha})(\mathbf{n}_{\alpha})}}\left(\mbox{\rm
*}E(R/{\mathfrak{p}}_{\alpha})(-
\mathbf{n}_{\alpha})\right)_{\mathbf{n}}
$$
under $\theta_0^{-1}$ is a non-zero $\Z^{r}$-graded submodule of
$\Gamma_{\fb}(\mbox{\rm
*}E^0(M))$; as the latter is a
$\mbox{\rm *}$essential extension of $\Gamma_{\fb}(M)$, it follows
that there exists $\mathbf{n} \in \Z^{r}$ with
$\phi(\fp_{\alpha})(\mathbf{n}) \geq
\phi(\fp_{\alpha})(\mathbf{n}_{\alpha})$ such that
$\left(\Gamma_{\fb}(M)\right)_{\mathbf{n}} \neq 0$. Moreover,
$$
\phi(\fb)(\mathbf{n}) =
\phi(\fp_{\alpha};\fb)\left(\phi(\fp_{\alpha})(\mathbf{n})\right)
\geq
\phi(\fp_{\alpha};\fb)\left(\phi(\fp_{\alpha})(\mathbf{n}_{\alpha})\right)
= \mathbf{m}.
$$
It follows that $\Phi_0 \preccurlyeq \Delta_0$, so that
$\max(\Delta_0) = \max(\Sigma_0) = \max(\Phi_0)$, and the desired
result has been proved when $j = 0$.

Now suppose that $j > 0$ and make the obvious inductive assumption.
As we have already proved that $\Delta_i \preccurlyeq \Sigma_i$ and
$\Sigma_i \preccurlyeq \Phi_i$ for all $i = 0, \ldots, j$, it will
be enough, in order to complete the inductive step, for us to prove
that $\Phi_j \preccurlyeq \bigcup_{k=0}^j\Delta_k$. So consider
$\alpha \in \Lambda_j$ such that $\fp_{\alpha} \in \mbox{\rm
*}\Var(\fb)$; we shall show
that
$\phi(\fp_{\alpha};\fb)(\phi(\fp_{\alpha})(\mathbf{n}_{\alpha}))$ is
dominated by a member of $\Delta_0 \cup \Delta_1 \cup \cdots \cup
\Delta_{j-1} \cup \Delta_j$.

Now the image of
$$
\bigoplus_{\stackrel{\scriptstyle \mathbf{n} \in
\Z^{r}}{\phi(\fp_{\alpha})(\mathbf{n}) \geq
\phi(\fp_{\alpha})(\mathbf{n}_{\alpha})}}\left(\mbox{\rm
*}E(R/{\mathfrak{p}}_{\alpha})(-
\mathbf{n}_{\alpha})\right)_{\mathbf{n}}
$$
under $\theta_j^{-1}$ is a non-zero $\Z^{r}$-graded submodule of
$\Gamma_{\fb}(\mbox{\rm *}E^j(M))$; as the latter is a $\mbox{\rm
*}$essential extension of $\Ker \Gamma_{\fb}(d^j)$, it follows that
there exists $\mathbf{n} \in \Z^{r}$ with
$\phi(\fp_{\alpha})(\mathbf{n}) \geq
\phi(\fp_{\alpha})(\mathbf{n}_{\alpha})$ such that $\left(\Ker
\Gamma_{\fb}(d^j)\right)_{\mathbf{n}} \neq 0$. There is an exact
sequence
$$
0 \lra \Ima \Gamma_{\fb}(d^{j-1})\lra \Ker\Gamma_{\fb}(d^j) \lra
H^j_{\fb}(M) \lra 0
$$
of graded $\Z^{r}$-modules and homogeneous homomorphisms. Therefore
either $H^j_{\fb}(M)_{\mathbf{n}} \neq 0$ or $$\left(\Ima
\Gamma_{\fb}(d^{j-1})\right)_{\mathbf{n}} \neq 0.$$ In the first
case, $\phi(\fp_{\alpha};\fb)(\phi(\fp_{\alpha})(\mathbf{n})) =
\phi(\mathbf{b})(\mathbf{n}) \in \Delta_j$. In the second case,
$(\Gamma_{\fb}(\mbox{\rm
*}E^{j-1}(M)))_{\mathbf{n}} \neq 0$, whence
$\phi(\mathbf{b})(\mathbf{n}) \in \Sigma_{j-1}$, so that, by the
inductive hypothesis, $\phi(\mathbf{b})(\mathbf{n})$ is dominated by
an element of $\Delta_0 \cup \Delta_1 \cup \cdots \cup
\Delta_{j-1}$; thus, in this case also,
$\phi(\fp_{\alpha};\fb)(\phi(\fp_{\alpha})(\mathbf{n}_{\alpha}))$ is
dominated by an element of $\bigcup_{k=0}^{j} \Delta_k$. This is
enough to complete the inductive step.
\end{proof}

\begin{ntn}
\label{el.6} Suppose that $R := \bigoplus_{\mathbf{n}\in \nn^r}
R_{\mathbf{n}}$ is positively graded and standard, and let $M$ be a
finitely generated $\Z^r$-graded $R$-module. Let $\mathcal{Q}$ be a
non-empty subset of $\{1, \ldots,r\}$. Define $\fc^{\mathcal{Q}} : =
\sum_{i \in \mathcal{Q}}R_{\mathbf{e}_i}R$. Then $\dir
(\fc^{\mathcal{Q}}) \supseteq \mathcal{Q}$, and $\fc^{\mathcal{Q}}$
is the smallest ideal (up to radical) with set of directions
containing $\mathcal{Q}$. We also define the {\em
$\mathcal{Q}$-bound} $\bnd^{\mathcal{Q}}(M)$ of $M$ by
$$
\bnd^{\mathcal{Q}}(M) := \max\left(\bigcup_{i \in
\nn}\nd(H^i_{\fc^{\mathcal{Q}}}(M))\right).
$$
Observe that $\bnd^{\mathcal{Q}}(M)$ is a finite set of points in
$\Z^{\#\dir(\fc^{\mathcal{Q}})}$, because
$H^i_{\fc^{\mathcal{Q}}}(M) = 0$ whenever $i$ exceeds the arithmetic
rank of $\fc^{\mathcal{Q}}$.

For consistency with our earlier notation in \ref{el.3}, we
abbreviate $\fc^{\{1,\ldots,r\}} =
\sum_{\mathbf{n}>\mathbf{0}}R_{\mathbf{n}}$ by $\fc$. Note that
$\bnd^{\{1,\ldots,r\}}(M) = \max\left(\bigcup_{i \in
\nn}\nd(H^i_{\fc}(M))\right)$ is a finite set of points in $\Z^r$.
\end{ntn}

The following corollaries, which are multi-graded analogues of
\cite[Corollaries 2.5, 2.6]{70}, can now be deduced immediately from
Theorem \ref{el.5}.

\begin{cor}
\label{el.7} Suppose that $R := \bigoplus_{\mathbf{n}\in \nn^r}
R_{\mathbf{n}}$ is positively graded and standard. Let $M$ be a
finitely generated $\Z^r$-graded $R$-module, and let
\[
I^{\scriptscriptstyle \bullet} : 0 \longrightarrow \mbox{\rm
*}E^0(M) \stackrel{d^0}{\longrightarrow} \mbox{\rm *}E^1(M)
\longrightarrow \cdots \longrightarrow \mbox{\rm *}E^i(M
)\stackrel{d^i}{\longrightarrow} \mbox{\rm *}E^{i+1}(M)
\longrightarrow \cdots
\]
be the minimal $\mbox{\rm *}$injective resolution of $M$.

Let $\fb$ be an $\nn^r$-graded ideal of $R$ such that $\dir(\fb)
\neq \emptyset$, and let $j \in \nn$. Then
\begin{align*}
\max \left( \bigcup_{i=0}^j \nd(H^j_{\fb}(M))\right)  & \preccurlyeq
\max\left(\bigcup_{i=0}^j \bigcup_{\fp \in \mbox{\rm
*}\Var(\fc^{\dir(\fb)})}\phi(\fp;\fc^{\dir(\fb)})(\anch^i(\fp,M))\right)\\&
= \max \left\{ \phi(\fb)(\mathbf{n}) : \mathbf{n} \in \Z^r
\mbox{~and~} (\Gamma_{\fc^{\dir(\fb)}}(\mbox{\rm
*}E^i(M)))_{\mathbf{n}} \neq 0 \mbox{~for an~} i\in\{0,\ldots,j\}
\right\}\\ & = \max \left( \bigcup_{i=0}^j
\nd(H^j_{\fc^{\dir(\fb)}}(M))\right) \preccurlyeq
\bnd^{\dir(\fb)}(M).
\end{align*}
\end{cor}

\begin{cor}
\label{el.8} Suppose that $R := \bigoplus_{\mathbf{n}\in \nn^r}
R_{\mathbf{n}}$ is positively graded and standard. Let $M$ be a
finitely generated $\Z^r$-graded $R$-module.

Let $\fb$ be an $\nn^r$-graded ideal of $R$ of arithmetic rank $t$
such that $\dir(\fb) \neq \emptyset$, and let $k \in \N$ with $k >
t$. Then
\begin{align*}
\max\left(\bigcup_{i=0}^t \bigcup_{\fp \in \mbox{\rm
*}\Var(\fb)}\phi(\fp;\fb)(\anch^i(\fp,M))\right)
& = \max \left( \bigcup_{i=0}^t \nd(H^i_{\fb}(M))\right)  = \max
\left( \bigcup_{i=0}^k \nd(H^i_{\fb}(M))\right)\\ & =
\max\left(\bigcup_{i=0}^k \bigcup_{\fp \in \mbox{\rm
*}\Var(\fb)}\phi(\fp;\fb)(\anch^i(\fp,M))\right).
\end{align*}
Consequently, for a $\fp \in \mbox{\rm
*}\Var(\fb)$ and $\mathbf{a} \in \anch(\fp,M)$, we can conclude that
$\phi(\fp;\fb)(\mathbf{a})$ is dominated by
$$
\max\left(\bigcup_{i=0}^t \bigcup_{\fp \in \mbox{\rm
*}\Var(\fb)}\phi(\fp;\fb)(\anch^i(\fp,M))\right),
$$
a set of points of $\Z^{\#\dir(\fb)}$ which arises from
consideration of just the $0$th, $1$st, $\ldots$, $(t-1)$th and
$t$th terms of the minimal $\mbox{\rm *}$injective resolution of
$M$.
\end{cor}

Our next aim is the establishment of multi-graded analogues of
\cite[Corollaries 3.1 and 3.2]{70}.

\begin{lem}\label{el.9} Suppose that $R := \bigoplus_{\mathbf{n}\in \nn^r}
R_{\mathbf{n}}$ is positively graded, and let $\fm$ be a $\mbox{\rm
*}$maximal ideal of $R$. Then $\fm_{\mathbf{0}} := \fm \cap
R_{\mathbf{0}}$ is a maximal ideal of $R_{\mathbf{0}}$ and $\fm =
\fm_{\mathbf{0}}\oplus \fc$, where $\fc$ is as defined in Notation\/
{\rm \ref{el.3}}.
\end{lem}

\begin{proof} Recall that $$\fc :=
\bigoplus_{\stackrel{\scriptstyle \mathbf{n} \in \nn^r}{\mathbf{n}
> \mathbf{0}}} R_\mathbf{n}.$$
Since $\fm_{\mathbf{0}} \in \Spec (R_{\mathbf{0}})$, it follows that
$R \supset \fm_{\mathbf{0}}\bigoplus \fc \supseteq \fm$, so that
$\fm = \fm_{\mathbf{0}}\bigoplus \fc$. Furthermore,
$\fm_{\mathbf{0}}$ must be a maximal ideal of $R_{\mathbf{0}}$.
\end{proof}

\begin{cor}
\label{el.10} Suppose that $R := \bigoplus_{\mathbf{n}\in \nn^r}
R_{\mathbf{n}}$ is positively graded and standard. Let $M$ be a
finitely generated $\Z^r$-graded $R$-module; let $\fb$ be an
$\nn^r$-graded ideal of $R$ such that $\dir(\fb) \neq \emptyset$.
Then
\[
\max \left( \bigcup_{i\in \nn} \nd(H^i_{\fb}(M))\right) = \max
\left(\bigcup_{\fm \in \mbox{\rm *}\Var(\fb) \cap \mbox{\rm
*}\Max(R)}\bigcup_{i\in \nn} \phi(\fm;\fb) \nd(H^i_{\fm}(M))\right).
\]
\end{cor}

\begin{proof} Let $\fm \in \mbox{\rm *}\Var(\fb) \cap \mbox{\rm
*}\Max(R)$. By Lemma \ref{el.9}, $\dir(\fm) = \{1, \ldots,r\}$;
therefore, by Theorem \ref{el.5}, $\max\left(\bigcup_{i \in
\nn}\nd(H^i_{\fm}(M))\right) = \max\left(\bigcup_{i\in\nn}
\anch^i(\fm,M)\right)$. Another use of Theorem \ref{el.5} therefore
shows that
\begin{align*}
\max\left(\bigcup_{i \in \nn}\phi(\fm;\fb)(\nd(H^i_{\fm}(M)))\right)
&=
\max\left(\bigcup_{i\in\nn} \phi(\fm;\fb)(\anch^i(\fm,M))\right)\\
&\preccurlyeq \max\left(\bigcup_{i\in\nn} \bigcup_{\fp \in \mbox{\rm
*}\Var(\fb)} \phi(\fp;\fb)(\anch^i(\fp,M))\right)\\&=
\max\left(\bigcup_{i \in \nn}\nd(H^i_{\fb}(M))\right).
\end{align*}
We have thus proved that
\[
\max \left( \bigcup_{i\in \nn} \nd(H^i_{\fb}(M))\right) \succcurlyeq
\max \left(\bigcup_{\fm \in \mbox{\rm *}\Var(\fb) \cap \mbox{\rm
*}\Max(R)}\bigcup_{i\in \nn} \phi(\fm;\fb) (\nd(H^i_{\fm}(M)))\right).
\]

Now let $\mathbf{n} \in \Z^{\#\dir(\fb)}$ be a maximal member of
$\bigcup_{i \in \nn}\nd(H^i_{\fb}(M))$. By Theorem \ref{el.5}, there
exist $s \in \nn$ and $\fp \in \mbox{\rm *}\Var(\fb)$ such that
$\mathbf{n} = \phi(\fp;\fb)(\mathbf{w})$ for some $s$th level anchor
point $\mathbf{w}$ of $\fp$ for $M$. Now use Theorem \ref{mi.3}
repeatedly, in conjunction with a saturated chain (of length $t$
say) of $\nn^r$-graded prime ideals of $R$ with $\fp$ as its
smallest term and a $\mbox{\rm *}$maximal ideal $\fm$ as its largest
term: the conclusion is that there exists $\mathbf{v} \in
\anch^{s+t}(\fm,M)$ such that $\phi(\fm;\fp)(\mathbf{v}) =
\mathbf{w}$. Now
$$
\mathbf{n} = \phi(\fp;\fb)(\mathbf{w}) =
\phi(\fp;\fb)(\phi(\fm;\fp)(\mathbf{v})) =
\phi(\fm;\fb)(\mathbf{v}).
$$
But, by Theorem \ref{el.5} again, $\mathbf{v}$ is dominated by
$\max\left(\bigcup_{i\in \nn} \nd(H^i_{\fm}(M))\right)$; it follows
that
\[
\max \left( \bigcup_{i\in \nn} \nd(H^i_{\fb}(M))\right) \preccurlyeq
\max \left(\bigcup_{\fm \in \mbox{\rm *}\Var(\fb) \cap \mbox{\rm
*}\Max(R)}\bigcup_{i\in \nn}
\phi(\fm;\fb)(\nd(H^i_{\fm}(M)))\right).
\]
The desired conclusion follows.
\end{proof}

\begin{cor}\label{el.11} Let the situation be as in Corollary\/ {\rm
\ref{el.10}}, but assume in addition that $(R_{\mathbf{0}},
\fm_{\mathbf{0}})$ is local and that $\fb$ is proper; set $\fm :=
\fm_{\mathbf{0}}\oplus \fc$, where $\fc$ is as defined in Notation\/
{\rm \ref{el.3}}. Then
\[
\max \left( \bigcup_{i\in \nn} \nd(H^i_{\fb}(M))\right) = \max
\left(\bigcup_{i\in \nn} \phi(\fm;\fb)(\nd(H^i_{\fm}(M)))\right).
\]
In particular,
\[
\max \left( \bigcup_{i\in \nn} \nd(H^i_{\fc}(M))\right) = \max
\left(\bigcup_{i\in \nn}\nd(H^i_{\fm}(M))\right).
\]
\end{cor}

\section{\sc Some vanishing results for multi-graded components of local cohomology modules}
\label{va}

It is well known that, when $r = 1$, if $M$ is a finitely generated
$\Z$-graded $R$-module, then there exists $t \in \Z$ such that
$H^i_{R_+}(M)_n = 0$ for all $i \in \nn$ and all $n \geq t$; it then
follows from \cite[Corollary 2.5]{70} that, if $\fb$ is any graded
ideal of $R$ with $\fb \supseteq R_+$, then $H^i_{\fb}(M)_n = 0$ for
all $i \in \nn$ and all $n \geq t$. One of the aims of this section
is to establish a multi-graded analogue of this result.

\begin{ntn}
\label{lc.2} Throughout this section, we shall be concerned with the
situation where $$R = \bigoplus_{\mathbf{n} \in \nn^r}
R_\mathbf{n}$$ is positively graded; we shall only assume that $R$
is standard when this is explicitly stated.

We shall be concerned with the $\nn^r$-graded ideal $R_+$ of $R$
given (see Notation \ref{el.3}) by
$$
R_+ := \bigoplus_{\stackrel{\scriptstyle \mathbf{n} \in
\nn^r}{\mathbf{n} \geq \mathbf{1}}} R_\mathbf{n}.
$$

Although it is well known (see Hyry \cite[Theorem 1.6]{Hyry99})
that, if $M$ is a finitely generated $\Z^r$-graded $R$-module, then
$H^i_{R_+}(M)_{(n_1, \ldots,n_r)} = 0$ for all $n_1, \ldots, n_r \gg
0$, we have not been able to find in the literature a proof of the
corresponding statement with $R_+$ replaced by an $\nn^r$-graded
ideal $\fb$ that contains $R_+$. We present such a proof below,
because we think it is of interest in its own right.
\end{ntn}

\begin{thm}
\label{lc.3} Suppose that $R = \bigoplus_{\mathbf{n} \in \nn^r}
R_\mathbf{n}$ is positively graded; let $M$ be a finitely generated
$\Z^r$-graded $R$-module. Let $\fb$ be an $\nn^r$-graded ideal of
$R$ such that $\fb \supseteq R_+$. Then there exists $t \in \Z$ such
that
$$H^i_{\fb}(M)_{\mathbf{n}} = 0 \quad \mbox{~for all~} i \in \nn
\mbox{~and all~} \mathbf{n} \geq (t,t,\ldots,t).
$$
\end{thm}

\begin{proof} We shall prove this by induction on $r$. In the
case where $r = 1$ the result follows from \cite[Corollary 2.5]{70},
as was explained in the introduction to this section.

Now suppose that $r >1$ and that the claim has been proved for
smaller values of $r$. We define three more $\nn^r$-graded ideals
$\fa$, $\fc$ and $\fd$ of $R$, as follows. Set
$$
\fa := \bigoplus_{\mathbf{n} = (n_1, \ldots, n_r) \in
\nn^r}\fa_{\mathbf{n}} \quad \mbox{~where~} \fa_{\mathbf{n}} =
\begin{cases} \fb_{\mathbf{n}} &    \text{if } n_r =
0,\\R_{\mathbf{n}} &    \text{if } n_r > 0\mbox{;}
\end{cases}
$$
$$
\fc := \bigoplus_{\mathbf{n} = (n_1, \ldots, n_r) \in
\nn^r}\fc_{\mathbf{n}} \quad \mbox{~where~} \fc_{\mathbf{n}} =
\begin{cases} \fb_{\mathbf{n}} &    \text{if } (n_1,\ldots,n_{r-1}) \not\geq(1, \ldots,1),\\
R_{\mathbf{n}} &    \text{if } (n_1,\ldots,n_{r-1}) \geq(1,
\ldots,1)\mbox{;}
\end{cases}
$$
and $\fd := \fa + \fc$.

Consider $\fa \cap \fc$: for each $\mathbf{n} = (n_1, \ldots, n_r)
\in \nn^r$, the $\mathbf{n}$th component $(\fa \cap
\fc)_{\mathbf{n}}$ satisfies
$$
(\fa \cap \fc)_{\mathbf{n}} = \fa_{\mathbf{n}} \cap \fc_{\mathbf{n}}
= \begin{cases} \fb_{\mathbf{n}} &    \text{if } n_r = 0 \text{ or }
(n_1,\ldots,n_{r-1}) \not\geq(1, \ldots,1),\\
R_{\mathbf{n}} &    \text{if } n_r > 0 \text{ and
}(n_1,\ldots,n_{r-1}) \geq(1, \ldots,1).
\end{cases}
$$
Since $\fb \supseteq R_+$, we see that $\fa \cap \fc = \fb$.

Let $\sigma : \Z^r \lra \Z^{r-1}$ be the group homomorphism defined
by
$$\sigma ((n_1,\ldots,n_r)) = (n_1 + n_r,\ldots, n_{r-1} + n_r) \quad
\mbox{~for all~} (n_1,\ldots,n_r) \in \Z^r.
$$
Note that, for $(n_1,\ldots,n_r) \in \nn^r$, we have $(n_1 +
n_r,\ldots, n_{r-1} + n_r) \geq \mathbf{1}$ in $\Z^{r-1}$ if and
only if $n_r \geq 1$ or $(n_1,\ldots, n_{r-1}) \geq \mathbf{1}$;
furthermore, if $n_r \geq 1$, then $\fa_{\mathbf{n}} =
R_{\mathbf{n}}$, and if $(n_1,\ldots, n_{r-1}) \geq \mathbf{1}$,
then $\fc_{\mathbf{n}} = R_{\mathbf{n}}$. Let $\mathbf{m} \in
\Z^{r-1}$ with $\mathbf{m} \geq \mathbf{1}$. Therefore, in the
$\nn^{r-1}$-graded ring $R^{\sigma}$, we have
$$(\fd^{\sigma})_{\mathbf{m}} = \bigoplus_{\stackrel{\scriptstyle
\mathbf{n} \in \Z^r}{\sigma(\mathbf{n}) =
\mathbf{m}}}(\fa_{\mathbf{n}} + \fc_{\mathbf{n}}) =
\bigoplus_{\stackrel{\scriptstyle \mathbf{n} \in
\Z^r}{\sigma(\mathbf{n}) = \mathbf{m}}} R_{\mathbf{n}} =
(R^{\sigma})_{\mathbf{m}}.
$$
Thus $\fd^{\sigma} \supseteq \bigoplus_{\mathbf{m} \geq \mathbf{1}}
(R^{\sigma})_\mathbf{m} = (R^{\sigma})_+$.

It therefore follows from the inductive hypothesis that there exists
$\widetilde{t} \in \Z$ such that
$(H^j_{\fd^{\sigma}}(M^{\sigma}))_{\mathbf{h}} = 0$ for all $j \in
\nn$ and all $\mathbf{h} \geq (\widetilde{t},\ldots,\widetilde{t})$
in $\Z^{r-1}$. In view of Remark \ref{lc.1}, this means that
$((H^j_{\fd}(M))^{\sigma})_{\mathbf{h}} = 0$ for all $j \in \nn$ and
all $\mathbf{h} \geq (\widetilde{t},\ldots,\widetilde{t})$ in
$\Z^{r-1}$, so that, for all $j \in \nn$,
$$
H^j_{\fd}(M)_{(n_1, \ldots, n_r)} = 0 \quad \mbox{~whenever~}
(n_1,\ldots, n_{r-1},n_r) \geq \left({\textstyle
\frac{1}{2}}\widetilde{t},\ldots,{\textstyle
\frac{1}{2}}\widetilde{t}, {\textstyle
\frac{1}{2}}\widetilde{t}\,\right) \mbox{~in~} \Z^r.
$$

We now give two similar, but simpler, arguments. Let $\pi : \Z^r
\lra \Z$ be the group homomorphism given by projection onto the
$r$th co-ordinate. Note that, for $\mathbf{n} \in \nn^r$, if
$\pi(\mathbf{n}) \geq 1$, then $\fa_{\mathbf{n}} = R_{\mathbf{n}}$.
Therefore $\fa^{\pi} \supseteq (R^{\pi})_+$. It therefore follows
from the case where $r = 1$ that there exists $\overline{t} \in \Z$
such that $(H^j_{\fa^{\pi}}(M^{\pi}))_n = 0$ for all $j \in \nn$ and
all $n \geq \overline{t}$. In view of Remark \ref{lc.1}, this means
that $((H^j_{\fa}(M))^{\pi})_n = 0$ for all $j \in \nn$ and all $n
\geq \overline{t}$, that is,
$$
H^j_{\fa}(M)_{(n_1, \ldots, n_r)} = 0 \quad \mbox{~whenever~} j \in
\nn \mbox{~and~} n_r \geq \overline{t}.
$$

Next, let $\theta : \Z^r \lra \Z^{r-1}$ be the group homomorphism
defined by
$$\theta ((n_1,\ldots,n_r)) = (n_1,\ldots, n_{r-1}) \quad \mbox{~for all~}
(n_1,\ldots,n_r) \in \Z^r.$$ Note that, if $\mathbf{n} \in \Z^r$ has
$\theta(\mathbf{n}) \geq \mathbf{1}$ in $\Z^{r-1}$, then
$\fc_{\mathbf{n}} = R_{\mathbf{n}}$. Therefore, for $\mathbf{m} \in
\Z^{r-1}$ with $\mathbf{m} \geq \mathbf{1}$, we have
$(\fc^{\theta})_{\mathbf{m}} = (R^{\theta})_{\mathbf{m}}$. This
means that, in the $\nn^{r-1}$-graded ring $R^{\theta}$, we have
$\fc^{\theta} \supseteq \bigoplus_{\mathbf{m} \geq \mathbf{1}}
(R^{\theta})_\mathbf{m} = (R^{\theta})_+$.

It therefore follows from the inductive hypothesis that there exists
$\widehat{t} \in \Z$ such that
$(H^j_{\fc^{\theta}}(M^{\theta}))_{\mathbf{h}} = 0$ for all $j \in
\nn$ and all $\mathbf{h} \geq (\widehat{t},\ldots,\widehat{t})$ in
$\Z^{r-1}$. In view of Remark \ref{lc.1}, this means that
$((H^j_{\fc}(M))^{\theta})_{\mathbf{h}} = 0$ for all $j \in \nn$ and
all $\mathbf{h} \geq (\widehat{t},\ldots,\widehat{t})$ in
$\Z^{r-1}$, so that
$$
H^j_{\fc}(M)_{(n_1, \ldots, n_r)} = 0 \quad \mbox{~whenever~} j \in
\nn \mbox{~and~} (n_1,\ldots, n_{r-1}) \geq
(\widehat{t},\ldots,\widehat{t}).
$$

We recall that $\fa \cap \fc = \fb$. There is an exact
Mayer--Vietoris sequence (in the category $\mbox{\rm
*}\mathcal{C}^{\Z^r}(R)$)
\[
\begin{picture}(260,95)(-150,-45)
\put(-155,40){\makebox(0,0){$ 0 $}} \put(-95,40){\makebox(0,0){$
H^0_{\fd}(M) $}} \put(0,40){\makebox(0,0){$ H^0_{\mathfrak{c}}(M)
\oplus H^0_{\fa}(M) $}} \put(95,40){\makebox(0,0){$ H^0_{\fb}(M) $}}
\put(-95,20){\makebox(0,0){$ H^1_{\fd}(M) $}}
\put(0,20){\makebox(0,0){$ H^1_{\mathfrak{c}}(M) \oplus H^1_{\fa}(M)
$}} \put(95,20){\makebox(0,0){$ H^1_{\fb}(M) $}}
\put(-95,0){\makebox(0,0){$ \cdots $}} \put(95,0){\makebox(0,0){$
\cdots $}} \put(-95,-20){\makebox(0,0){$ H^i_{\fd}(M) $}}
\put(0,-20){\makebox(0,0){$ H^i_{\mathfrak{c}}(M) \oplus
H^i_{\fa}(M) $}} \put(95,-20){\makebox(0,0){$ H^i_{\fb}(M) $}}
\put(-95,-40){\makebox(0,0){$ H^{i+1}_{\fd}(M) $}}
\put(0,-40){\makebox(0,0){$ \cdots\mbox{.} $}}
\put(-150,40){\vector(1,0){30}} \put(-73,40){\vector(1,0){30}}
\put(43,40){\vector(1,0){30}} \put(-150,20){\vector(1,0){30}}
\put(-73,20){\vector(1,0){30}} \put(43,20){\vector(1,0){30}}
\put(-150,0){\vector(1,0){30}} \put(-150,-20){\vector(1,0){30}}
\put(-73,-20){\vector(1,0){30}} \put(43,-20){\vector(1,0){30}}
\put(-150,-40){\vector(1,0){30}} \put(-70,-40){\vector(1,0){30}}
\end{picture}
\]
It now follows from this Mayer--Vietoris sequence that, if we set $t
:= \max\{{\textstyle
\frac{1}{2}}\widetilde{t},\widehat{t},\overline{t}\}$, then
$$
H^j_{\fb}(M)_{(n_1,\ldots, n_r)} = 0 \quad \mbox{~whenever~} j \in
\nn \mbox{~and~} (n_1,\ldots,n_r) \geq (t,\ldots,t).
$$
This completes the inductive step, and the proof.
\end{proof}

We can deduce from the above Theorem \ref{lc.3} a vanishing result
for multi-graded components of local cohomology modules with respect
to a multi-graded ideal that has both directions and non-directions.

\begin{cor}\label{va.4} Suppose that $R = \bigoplus_{\mathbf{n} \in \nn^r}
R_\mathbf{n}$ is positively graded and standard; let $M$ be a
finitely generated $\Z^r$-graded $R$-module. Let $\fb$ be an
$\nn^r$-graded ideal of $R$ that has some directions and some
non-directions: to be precise, and for ease of notation, suppose
that $\dir(\fb) = \{m+1,\ldots,r\}$, where $1 \leq m < r$. Then
there exists $t \in \Z$ such that, for all $j \in \nn$, and for all
$\mathbf{n} = (n_1, \ldots,n_r) \in \Z^r$ for which $(n_1,
\ldots,n_{m},n_{m+1}+ \cdots +n_r) \geq (t,\ldots,t)$ in $\Z^{m+1}$,
we have $H^j_{\fb}(M)_{\mathbf{n}}= 0$.
\end{cor}

\begin{note} As $\fb$ has some directions and $R$ is standard, it
follows from Lemma \ref{mu.2} that $R_+ \subseteq \fb$, so that
Theorem \ref{lc.3} yields a $t' \in \Z$ such that
$H^i_{\fb}(M)_{\mathbf{n}} = 0$ for all $\mathbf{n} \geq
(t',\ldots,t')$. Thus, when $m = r-1$, the conclusion of Corollary
\ref{va.4} already follows from Theorem \ref{lc.3}.
\end{note}

\begin{proof} Without loss of generality, we can, and do, assume
that $\fb = \sqrt{\fb}$.

Let $\phi : \Z^r \lra \Z^{m+1}$ be the group homomorphism defined by
$$\phi ((n_1,\ldots,n_r)) = (n_1,\ldots, n_{m},n_{m+1}+ \cdots +n_r) \quad \mbox{~for all~}
(n_1,\ldots,n_r) \in \Z^r.$$ Let $\mathbf{n}= (n_1,\ldots,n_r) \in
\nn^r$ be such that $\phi(\mathbf{n}) \geq \mathbf{1}$ in
$\Z^{m+1}$. Then $n_{m+1}+ \cdots +n_r \geq 1$, so that one of
$n_{m+1}, \ldots, n_r$ is positive. Now $R_{\mathbf{e}_i} \subseteq
\sqrt{\fb} = \fb$ for all $i = m+1, \ldots, r$, and since
$\mathbf{n} \geq \mathbf{e}_i$ for one of these $i$s, it follows
from Lemma \ref{mu.2} that $\fb \supseteq R_{\mathbf{n}}$. It
therefore follows that, in the $\nn^{m+1}$-graded ring $R^{\phi}$,
we have $\fb^{\phi} \supseteq \bigoplus_{\mathbf{m} \geq \mathbf{1}}
(R^{\phi})_\mathbf{m} = (R^{\phi})_+$.

We can now appeal to Theorem \ref{lc.3} to deduce that there exists
$t \in \Z$ such that $(H^j_{\fb^{\phi}}(M^{\phi}))_{\mathbf{h}} = 0$
for all $j \in \nn$ and all $\mathbf{h} \geq (t,\ldots,t)$ in
$\Z^{m+1}$. In view of Remark \ref{lc.1}, this means that
$((H^j_{\fb}(M))^{\phi})_{\mathbf{h}} = 0$ for all $j \in \nn$ and
all $\mathbf{h} \geq (t,\ldots,t)$ in $\Z^{m+1}$, so that
$$
H^j_{\fb}(M)_{(n_1, \ldots, n_r)} = 0 \quad \mbox{~whenever~} j \in
\nn \mbox{~and~} (n_1,\ldots, n_{m},n_{m+1}+ \cdots +n_r) \geq
(t,\ldots,t).
$$
\end{proof}

One of the reasons why we consider that Theorem \ref{lc.3} is of
interest in its own right concerns the structure of the
(multi-)graded components $H^i_{\fb}(M)_{\mathbf{n}}~(\mathbf{n} \in
\Z^r)$ as modules over $R_{\mathbf{0}}$ (the hypotheses and notation
here are as in Theorem \ref{lc.3}). The example in \cite[Exercise
15.1.7]{LC} shows that these graded components need not be finitely
generated $R_{\mathbf{0}}$-modules; however, it is always the case
that (for a finitely generated $\Z^r$-graded $R$-module $M$) the
(multi-)graded components $H^i_{R_+}(M)_{\mathbf{n}}~(\mathbf{n} \in
\Z^r)$ of the $i$th local cohomology module of $M$ with respect to
$R_+$ are finitely generated $R_{\mathbf{0}}$-modules (for all $i
\in \nn$), as we now show.

\begin{thm}
\label{va.5} Suppose that $R = \bigoplus_{\mathbf{n} \in \nn^r}
R_\mathbf{n}$ is positively graded; let $M$ be a finitely generated
$\Z^r$-graded $R$-module. Then $H^i_{R_+}(M)_{\mathbf{n}}$ is a
finitely generated $R_{\mathbf{0}}$-module, for all $i \in \nn$ and
all $\mathbf{n} \in \Z^r$.
\end{thm}

\begin{note} In the case where $r= 1$, this result is well known:
see \cite[Proposition 15.1.5]{LC}.
\end{note}

\begin{proof} We use induction on $i$. When $i = 0$, the claim is
immediate from the fact that $H^0_{R_+}(M)$ is isomorphic to a
submodule of $M$, and so is finitely generated. So suppose that $i
> 0$ and that the claim has been proved for smaller values of $i$,
for all finitely generated $\Z^r$-graded $R$-modules.

Recall that all the associated prime ideals of $M$ are
$\nn^r$-graded. Set $B(M) := \Ass_R(M) \setminus \mbox{\rm *}\Var
(R_+)$, and denote $\#B(M)$ by $b(M)$; we shall argue by induction
on $b(M)$. If $b(M) = 0$, then $M$ is $R_+$-torsion, so that
$H^i_{R_+}(M) = 0$ and the desired result is clear in this case.

Now suppose that $b(M) = 1$: let $\fp$ be the unique member of
$B(M)$. Set $\overline{M} := M/\Gamma_{R_+}(M)$. We can use the long
exact sequence of local cohomology modules induced by the exact
sequence $$0 \lra \Gamma_{R_+}(M) \lra M \lra \overline{M} \lra 0,$$
together with the fact that $H^j_{R_+}(\Gamma_{R_+}(M)) = 0$ for all
$j \in \N$, to see that, in order to complete the proof in this
case, it is sufficient for us to prove the result for
$\overline{M}$. Now $\overline{M}$ is $R_+$-torsion-free, and $\Ass
(\overline{M}) = \{\fp\}$. (See \cite[Exercise 2.1.12]{LC}.) There
exists a $\Z^r$-homogeneous element $a \in R_+ \setminus \fp$; note
that $a$ is a non-zero-divisor on $\overline{M}$. Let the degree of
$a$ be $\mathbf{v} = (v_1, \ldots,v_r)$, and note that $v_j>0$ for
all $j = 1, \ldots, r$. By Theorem \ref{lc.3}, there exists $t \in
\Z$ such that $H^j_{R_+}(\overline{M})_{\mathbf{n}} = 0$ for all
$\mathbf{n} \geq (t,t, \ldots,t)$.

Let $\mathbf{n} = (n_1, \ldots, n_r) \in \Z^r$. Since $v_j>0$ for
all $j = 1, \ldots, r$, there exists $w \in \N$ such that $n_j +
v_jw \geq t$ for all $j = 1, \ldots, r$. The exact sequence $$0 \lra
\overline{M} \stackrel{a^w}{\lra} \overline{M}(w\mathbf{v}) \lra
\left(\overline{M}/a^w\overline{M}\right)(w\mathbf{v}) \lra 0$$
induces an exact sequence of $R_{\mathbf{0}}$-modules
$$
H^{i-1}_{R_+}(\overline{M}/a^w\overline{M})_{\mathbf{n} +
w\mathbf{v}} \lra H^{i}_{R_+}(\overline{M})_{\mathbf{n}} \lra
H^{i}_{R_+}(\overline{M})_{\mathbf{n} + w\mathbf{v}},
$$
and since $w$ was chosen to ensure that the rightmost term in this
sequence is zero, it follows from the inductive hypothesis that
$H^{i}_{R_+}(\overline{M})_{\mathbf{n}}$ is a finitely generated
$R_{\mathbf{0}}$-module. This completes the proof in the case where
$b(M) = 1$.

Now suppose that $b(M) = b > 1$ and that it has been proved that all
the graded components of $H^i_{R_+}(L)$ are finitely generated
$R_{\mathbf{0}}$-modules for all choices of finitely generated
$\Z^r$-graded $R$-module $L$ with $b(L) < b$. Let $\fp, \fq \in
B(M)$ with $\fp \neq \fq$: suppose, for the sake of argument, that
$\fp \not\subseteq \fq$. Consider the $\fp$-torsion submodule
$\Gamma_{\fp}(M)$ of $M$. By \cite[Exercise 2.1.12]{LC},
$\Ass(\Gamma_{\fp}(M))$ and $\Ass(M/\Gamma_{\fp}(M))$ are disjoint
and $ \Ass M = \Ass(\Gamma_{\fp}(M)) \cup \Ass(M/\Gamma_{\fp}(M)). $
Now $\fp \in \Ass(\Gamma_{\fp}(M))$ and $\fq \not\in
\Ass(\Gamma_{\fp}(M))$; hence $b(\Gamma_{\fp}(M)) < b$ and
$b(M/\Gamma_{\fp}(M)) < b$. Therefore, by the inductive hypothesis,
both $H^i_{R_+}(\Gamma_{\fp}(M))_{\mathbf{n}}$ and
$H^i_{R_+}(M/\Gamma_{\fp}(M))_{\mathbf{n}}$ are finitely generated
$R_{\mathbf{0}}$-modules, for all $\mathbf{n} \in \Z^r$. We can now
use the long exact sequence of local cohomology modules (with
respect to $R_+$) induced from the exact sequence $0 \lra
\Gamma_{\fp}(M) \lra M \lra M/\Gamma_{\fp}(M) \lra 0$ to deduce that
$H^i_{R_+}(M)_{\mathbf{n}}$ is a finitely generated
$R_{\mathbf{0}}$-module for all $\mathbf{n} \in \Z^r$. The result
follows.
\end{proof}

\section{\sc A multi-graded analogue of Marley's work on finitely graded local cohomology modules}
\label{tm}

As was mentioned in the Introduction, the purpose of this section is
to obtain some multi-graded analogues of results about finitely
graded local cohomology modules that were proved, in the case where
$r = 1$, by Marley in \cite{Marle95}. We shall present a
multi-graded analogue of one of Marley's results and some extensions
of that analogue.

\begin{ntn}\label{tm.1} Throughout this section, we shall be concerned with the
situation where $R = \bigoplus_{\mathbf{n} \in \nn^r} R_\mathbf{n}$
is positively graded and standard, and we shall let $M
=\bigoplus_{\mathbf{n} \in \Z^r} M_\mathbf{n}$ be a $\Z^r$-graded
$R$-module. Also, $\fb$ will always denote an $\nn^r$-graded ideal
of $R$.

For $\mathbf{n} = (n_1, \ldots, n_r) \in \nn^r$, we shall denote
$\{i \in \{1, \ldots, r\}: n_i \neq 0\}$ by
$\mathcal{P}(\mathbf{n})$.
\end{ntn}

\begin{defi}\label{tm.2} An $r$-tuple $\mathbf{n} \in \Z^r$ is
called a {\em supporting degree of $M$\/} precisely when
$M_{\mathbf{n}} \neq 0$; we denote the set of all supporting degrees
of $M$ by $\mathcal{S}(M)$.
\end{defi}

Note that Theorem \ref{lc.3} imposes substantial restrictions on
$\mathcal{S}(H^i_{\fb}(M))$ when ($i \in \nn$ and) $\fb \supseteq
R_+$. The example below is included as motivation for the
introduction of some notation.

\begin{ex}
\label{mpbex} Let $k$ be an algebraically closed field and let $A =
k \oplus A_1 \oplus \cdots \oplus A_m \oplus \cdots$ and $B =
k\oplus B_1 \oplus \cdots \oplus B_n \oplus \cdots$ be two normal
Noetherian standard $\nn$-graded $k$-algebra domains with $w := \dim
A > 1$ and $v := \dim B > 1$. We consider the $\nn^2$-graded
$k$-algebra
$$ R := A \otimes_kB = \bigoplus_{(m,n) \in \nn^2} A_m\otimes_k
B_n.$$ Clearly $R = k[R_{(1,0)},R_{(0,1)}]$ is positively graded and
standard, and, as a finitely generated $k$-algebra, is Noetherian.
By \cite[Chapter III, \S 15, Theorem 40, Corollary 1]{ZS}, $R$ is
again an integral domain.  Observe that $R_+ = R_{(1,1)}R = A_+
\otimes_kB_+$. As $A$ and $B$ are normal and their dimensions exceed
$1$, we have $H^i_{A_+}(A) = H^i_{B_+}(B) = 0$ for $i = 0,1$. The
K\"unneth relations for tensor products (see \cite{Fumas99} or
\cite[Theorem 10.1]{Macla63}) yield, for each $i \in \nn$, an
isomorphism of $\Z^2$-graded $R$ modules
$$
H^i_{R_+}(R) \cong \left(A \otimes_k H^i_{B_+}(B)\right) \oplus
\left(H^i_{A_+}(A)\otimes_k B\right) \oplus \left(
\bigoplus_{\stackrel{\scriptstyle j,l \in \N\setminus\{1\}}{j+l =
i+1}}\left(H^j_{A_+}(A)\otimes_kH^l_{B_+}(B)\right)\right).
$$
As $\mathcal{S}(A) = \mathcal{S}(B) = \nn$, it follows that, for
each $i \in \nn$,
$$
\mathcal{S}(H^i_{R_+}(R)) = \left( \nn \times
\mathcal{S}(H^i_{B_+}(B))\right) \cup \left(
\mathcal{S}(H^i_{A_+}(A))\times \nn\right) \cup \left(
\bigcup_{\stackrel{\scriptstyle j,l \in \N\setminus\{1\}}{j+l =
i+1}}\left(\mathcal{S}(H^j_{A_+}(A))\times
\mathcal{S}(H^l_{B_+}(B))\right)\right).
$$
Observe, in particular, that $H^i_{R_+}(R) = 0$ for $i = 0,1$ and
for all $i \geq w + v$.

Appropriate choices for $A$ and $B$ yield many examples for $R$. We
shall just concentrate on a class of examples obtained by this
procedure when $A$ and $B$ are chosen in a particular way, which we
now describe. We can use \cite[Proposition (2.13)]{Brodm01}, in
conjunction with the Serre--Grothendieck correspondence (see
\cite[20.4.4]{LC}), to choose the algebra $A$ (as above) so that,
for a prescribed set $W \subseteq \{2, \ldots, w-1\}$, we have
$$
\mathcal{S}(H^i_{A_+}(A)) = \begin{cases} \emptyset & \text{for
all~} i \in \nn \setminus (W \cup\{w\}),\\ \{0\} & \text{for all~} i
\in W,\\ \{k \in \Z : k < 0\} & \text{for~} i = w. \end{cases}
$$
Similarly, for a prescribed set $V \subseteq \{2, \ldots, v-1\}$, we
choose $B$ (as above) so that
$$
\mathcal{S}(H^i_{B_+}(B)) = \begin{cases} \emptyset & \text{for
all~} i \in \nn \setminus (V \cup\{v\}),\\ \{0\} & \text{for all~} i
\in V,\\ \{k \in \Z : k < 0\} & \text{for~} i = v. \end{cases}
$$
With such a choice of $A$ for $w = 5$ and $W = \{2\}$, and such a
choice of $B$ for $v = 5$ and $V = \{3\}$, the sets of supporting
degrees $\mathcal{S}(H^i_{R_+}(R))$ for $i = 2,3,4,5$ are as in
Figure 1 below. \thinlines
\begin{center}

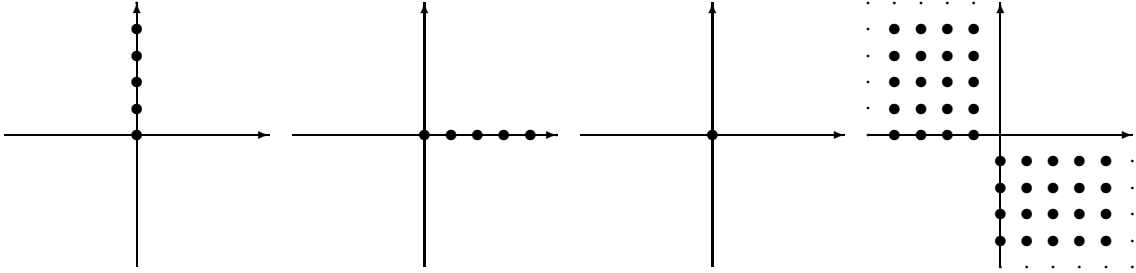
\begin{figure}[h]\begin{picture}(100,100)(-50,-50) \put(-50,0){\vector(1,0){100}}
\put(0,-50){\vector(0,1){100}}
\put(0,0){\circle*{4}}\put(0,10){\circle*{4}}\put(0,20){\circle*{4}}
\put(0,30){\circle*{4}}\put(0,40){\circle*{4}}\put(0,50){\circle*{1}}
\end{picture}\;\;
%\caption{$\mathcal{S}(H^2_{R_+}(R))$}\end{figure}
%\begin{figure}
\begin{picture}(100,100)(-50,-50)
\put(-50,0){\vector(1,0){100}} \put(0,-50){\vector(0,1){100}}
\put(0,0){\circle*{4}}\put(10,0){\circle*{4}}\put(20,0){\circle*{4}}
\put(30,0){\circle*{4}}\put(40,0){\circle*{4}}\put(50,0){\circle*{1}}
\end{picture}\;\;
%\caption{$\mathcal{S}(H^3_{R_+}(R))$}\end{figure}
%\begin{figure}
\begin{picture}(100,100)(-50,-50)
\put(-50,0){\vector(1,0){100}} \put(0,-50){\vector(0,1){100}}
\put(0,0){\circle*{4}}
\end{picture}\;\;
%\caption{$\mathcal{S}(H^4_{R_+}(R))$}\end{figure}
%\begin{figure}
\begin{picture}(100,100)(-50,-50)
\put(-50,0){\vector(1,0){100}} \put(0,-50){\vector(0,1){100}}
\put(-10,0){\circle*{4}}\put(-10,10){\circle*{4}}\put(-10,20){\circle*{4}}
\put(-10,30){\circle*{4}}\put(-10,40){\circle*{4}}\put(-10,50){\circle*{1}}
\put(-20,0){\circle*{4}}\put(-20,10){\circle*{4}}\put(-20,20){\circle*{4}}
\put(-20,30){\circle*{4}}\put(-20,40){\circle*{4}}\put(-20,50){\circle*{1}}
\put(-30,0){\circle*{4}}\put(-30,10){\circle*{4}}\put(-30,20){\circle*{4}}
\put(-30,30){\circle*{4}}\put(-30,40){\circle*{4}}\put(-30,50){\circle*{1}}
\put(-40,0){\circle*{4}}\put(-40,10){\circle*{4}}\put(-40,20){\circle*{4}}
\put(-40,30){\circle*{4}}\put(-40,40){\circle*{4}}\put(-40,50){\circle*{1}}
\put(-50,0){\circle*{1}}\put(-50,10){\circle*{1}}\put(-50,20){\circle*{1}}
\put(-50,30){\circle*{1}}\put(-50,40){\circle*{1}}\put(-50,50){\circle*{1}}
\put(0,-10){\circle*{4}}\put(10,-10){\circle*{4}}\put(20,-10){\circle*{4}}
\put(30,-10){\circle*{4}}\put(40,-10){\circle*{4}}\put(50,-10){\circle*{1}}
\put(0,-20){\circle*{4}}\put(10,-20){\circle*{4}}\put(20,-20){\circle*{4}}
\put(30,-20){\circle*{4}}\put(40,-20){\circle*{4}}\put(50,-20){\circle*{1}}
\put(0,-30){\circle*{4}}\put(10,-30){\circle*{4}}\put(20,-30){\circle*{4}}
\put(30,-30){\circle*{4}}\put(40,-30){\circle*{4}}\put(50,-30){\circle*{1}}
\put(0,-40){\circle*{4}}\put(10,-40){\circle*{4}}\put(20,-40){\circle*{4}}
\put(30,-40){\circle*{4}}\put(40,-40){\circle*{4}}\put(50,-40){\circle*{1}}
\put(0,-50){\circle*{1}}\put(10,-50){\circle*{1}}\put(20,-50){\circle*{1}}
\put(30,-50){\circle*{1}}\put(40,-50){\circle*{1}}\put(50,-50){\circle*{1}}
\end{picture}\caption{$\mathcal{S}(H^i_{R_+}(R))$ for $i = 2,3,4,5$ respectively}\end{figure}
\end{center}

In view of Theorem \ref{lc.3}, the supporting set
$\mathcal{S}(H^5_{R_+}(R))$ seems unremarkable. The local cohomology
module $H^4_{R_+}(R)$ is finitely graded. Although neither
$H^3_{R_+}(R)$ nor $H^2_{R_+}(R)$ is finitely graded, both have sets
of supporting degrees that are quite restricted.
\end{ex}

We now return to the general situation described in Notation
\ref{tm.1}. In the case where $r=1$, one way of recording that a
local cohomology module $H^i_{\fb}(M)$ is finitely graded is to
state that there exist $s, t \in \Z$ with $s < t$ such that
$$
\mathcal{S}(H^i_{\fb}(M)) = \left\{n \in \Z: H^i_{\fb}(M)_n \neq 0
\right\} \subseteq \left\{n \in \Z : s \leq n < t\right\}.
$$
One might expect the natural generalization to our multi-graded
situation to involve conditions such as
$$
\mathcal{S}(H^i_{\fb}(M)) = \left\{\mathbf{n} \in \Z^r:
H^i_{\fb}(M)_{\mathbf{n}} \neq 0 \right\} \subseteq
\left\{\mathbf{n}=(n_1, \ldots, n_r) \in \Z^r : s_i \leq n_i < t_i
\mbox{~for all~} i = 1, \ldots,r\right\},
$$
where $\mathbf{s} = (s_1, \ldots,s_r)$, $\mathbf{t} = (t_1,
\ldots,t_r) \in \Z^r$ satisfy $\mathbf{s} \leq \mathbf{t}$. However,
in the light of evidence like that provided by Example \ref{mpbex}
above, and other examples, we introduce the following.

\begin{ntn}\label{tm.5A} Let $\mathbf{s} = (s_1, \ldots,s_r)$, $\mathbf{t} = (t_1,
\ldots,t_r) \in \Z^r$ with $\mathbf{s} \leq \mathbf{t}$. We set
$$
\X(\mathbf{s},\mathbf{t}) := \left\{\mathbf{n}=(n_1, \ldots, n_r)
\in \Z^r : \mbox{~there exists~} i \in \{1, \ldots, r\} \mbox{~such
that~} s_i \leq n_i < t_i \right\}.
$$
\end{ntn}

\begin{ex}\label{tm.5Aex} Figure 2 below illustrates, in the case
where $r = 2$, the set $\X((-2,1),(0,2))$.
\begin{center}
\begin{figure}[h]\begin{picture}(100,100)(-50,-50) \put(-50,0){\vector(1,0){100}}
\put(0,-50){\vector(0,1){100}}
\put(-20,0){\circle*{4}}\put(-20,10){\circle*{4}}\put(-20,20){\circle*{4}}
\put(-20,30){\circle*{4}}\put(-20,40){\circle*{4}}\put(-20,50){\circle*{1}}
\put(-20,-10){\circle*{4}}\put(-20,-20){\circle*{4}}
\put(-20,-30){\circle*{4}}\put(-20,-40){\circle*{4}}\put(-20,-50){\circle*{1}}
\put(-10,0){\circle*{4}}\put(-10,10){\circle*{4}}\put(-10,20){\circle*{4}}
\put(-10,30){\circle*{4}}\put(-10,40){\circle*{4}}\put(-10,50){\circle*{1}}
\put(-10,-10){\circle*{4}}\put(-10,-20){\circle*{4}}
\put(-10,-30){\circle*{4}}\put(-10,-40){\circle*{4}}\put(-10,-50){\circle*{1}}
\put(0,10){\circle*{4}}\put(10,10){\circle*{4}}\put(20,10){\circle*{4}}
\put(30,10){\circle*{4}}\put(40,10){\circle*{4}}\put(50,10){\circle*{1}}
\put(-10,10){\circle*{4}}\put(-20,10){\circle*{4}}
\put(-30,10){\circle*{4}}\put(-40,10){\circle*{4}}\put(-50,10){\circle*{1}}
\end{picture}
\caption{The set $\X((-2,1),(0,2))$ in $\Z^2$}
\end{figure}
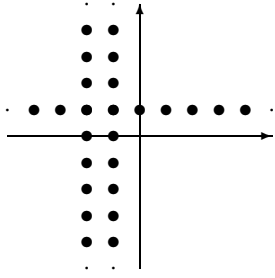
\end{center}
\end{ex}

\begin{rmk}\label{tm.5BCDE} Let $\mathbf{s},
\mathbf{s}', \mathbf{s}'', \mathbf{t}, \mathbf{t}', \mathbf{t}''\in
\Z^r$ with $\mathbf{s} \leq \mathbf{t}$, $\mathbf{s}' \leq
\mathbf{t}'$ and $\mathbf{s}'' \leq \mathbf{t}''$. Let $\mathbf{m}
\in \nn^r \setminus\{\mathbf{0}\}$.

\begin{enumerate}
\item Clearly $(\mathbf{s} + \nn^r) \setminus (\mathbf{t} +
\nn^r)\subseteq \X(\mathbf{s},\mathbf{t})$.
\item  Suppose
that $\mathcal{P}(\mathbf{t} - \mathbf{s}) \subseteq
\mathcal{P}(\mathbf{m})$. Let $\mathbf{w} \in \Z^r$ be such that
$\mathcal{P}(\mathbf{w}) \subseteq \{1, \ldots,r\} \setminus
\mathcal{P}(\mathbf{m})$. Then
$$
\X(\mathbf{s} + \mathbf{w},\mathbf{t} + \mathbf{w}) =
\X(\mathbf{s},\mathbf{t}) = \left\{\mathbf{n} \in\Z^r : \mbox{~there
exists~}i \in \mathcal{P}(\mathbf{m}) \mbox{~such that~} s_i \leq
n_i <t_i\right\}.
$$
\item Clearly
$ \X(\mathbf{s}',\mathbf{t}') \cup \X(\mathbf{s}'',\mathbf{t}'')
\subseteq
\X(\min\{\mathbf{s}',\mathbf{s}''\},\max\{\mathbf{t}',\mathbf{t}''\}).
$
\item Assume that $\mathcal{P}(\mathbf{t}' - \mathbf{s}') \subseteq
\mathcal{P}(\mathbf{m})$ and $\mathcal{P}(\mathbf{t}'' -
\mathbf{s}'') \subseteq \mathcal{P}(\mathbf{m})$. For each $i \in
\{1, \ldots, r\}$, set
$$
\widetilde{s}_i := \min\{s_i',s_i''\}\quad \mbox{and} \quad
\widetilde{t}_i := \begin{cases} \max\{t_i',t_i''\} & \text{if } i
\in \mathcal{P}(\mathbf{m}),\\\widetilde{s}_i & \text{if } i \in
\{1, \ldots,r\} \setminus \mathcal{P}(\mathbf{m}). \end{cases}
$$
Set $\widetilde{\mathbf{s}} := (\widetilde{s}_1,
\ldots,\widetilde{s}_r)$ and $\widetilde{\mathbf{t}} :=
(\widetilde{t}_1, \ldots,\widetilde{t}_r)$. Then
$$
\widetilde{\mathbf{s}} \leq \widetilde{\mathbf{t}}, \quad
\mathcal{P}(\widetilde{\mathbf{t}} - \widetilde{\mathbf{s}})
\subseteq \mathcal{P}(\mathbf{m}) \quad \mbox{and}\quad
\X(\mathbf{s}',\mathbf{t}') \cup \X(\mathbf{s}'',\mathbf{t}'')
\subseteq \X(\widetilde{\mathbf{s}} ,\widetilde{\mathbf{t}}).
$$
\end{enumerate}
\end{rmk}

The next lemma provides a small hint about the importance of the
sets $\X(\mathbf{s},\mathbf{t})$ of Notation \ref{tm.5A} for our
work.

\begin{lem}\label{tm.4} Let $\mathbf{m} \in \nn^r \setminus
\{\mathbf{0}\}$. Assume that $M$ is finitely generated and that
$R_{\mathbf{m}} \subseteq \sqrt{(0:_RM)}$. Then there exist
$\mathbf{s}, \mathbf{t} \in \Z^r$ such that $\mathbf{s}\leq
\mathbf{t}$, $\mathcal{P}(\mathbf{t}-\mathbf{s}) \subseteq
\mathcal{P}(\mathbf{m})$ and $\mathcal{S}(M) \subseteq (\mathbf{s} +
\nn^r) \setminus (\mathbf{t} + \nn^r)$, so that $\mathcal{S}(M)
\subseteq \X(\mathbf{s},\mathbf{t})$ in view of Remark\/ {\rm
\ref{tm.5BCDE}(i)}.
\end{lem}

\begin{proof} As $M$ is finitely generated, there exist $\mathbf{s}, \mathbf{w} \in \Z^r$
such that $\mathbf{s}\leq \mathbf{w}$ and $M = \sum_{\mathbf{s} \leq
\mathbf{n} \leq \mathbf{w}} RM_{\mathbf{n}}$. In particular,
$\mathcal {S}(M) \subseteq \mathbf{s} + \nn^r$.

Moreover, there exists $u \in \N$ such that $(R_{\mathbf{m}})^u
\subseteq (0:_RM)$; since $R$ is standard, $(R_{\mathbf{m}})^u =
R_{u\mathbf{m}}$; hence $R_{u\mathbf{m}}M_{\mathbf{n}} = 0$ for all
$\mathbf{n} \in \Z^r$.

Let $\mathbf{t} = \mathbf{s} + \sum_{i \in \mathcal{P}(\mathbf{m})}
(w_i - s_i + um_i)\mathbf{e}_i$. Now, let $\mathbf{h} = (h_1,
\ldots, h_r) \in \mathbf{t} + \nn^r$. Our proof will be complete
once we have shown that $M_{\mathbf{h}} = 0$. For each $i \in
\mathcal{P}(\mathbf{m})$, we have $h_i \geq t_i = w_i + um_i$.
Moreover,
$$
M_{\mathbf{h}} = \sum_{\mathbf{n} \in
\mathcal{T}}R_{\mathbf{h}-\mathbf{n}}M_{\mathbf{n}},\quad
\mbox{where~} \mathcal{T} = \left\{ \mathbf{n} \in\Z^r:\mathbf{s}
\leq \mathbf{n} \leq \mathbf{w},~ \mathbf{n} \leq
\mathbf{h}\right\}.
$$
Let $\mathbf{n} = (n_1, \ldots, n_r) \in \mathcal{T}$. If $i \in
\mathcal{P}(\mathbf{m})$, then $n_i + um_i \leq w_i + um_i \leq
h_i$; if $i \in \{1, \ldots, r\} \setminus \mathcal{P}(\mathbf{m})$,
then $n_i + um_i = n_i \leq h_i$. Consequently $\mathbf{n} +
u\mathbf{m} \leq \mathbf{h}$ . Therefore $u\mathbf{m} \leq
\mathbf{h} - \mathbf{n}$ for all $\mathbf{n} \in \mathcal{T}$, and
hence
$$
M_{\mathbf{h}} = \sum_{\mathbf{n} \in
\mathcal{T}}R_{\mathbf{h}-\mathbf{n}}M_{\mathbf{n}} =
\sum_{\mathbf{n} \in
\mathcal{T}}R_{\mathbf{h}-\mathbf{n}-u\mathbf{m}}R_{u\mathbf{m}}M_{\mathbf{n}}
= 0.
$$
\end{proof}

\begin{defi}\label{tm.6def} Let $\mathcal{Q} \subseteq \{1, \ldots,
r\}$. By a {\em $\mathcal{Q}$-domain in $\Z^r$\/} we mean a set of
the form
$$
\X(\mathbf{s},\mathbf{t}) \quad \mbox{with ~} \mathbf{s},\mathbf{t}
\in \Z^r,~\mathbf{s} \leq \mathbf{t} \mbox{~and~}
\mathcal{P}(\mathbf{t} - \mathbf{s}) \subseteq \mathcal{Q}.
$$
\end{defi}

\begin{rmks}\label{tm.6}
The following statements are immediate from the definition.
\begin{enumerate}
\item A $\emptyset$-domain in $\Z^r$ is empty.
\item If $\mathcal{Q} \subseteq \mathcal{Q}' \subseteq \{1, \ldots,
r\}$ and if $\X$ is a $\mathcal{Q}$-domain in $\Z^r$, then $\X$ is a
$\mathcal{Q}'$-domain in $\Z^r$.
\item If $\X$ is a $\mathcal{Q}$-domain in $\Z^r$ and $\mathbf{w} \in \Z^r$,
then $\mathbf{w} + \X := \left\{ \mathbf{w} + \mathbf{n} :
\mathbf{n} \in \X \right\}$ is a $\mathcal{Q}$-domain in $\Z^r$.
\item If $\mathbf{s},\mathbf{t}
\in \Z^r$ with $\mathbf{s} \leq \mathbf{t}$ and
$\mathcal{P}(\mathbf{t} - \mathbf{s}) \subseteq \mathcal{Q}$, then
$(\mathbf{s} + \nn^r) \setminus (\mathbf{t} + \nn^r)$ is contained
in a $\mathcal{Q}$-domain in $\Z^r$, by Remark \ref{tm.5BCDE}(i).
\item If $\X$ is a $\mathcal{Q}$-domain in $\Z^r$ and $\mathbf{w} \in \Z^r$ is
such that $\mathcal{P}(\mathbf{w}) \cap \mathcal{Q} = \emptyset$,
then $\X = \mathbf{w} + \X$, by Remark \ref{tm.5BCDE}(ii).
\item By Remark \ref{tm.5BCDE}(iv), the union of finitely many
$\mathcal{Q}$-domains in $\Z^r$ is contained in a
$\mathcal{Q}$-domain in $\Z^r$.
\end{enumerate}
\end{rmks}

\begin{lem}\label{tm.7} Let $\mathbf{m}, \mathbf{k} \in \nn^r
\setminus \{\mathbf{0}\}$, and let $T$ be a $\Z^r$-graded $R$-module
such that $R_{\mathbf{m}}T = 0$. Let $y \in R_{\mathbf{k}}$, and let
$K$ denote the kernel of the homogeneous $R$-homomorphism $T \lra
T(\mathbf{k})$ given by multiplication by $y$.

\begin{enumerate}
\item If $\mathcal{P}(\mathbf{m}) \subseteq
\mathcal{P}(\mathbf{k})$, then there exists $v \in \nn$ such that
$\mathcal{S}(T) \subseteq \bigcup_{j=0}^v\left(\mathcal{S}(K) -
j\mathbf{k}\right)$.
\item If $\mathcal{P}(\mathbf{m}) \not\subseteq
\mathcal{P}(\mathbf{k})$, if multiplication by $y$ provides an
isomorphism $T \stackrel{\cong}{\lra} T(\mathbf{k})$, and if $T$
considered as an $R_y$-module is finitely generated, then
$\mathcal{S}(T)$ is contained in a $(\mathcal{P}(\mathbf{m})
\setminus \mathcal{P}(\mathbf{k}))$-domain in $\Z^r$.
\end{enumerate}
\end{lem}

\begin{proof} Write $\mathbf{m} = (m_1, \ldots, m_r)$ and $\mathbf{k} = (k_1, \ldots,
k_r)$. Let $u \in \N$ be such that $m_i \leq uk_i$ for all $i \in
\mathcal{P}(\mathbf{k})$. Set $\mathbf{h} := \sum_{i \in \{1,
\ldots,r\}\setminus \mathcal{P}(\mathbf{k})}m_i\mathbf{e}_i$. Then,
if $i \in \mathcal{P}(\mathbf{k})$, we have $(u\mathbf{k} +
\mathbf{h})_i = uk_i \geq m_i$, whereas, if $i \in \{1,\ldots,r\}
\setminus  \mathcal{P}(\mathbf{k})$, we have $(u\mathbf{k} +
\mathbf{h})_i = uk_i +m_i\geq m_i$. Therefore $\mathbf{m} \leq
u\mathbf{k} + \mathbf{h}$.

Now, let $z \in R_{\mathbf{h}}$. Then, because $R$ is standard,
$y^uz \in R_{u\mathbf{k} + \mathbf{h}} = R_{u\mathbf{k} +
\mathbf{h}- \mathbf{m}}R_{\mathbf{m}}$. As $R_{\mathbf{m}}T = 0$, it
follows that $y^uzT = 0$. Therefore $y^uR_{\mathbf{h}}T = 0$.

(i) Assume that $\mathcal{P}(\mathbf{m}) \subseteq
\mathcal{P}(\mathbf{k})$. Then $\mathcal{P}(\mathbf{h}) =
\mathcal{P}(\mathbf{m}) \setminus \mathcal{P}(\mathbf{k}) =
\emptyset$, so that $\mathbf{h} = \mathbf{0}$. Hence $y^uT =
y^uR_{\mathbf{0}}T = 0$.

Now let $\K := \bigcup_{j=0}^{u-1}(\mathcal{S}(K) - j\mathbf{k})$,
and let $\mathbf{n} \in \Z^r \setminus \K$. If we show that
$T_{\mathbf{n}} = 0$, then we shall have proved part (i). Now
$\mathbf{n} + j\mathbf{k} \not\in \mathcal{S}(K)$ for all $j \in
\{0,\ldots, u-1\}$, and so the $R_{\mathbf{0}}$-homomorphism $y^u :
T_{\mathbf{n}} \lra T_{\mathbf{n}+u\mathbf{k}}$, which is the
composition of the $R_{\mathbf{0}}$-homomorphisms $y :
T_{\mathbf{n}+j\mathbf{k}} \lra T_{\mathbf{n}+(j+1)\mathbf{k}}$ for
$j = 0, \ldots, u-1$, is injective. But $y^uT_{\mathbf{n}} = 0$, and
so $T_{\mathbf{n}} = 0$.

(ii) Now assume that $\mathcal{P}(\mathbf{m}) \not\subseteq
\mathcal{P}(\mathbf{k})$, that multiplication by $y$ provides an
isomorphism $T \stackrel{\cong}{\lra} T(\mathbf{k})$, and that $T$
considered as an $R_y$-module is finitely generated. As
$y^uR_{\mathbf{h}}T = 0$, it follows that $R_{\mathbf{h}}T = 0$.

As $T$ is finitely generated over $R_y$, there are finitely many
$r$-tuples $\mathbf{g}^{(1)}, \ldots, \mathbf{g}^{(q)} \in \Z^r$
such that $T = \sum_{j=1}^q R_yT_{\mathbf{g}^{(j)}}$. Now, for $i
\in \{1, \ldots, r\}$, set
$$
s_i := \begin{cases} 0 & \text{if } i
\not\in\mathcal{P}(\mathbf{h}),\\ \min\{g^{(j)}_i : j = 1, \ldots,
q\}& \text{if } i \in\mathcal{P}(\mathbf{h}), \end{cases} \quad
\quad t_i := \begin{cases} 0 & \text{if } i
\not\in\mathcal{P}(\mathbf{h}),\\ \max\{g^{(j)}_i : j = 1, \ldots,
q\} + h_i& \text{if } i \in\mathcal{P}(\mathbf{h}),
\end{cases}
$$
and put $\mathbf{s} = (s_1, \ldots, s_r)$, $\mathbf{t} = (t_1,
\ldots, t_r)$. Then $\mathbf{s} \leq \mathbf{t}$ and
$\mathcal{P}(\mathbf{t}-\mathbf{s}) = \mathcal{P}(\mathbf{h}) =
\mathcal{P}(\mathbf{m}) \setminus \mathcal{P}(\mathbf{k})$.  Let
$\mathbf{n} \in \Z^r \setminus \X(\mathbf{s},\mathbf{t})$. If we
show that $T_{\mathbf{n}} = 0$, then we shall have proved part (ii).
Let $\alpha \in T_{\mathbf{n}}$. There exist integers $v_1, \ldots,
v_q$ such that $\alpha \in \sum_{j=1}^qy^{v_j}R_{\mathbf{n} -
v_j\mathbf{k} - \mathbf{g}^{(j)}}T_{\mathbf{g}^{(j)}}$.

Note that, for each $i \in \mathcal{P}(\mathbf{h}) =
\mathcal{P}(\mathbf{m}) \setminus \mathcal{P}(\mathbf{k})$, we have
either $n_i < s_i$ or $t_i \leq n_i$ (because $\mathbf{n} \not\in
\X(\mathbf{s},\mathbf{t})$).

Assume first that there is some $i \in \mathcal{P}(\mathbf{h})$ with
$n_i < s_i$. As $i \not\in \mathcal{P}(\mathbf{k})$, it follows that
$$
(\mathbf{n} - v_j\mathbf{k} - \mathbf{g}^{(j)})_i = n_i -v_jk_i -
g^{(j)}_i = n_i - g^{(j)}_i < s_i - g^{(j)}_i \leq 0,
$$
for all $j \in \{1, \ldots, q\}$, so that $R_{\mathbf{n} -
v_j\mathbf{k} - \mathbf{g}^{(j)}} = 0$ and $\alpha = 0$.

Therefore, we can, and do, assume that $t_i \leq n_i$ for all $i \in
\mathcal{P}(\mathbf{h})$. In this case, for each $i \in
\mathcal{P}(\mathbf{h})$ and each $j \in \{1, \ldots, q\}$, we have
$$
(\mathbf{n} - v_j\mathbf{k} - \mathbf{g}^{(j)})_i = n_i -v_jk_i -
g^{(j)}_i = n_i - g^{(j)}_i \geq t_i - g^{(j)}_i \geq h_i.
$$
Therefore, for each $j \in \{1, \ldots, q\}$, either $\mathbf{n} -
v_j\mathbf{k} - \mathbf{g}^{(j)} \geq \mathbf{h}$, or $\mathbf{n} -
v_j\mathbf{k} - \mathbf{g}^{(j)}$ has a negative component and
$R_{\mathbf{n} - v_j\mathbf{k} - \mathbf{g}^{(j)}} = 0$. This means
that
$$
\alpha \in \sum_{j=1}^qy^{v_j}R_{\mathbf{n} - v_j\mathbf{k} -
\mathbf{g}^{(j)}}T_{\mathbf{g}^{(j)}} = \sum_{\stackrel{\scriptstyle
j=1}{\mathbf{n} - v_j\mathbf{k} - \mathbf{g}^{(j)}\geq
\mathbf{0}}}^qy^{v_j}R_{\mathbf{n} - v_j\mathbf{k} -
\mathbf{g}^{(j)}-\mathbf{h}}R_{\mathbf{h}}T_{\mathbf{g}^{(j)}} = 0.
$$
It follows that $T_{\mathbf{n}} = 0$, as required.
\end{proof}

\begin{lem}\label{tm.8} Let $\mathbf{m} \in \nn^r
\setminus \{\mathbf{0}\}$ and $\mathbf{k} \in \nn^r$. Assume that
$M$ is finitely generated and that $R_{\mathbf{m}} \subseteq
\sqrt{(0:_RM)}$. Let $y \in R_{\mathbf{k}}$. Then there exists a
$(\mathcal{P}(\mathbf{m}) \setminus \mathcal{P}(\mathbf{k}))$-domain
$\X$ in $\Z^r$ such that $\mathcal{S}(H^1_{yR}(M)) \subseteq \X$.
\end{lem}

\begin{proof} Assume first that $\mathbf{k} = \mathbf{0}$. Then
$\mathcal{P}(\mathbf{k}) = \emptyset$ and, by the multi-graded
analogue of \cite[Lemma 13.1.10]{LC}, there are
$R_{\mathbf{0}}$-isomorphisms $H^1_{yR}(M)_{\mathbf{n}} \cong
H^1_{yR_{\mathbf{0}}}(M_{\mathbf{n}})$ for all $\mathbf{n} \in
\Z^r$. Therefore $\mathcal{S}(H^1_{yR}(M)) \subseteq
\mathcal{S}(M)$, and the claim follows in this case from Lemma
\ref{tm.4}.

We now deal with the remaining case, where $\mathbf{k} \neq
\mathbf{0}$. Since (by the multi-graded analogue of
\cite[12.4.2]{LC}) there is a $\Z^r$-homogeneous epimorphism of
$\Z^r$-graded $R$-modules $D_{yR}(M) \lra H^1_{yR}(M)$, it suffices
for us to show that $\mathcal{S}(D_{yR}(M))$ is contained in a
$(\mathcal{P}(\mathbf{m}) \setminus \mathcal{P}(\mathbf{k}))$-domain
in $\Z^r$.

Recall that there is a homogeneous isomorphism $D_{yR}(M) \cong
M_y$, and so the multiplication map $y : D_{yR}(M) \lra
D_{yR}(M)(\mathbf{k})$ is an isomorphism, and $D_{yR}(M)$ is
finitely generated as an $R_y$-module. Since $R_{\mathbf{m}}
\subseteq \sqrt{(0:_RM)}$, there exists $u \in \N$ such that
$R_{u\mathbf{m}}M = 0$, so that $R_{u\mathbf{m}}M_y = 0$ and
$R_{u\mathbf{m}}D_{yR}(M) = 0$. Observe that
$\mathcal{P}(u\mathbf{m}) = \mathcal{P}(\mathbf{m})$. We now apply
Lemma \ref{tm.7}, with $D_{yR}(M)$ as the module $T$ and
$u\mathbf{m}$ in the r\^ole of $\mathbf{m}$: if
$\mathcal{P}(u\mathbf{m}) = \mathcal{P}(\mathbf{m}) \subseteq
\mathcal{P}(\mathbf{k})$, then part (i) of Lemma \ref{tm.7} yields
that $\mathcal{S}(D_{yR}(M)) = \emptyset$, while if
$\mathcal{P}(u\mathbf{m}) = \mathcal{P}(\mathbf{m}) \not\subseteq
\mathcal{P}(\mathbf{k})$, then it follows from part (ii) of Lemma
\ref{tm.7} that $\mathcal{S}(D_{yR}(M))$ is contained in a
$(\mathcal{P}(\mathbf{m}) \setminus \mathcal{P}(\mathbf{k}))$-domain
in $\Z^r$.
\end{proof}

\begin{lem}\label{tm.9} Let $\mathbf{m} \in \nn^r
\setminus \{\mathbf{0}\}$. Assume that $M$ is finitely generated and
that $R_{\mathbf{m}} \subseteq \sqrt{(0:_RM)}$. Then there exists a
$\mathcal{P}(\mathbf{m})$-domain $\X$ in $\Z^r$ such that
$\mathcal{S}(H^i_{\fb}(M)) \subseteq \X$ for all $i \in \nn$.
\end{lem}

\begin{proof} Since $H^i_{\fb}(M) = 0$ for all $i > \ara (\fb)$, it
follows from Remark \ref{tm.6}(vi) that it is sufficient for us to
show that, for each $i \in \nn$, there exists a
$\mathcal{P}(\mathbf{m})$-domain $\X_i$ in $\Z^r$ such that
$\mathcal{S}(H^i_{\fb}(M)) \subseteq \X_i$. For $i = 0$, this is
immediate from Lemma \ref{tm.4}.

Let $y_1, \ldots, y_s$ be $\nn^r$-homogeneous elements of $R$ that
generate $\fb$. We argue by induction on $s$. When $s = 1$ and $i =
1$, the desired result follows from Lemma \ref{tm.8}; as we have
already dealt, in the preceding paragraph, with the case where $i =
0$, and as $H^i_{y_1R}(M) = 0$ for all $i > 1$, we have established
the desired result in all cases when $s = 1$.

So suppose now that $s > 1$ and that the desired result has been
proved in all cases where $\fb$ can be generated by fewer than $s$
$\nn^r$-homogeneous elements. Again, we have already dealt with the
case where $i = 0$. For $i \in \N$, there is an exact
Mayer--Vietoris sequence (in the category $\mbox{\rm
*}\mathcal{C}^{\Z^r}(R)$)
$$
\cdots \lra H^{i-1}_{(y_1y_s,\ldots,y_{s-1}y_s)R}(M) \lra
H^i_{\fb}(M) \lra H^{i}_{(y_1,\ldots,y_{s-1})R}(M)\oplus
H^i_{y_sR}(M) \lra \cdots.
$$
By the inductive hypothesis, there exist
$\mathcal{P}(\mathbf{m})$-domains $\X_i', \X_i'', \X_i'''$ in $\Z^r$
such that
$$
\mathcal{S}(H^{i-1}_{(y_1y_s,\ldots,y_{s-1}y_s)R}(M)) \subseteq
\X_i', \quad \mathcal{S}(H^{i}_{(y_1,\ldots,y_{s-1})R}(M)) \subseteq
\X_i'' \quad \mbox{and} \quad \mathcal{S}(H^{i}_{ y_{s} R}(M))
\subseteq \X_i'''.
$$
Therefore $\mathcal{S}(H^i_{\fb}(M)) \subseteq \X_i'\cup
\X_i''\cup\X_i'''$, and so the desired result follows from Remark
\ref{tm.6}(vi).
\end{proof}

\begin{lem}\label{tm.10} Let $\mathbf{m} \in \nn^r
\setminus \{\mathbf{0}\}$. Let $\fp_1, \ldots,\fp_n$ be prime ideals
of $R$ such that $R_{\mathbf{m}} \not\subseteq \fp_i$ for each $i=
1, \ldots, n$. Then there exists $u \in \N$ such that
$R_{u\mathbf{m}} \not\subseteq \bigcup_{i=1}^n\fp_i$
\end{lem}

\begin{proof} Consider the (Noetherian) $\nn$-graded ring
$R_{\mathbf{0}}[R_{\mathbf{m}}] =
\bigoplus_{j\in\nn}R_{j\mathbf{m}}$ (in which  $R_{j\mathbf{m}}$ is
the component of degree $j$, for all $j \in \nn$). Apply the
ordinary Homogeneous Prime Avoidance Lemma (see \cite[Lemma
15.1.2]{LC}) to the graded ideal
$R_{\mathbf{m}}R_{\mathbf{0}}[R_{\mathbf{m}}] =
\bigoplus_{j\in\N}R_{j\mathbf{m}}$ and the prime ideals $\fp_i \cap
R_{\mathbf{0}}[R_{\mathbf{m}}]~(i=1, \ldots, n)$.
\end{proof}

\begin{lem}\label{tm.11} Let $\mathbf{m} \in \nn^r
\setminus \{\mathbf{0}\}$ and let $\X$ be a
$\mathcal{P}(\mathbf{m})$-domain in $\Z^r$. Then there exists $u \in
\N$ such that, for each $\mathbf{w} \in \Z^r$, there is some $j \in
\{0, \ldots, \#\mathcal{P}(\mathbf{m}) \}$ with $\mathbf{w}
+ju\mathbf{m} \not\in \X$.
\end{lem}

\begin{proof} There exist $\mathbf{s},\mathbf{t}
\in \Z^r$ with $\mathbf{s} \leq \mathbf{t}$ and
$\mathcal{P}(\mathbf{t} - \mathbf{s}) \subseteq
\mathcal{P}(\mathbf{m})$ for which $\X = \X(\mathbf{s},\mathbf{t})$.
Choose $u \in \N$ such that $u\mathbf{m} \geq \mathbf{t} -
\mathbf{s}$.

For an arbitrary $\mathbf{w} \in \Z^r$, set $\mathcal{I}(\mathbf{w})
= \left\{ i \in \{1, \ldots, r\} : s_i \leq w_i < t_i\right\}$, and
observe that $\mathcal{I}(\mathbf{w}) \subseteq
\mathcal{P}(\mathbf{m})$, and that $ \mathbf{w} \in \X$ if and only
if $\mathcal{I}(\mathbf{w}) \neq \emptyset$. Note also that, for $i
\in \mathcal{I}(\mathbf{w})$ and $ j \in \N$, we have
$$
(\mathbf{w} + ju\mathbf{m})_i = w_i + jum_i \geq s_i + um_i \geq s_i
+ t_i - s_i = t_i,
$$
so that $i \not\in \mathcal{I}(\mathbf{w}+ ju\mathbf{m})$. So, for
each $i \in \mathcal{P}(\mathbf{m})$, if there is a $j' \in \nn$
with $i \in \mathcal{I}(\mathbf{w}+ j'u\mathbf{m})$, then $i \not\in
\mathcal{I}(\mathbf{w}+ ju\mathbf{m})$ for all $j > j'$. This means
that, for each $i \in \mathcal{P}(\mathbf{m})$, there is at most one
$j' \in \nn$ with $i \in \mathcal{I}(\mathbf{w}+ j'u\mathbf{m})$. By
the pigeon-hole principle, it is therefore possible to choose a $j
\in \{0, \ldots, \#\mathcal{P}(\mathbf{m}) \}$ for which
$\mathcal{I}(\mathbf{w}+ ju\mathbf{m})\cap\mathcal{P}(\mathbf{m}) =
\emptyset$, and then $\mathbf{w}+ ju\mathbf{m} \not\in \X$.
\end{proof}

The concept introduced in the next definition can be regarded as a
multi-graded analogue of one defined by Marley in \cite[\S
2]{Marle95}.

\begin{defi}\label{tm.13def} Let $\mathcal{Q} \subseteq \{1, \ldots,
r\}$, and let $\fb$ be an $\nn^r$-graded ideal of $R$. We define the
{\em $\mathcal{Q}$-finiteness dimension\/}
$g^{\mathcal{Q}}_{\frak{b}}(M)$ {\em of\/} $M$ {\em with respect
to\/} $\frak{b}$ by
\[
g^{\mathcal{Q}}_{\frak{b}}(M) := \sup\!\left\{ k \in \nn :
\mbox{~for all~}i<k,\mbox{~there exists a $\mathcal{Q}$-domain
$\X_i$ in $\Z^r$ with~} \mathcal{S}(H^i_{\frak{b}}(M)) \subseteq
\X_i \right\},
\]
if this supremum exists, and $\infty$ otherwise.
\end{defi}

\begin{ex}\label{tm.13ex} For $R$ as in Example \ref{mpbex}, we have
$$
g^{\emptyset}_{R_+}(R)= 2, \qquad g^{\{1\}}_{R_+}(R)= 3, \qquad
g^{\{2\}}_{R_+}(R)= 2, \qquad g^{\{1,2\}}_{R_+}(R)= 5.
$$
\end{ex}

\begin{rmks}\label{tm.13}
The first three of the statements below are immediate from Remarks
\ref{tm.6}(i),(ii),(iii) respectively.

\begin{enumerate}
\item In the case where $\mathcal{Q} = \emptyset$, we have
$g^{\emptyset}_{\frak{b}}(M) = \inf\!\left\{ i \in \nn :
H^i_{\frak{b}}(M) \neq 0\right\}$ (with the usual convention that
the infimum of the empty set of integers is interpreted as
$\infty$).
\item If $\mathcal{Q} \subseteq \mathcal{Q}' \subseteq \{1, \ldots,
r\}$, then $g^{\mathcal{Q}}_{\frak{b}}(M) \leq
g^{\mathcal{Q}'}_{\frak{b}}(M)$.
\item For $\mathbf{n} \in \Z^r$, we have
$g^{\mathcal{Q}}_{\frak{b}}(M(\mathbf{n})) =
g^{\mathcal{Q}}_{\frak{b}}(M)$.
\item Let $\left( \mathcal{Q}_{\lambda}\right)_{\lambda \in
\Lambda}$ be a family of subsets of $\{1, \ldots, r\}$. Set
$$ \Omega := \left\{\bigcap_{\lambda \in \Lambda} \X_{\lambda} :
\X_{\lambda} \mbox{~is a $\mathcal{Q}_{\lambda}$-domain in $\Z^r$
for all~}\lambda \in \Lambda\right\}.
$$
It is straightforward to check that
$$
\inf\!\left\{ g^{\mathcal{Q}_{\lambda}}_{\frak{b}}(M) : \lambda \in
\Lambda\right\} = \sup\!\left\{ k \in \nn : \mbox{~for
all~}i<k,\mbox{~there exists $\Y_i \in \Omega$ with~}
\mathcal{S}(H^i_{\frak{b}}(M)) \subseteq \Y_i \right\}.
$$
\item Since a subset of $\Z^r$ is finite if and only if it is contained
in a set of the form $\bigcap_{j=1}^r \X_j$, where $\X_j$ is a
$\{j\}$-domain in $\Z^r$ for all $j \in \{1, \ldots, r\}$, it
therefore follows from part (iv) that
$$
\min\!\left\{ g^{\{1\}}_{\frak{b}}(M), \ldots ,
g^{\{r\}}_{\frak{b}}(M)\right\} = \sup\!\left\{ k \in \nn :
\mathcal{S}(H^i_{\frak{b}}(M)) \mbox{~is finite for all~}
i<k\right\}.
$$
Thus we can say that $\min\!\left\{ g^{\{1\}}_{\frak{b}}(M), \ldots
, g^{\{r\}}_{\frak{b}}(M)\right\}$ identifies the smallest integer
$i$ (if there be any) for which $H^i_{\frak{b}}(M)$ is not finitely
graded.
\end{enumerate}
\end{rmks}

\begin{prop}\label{tm.12} Let $\mathbf{m} \in \nn^r
\setminus \{\mathbf{0}\}$, and let $f \in \N$. Assume that $M$ is
finitely generated. The following statements are equivalent:

\begin{enumerate}
\item $R_{\mathbf{m}} \subseteq \sqrt{(0:_RH^i_{\fb}(M))}$ for all integers $i
< f$;
\item for each integer $i < f$, there is a
$\mathcal{P}(\mathbf{m})$-domain $\X_i$ in $\Z^r$ such that
$\mathcal{S}(H^i_{\fb}(M)) \subseteq \X_i$, that is $f \leq
g^{\mathcal{P}(\mathbf{m})}_{\fb}(M)$;
\item there is a
$\mathcal{P}(\mathbf{m})$-domain $\X$ in $\Z^r$ such that
$\mathcal{S}(H^i_{\fb}(M)) \subseteq \X$ for all integers $i < f$.
\end{enumerate}
\end{prop}

\begin{proof} (ii) $\Leftrightarrow$ (iii) This is immediate from
Remark \ref{tm.6}(vi).

(iii) $\Rightarrow$ (i) Assume that statement (iii) holds. By Lemma
\ref{tm.11}, there exist $u, v := \#\mathcal{P}(\mathbf{m}) \in \N$
such that, for each $\mathbf{n} \in \Z^r$, there exists
$j(\mathbf{n}) \in \{0, \ldots, v\}$ with $\mathbf{n} +
j(\mathbf{n})u\mathbf{m} \not\in\X$. So, for each $\mathbf{n} \in
\Z^r$ and each integer $i < f$, we have $H^i_{\fb}(M)_{\mathbf{n} +
j(\mathbf{n})u\mathbf{m}} = 0$ and
$$
R_{vu\mathbf{m}}H^i_{\fb}(M)_{\mathbf{n}} =
R_{vu\mathbf{m}-j(\mathbf{n})u\mathbf{m}}R_{j(\mathbf{n})u\mathbf{m}}H^i_{\fb}(M)_{\mathbf{n}}
\subseteq
R_{vu\mathbf{m}-j(\mathbf{n})u\mathbf{m}}H^i_{\fb}(M)_{\mathbf{n}+j(\mathbf{n})u\mathbf{m}}
= 0.
$$
Therefore $R_{vu\mathbf{m}}H^i_{\fb}(M) = 0$ for all integers $i <
f$, and hence $$(R_{\mathbf{m}})^{vu} \subseteq R_{vu\mathbf{m}}
\subseteq (0:_RH^i_{\fb}(M))\quad \mbox{for all~} i < f.$$

(i) $\Rightarrow$ (ii) Assume that statement (i) holds. We argue by
induction on $f$. When $f = 1$, the desired conclusion is immediate
from Lemma \ref{tm.4} (applied to $H^0_{\fb}(M)$).

So assume now that $f > 1$ and that statement (ii) has been proved
for smaller values of $f$. This inductive hypothesis implies that
there exist $\mathcal{P}(\mathbf{m})$-domains $\X_0, \ldots,
\X_{f-2}$ in $\Z^r$ such that $\mathcal{S}(H^i_{\fb}(M)) \subseteq
\X_i$ for all $i \in \{0, \ldots, f-2\}$. It thus remains to find a
$\mathcal{P}(\mathbf{m})$-domain $\X_{f-1}$ in $\Z^r$ such that
$\mathcal{S}(H^{f-1}_{\fb}(M)) \subseteq \X_{f-1}$.

Set $\overline{M} := M/\Gamma_{R_{\mathbf{m}}R}(M)$, and observe
that $R_{\mathbf{m}} \subseteq
\sqrt{(0:_R\Gamma_{R_{\mathbf{m}}R}(M))}$. It therefore follows from
Lemma \ref{tm.9} that there is a $\mathcal{P}(\mathbf{m})$-domain
$\X'$ in $\Z^r$ such that
$\mathcal{S}(H^{f-1}_{\fb}(\Gamma_{R_{\mathbf{m}}R}(M))) \subseteq
\X'$. In view of the exact sequence of $\Z^r$-graded $R$-modules
$$
H^{f-1}_{\fb}(\Gamma_{R_{\mathbf{m}}R}(M)) \lra H^{f-1}_{\fb}(M)
\lra H^{f-1}_{\fb}(\overline{M})
$$
and Remark \ref{tm.6}(vi), it is now enough for us to show that
$\mathcal{S}(H^{f-1}_{\fb}(\overline{M}))$ is contained in a
$\mathcal{P}(\mathbf{m})$-domain in $\Z^r$.

As $R_{\mathbf{m}} \subseteq
\sqrt{(0:_RH^j_{\fb}(\Gamma_{R_{\mathbf{m}}R}(M)))}$ for all $j \in
\nn$, the exact sequence
$$
H^{i}_{\fb}(M) \lra H^{i}_{\fb}(\overline{M}) \lra
H^{i+1}_{\fb}(\Gamma_{R_{\mathbf{m}}R}(M))
$$
shows that $R_{\mathbf{m}} \subseteq
\sqrt{(0:_RH^i_{\fb}(\overline{M}))}$ for all integers $i < f$. Set
$\Ass_R(\overline{M}) =:\left\{ \fp_1, \ldots, \fp_k\right\}$. As
$R_{\mathbf{m}}R$ does not consist entirely of zero-divisors on
$\overline{M}$, we have $R_{\mathbf{m}} \not\subseteq \fp_i$ for
each $i = 1, \ldots, k$. Therefore, by Lemma \ref{tm.10}, there
exists $u' \in \N$ such that $R_{u'\mathbf{m}} \not\subseteq
\bigcup_{i=1}^k\fp_i$, and hence there exists $y' \in
R_{u'\mathbf{m}}$ which is not a zero-divisor on $\overline{M}$. We
can now take a sufficiently high power $y$ of $y'$ to find $u \in
\N$ and $y \in R_{u\mathbf{m}}$ such that
$R_{u\mathbf{m}}H^{f-1}_{\fb}(\overline{M}) = 0$ and $y$ is a
non-zero-divisor on $\overline{M}$, so that there is a short exact
sequence of $\Z^r$-graded $R$ modules
$$
0 \lra \overline{M}(-u\mathbf{m}) \stackrel{y}{\lra}\overline{M}
\lra \overline{M}/y\overline{M} \lra 0.
$$
It now follows from the long exact sequence of local cohomology
modules induced from the above short exact sequence that
$R_{\mathbf{m}} \subseteq
\sqrt{(0:_RH^i_{\fb}(\overline{M}/y\overline{M}))}$ for all integers
$i < f-1$. Therefore, by the inductive hypothesis, there is a
$\mathcal{P}(\mathbf{m})$-domain $\X''$ in $\Z^r$ such that
$\mathcal{S}(H^{f-2}_{\fb}(\overline{M}/y\overline{M})) \subseteq
\X''$. Let $K$ be the kernel of the map $H^{f-1}_{\fb}(\overline{M})
\lra H^{f-1}_{\fb}(\overline{M})(u\mathbf{m})$ provided by
multiplication by $y$. The long exact sequence of local cohomology
modules induced from the last-displayed short exact sequence now
shows that $\mathcal{S}(K) \subseteq \X'' - u\mathbf{m}$.

We now apply Lemma \ref{tm.7}(i) to $H^{f-1}_{\fb}(\overline{M})$,
with $u\mathbf{m}$ playing the r\^oles of both $\mathbf{m}$ and
$\mathbf{k}$: the conclusion is that there exists $v \in \nn$ such
that
$$\mathcal{S}(H^{f-1}_{\fb}(\overline{M})) \subseteq
\bigcup_{j=0}^v (\mathcal{S}(K) - ju\mathbf{m}) \subseteq
\bigcup_{j=0}^v (\X'' - u\mathbf{m} - ju\mathbf{m}).
$$
We can now use Remarks \ref{tm.6}(iii),(vi) to deduce the existence
of a $\mathcal{P}(\mathbf{m})$-domain $\X_{f-1}$ in $\Z^r$ such that
$\mathcal{S}(H^{f-1}_{\fb}(\overline{M})) \subseteq \X_{f-1}$. With
this, the proof is complete.
\end{proof}

We now connect the concept of $\mathcal{Q}$-finiteness dimension of
$M$ with respect to $\frak{b}$, introduced in Definition
\ref{tm.13}, with the concept of $\fa$-finiteness dimension of $M$
relative to $\frak{b}$ (where $\fa$ is a second ideal of $R$),
studied by Faltings in \cite{Falti78}. (See also \cite[Chapter
9]{LC}.)

\begin{rmd}\label{tm.14} Assume that $M$ is finitely generated, and
let $\fa, \fd$ be ideals of $R$ (not necessarily graded).

The {\it $\fa$-finiteness dimension\/} $f^{\fa}_{\fd}(M)$ {\it of\/}
$M$ {\it relative to\/} $\fd$ is defined by
$$  f^{\fa}_{\fd}(M)  =
\inf\!\left\{ i \in \nn : \fa \not\subseteq \sqrt{(0 :
H^i_{\fd}(M))} \right\}
$$
and the {\it $\fa$-minimum $\fd$-adjusted depth
$\lambda^{\fa}_{\fd}(M)$ of $M$} is defined by
$$
\lambda^{\fa}_{\fd}(M) := \inf\!\left\{ \depth M_{\frak{p}} +
\height (\fd + \frak{p})/\frak{p} : \frak{p} \in \Spec(R) \setminus
\Var(\fa)  \right\}.
$$
(Here, $\Var(\fa)$ denotes the {\em variety $\{\fp \in \Spec (R) :
\fp \supseteq \fa\}$ of $\fa$\/.}) It is always the case that
$f^{\fa}_{\fd}(M) \leq \lambda^{\fa}_{\fd}(M)$; Faltings' (Extended)
Annihilator Theorem \cite{Falti78} states that if $R$ admits a
dualizing complex or is a homomorphic image of a regular ring, then
$f^{\fa}_{\fd}(M) = \lambda^{\fa}_{\fd}(M)$. (See \cite[Corollary
3.8]{BRS} for an account of the extended version of Faltings'
Annihilator Theorem.)
\end{rmd}

\begin{rmk}\label{tm.14B} Let the situation be as in Reminder
\ref{tm.14}, let $K \subseteq R$, and let $\left( K_{j}\right)_{j
\in J}$ be a family of subsets of $R$.

\begin{enumerate}
\item It is easy to deduce from the definition that $f^{KR}_{\fd}(M) =
\inf\!\left\{f^{aR}_{\fd}(M) : a \in K\right\}$.
\item We can then deduce from part (i) that $f^{(\bigcup_{j \in J}K_{j})R}_{\fd}(M) =
\inf\!\left\{f^{K_{j}R}_{\fd}(M) : j \in J \right\}$.
\item Similarly, it is easy to deduce from the definition that $\lambda^{KR}_{\fd}(M) =
\inf\!\left\{\lambda^{aR}_{\fd}(M) : a \in K\right\}$.
\item We can then deduce from part (iii) that $\lambda^{(\bigcup_{j \in J}K_{j})R}_{\fd}(M) =
\inf\!\left\{\lambda^{K_{j}R}_{\fd}(M) : j \in J \right\}$.
\end{enumerate}
\end{rmk}

\begin{thm}\label{tm.15} Assume that $M$ is
finitely generated, and let $\emptyset \neq \mathcal{T} \subseteq
\nn^r$.
\begin{enumerate}
\item We have
\begin{align*}
\sup\!\left\{ k \in\nn \right. &: \mbox{for all $i < k$ and all
$\mathbf{m} \in \mathcal{T}$, there exists a
$\mathcal{P}(\mathbf{m})$-domain $\X_i^{(\mathbf{m})}$ in $\Z^r$} \\
&\left.\mbox{~such that
$\mathcal{S}(H^i_{\fb}(M)) \subseteq \X_i^{(\mathbf{m})}$}\right\}\\
& = \inf\!\left\{ g^{\mathcal{P}(\mathbf{m})}_{\fb}(M) : \mathbf{m}
\in \mathcal{T} \right\}\\
& = f^{\sum_{\mathbf{m} \in \mathcal{T}}R_{\mathbf{m}}R}_{\fb}(M)
\leq \lambda^{\sum_{\mathbf{m} \in
\mathcal{T}}R_{\mathbf{m}}R}_{\fb}(M).
\end{align*}
\item If $R$ admits a dualizing complex or is a homomorphic image of a regular ring, then we can replace the
inequality in part\/ {\rm (i)} by equality.
\end{enumerate}
\end{thm}

\begin{proof} Apply Remark \ref{tm.13}(iv) to the family
$\left(\mathcal{P}(\mathbf{m})\right)_{\mathbf{m} \in \mathcal{T}}$
of subsets of $\{1, \ldots, r\}$ to conclude that
\begin{align*}
\sup\!\left\{ k \in\nn \right. &: \mbox{for all $i < k$ and all
$\mathbf{m} \in \mathcal{T}$, there exists a
$\mathcal{P}(\mathbf{m})$-domain $\X_i^{(\mathbf{m})}$ in $\Z^r$} \\
&\left.\mbox{~such that
$\mathcal{S}(H^i_{\fb}(M)) \subseteq \X_i^{(\mathbf{m})}$}\right\}\\
& = \inf\!\left\{ g^{\mathcal{P}(\mathbf{m})}_{\fb}(M) : \mathbf{m}
\in \mathcal{T} \right\}.
\end{align*}
By Proposition \ref{tm.12}, we have
$g^{\mathcal{P}(\mathbf{m})}_{\fb}(M) =
f^{R_{\mathbf{m}}R}_{\fb}(M)$ for all $\mathbf{m} \in \mathcal{T}$.
Therefore, on use of Remark \ref{tm.14B}(ii), we deduce that
$$ \inf\!\left\{ g^{\mathcal{P}(\mathbf{m})}_{\fb}(M) : \mathbf{m}
\in \mathcal{T} \right\} = \inf\!\left\{
f^{R_{\mathbf{m}}R}_{\fb}(M) : \mathbf{m} \in \mathcal{T} \right\} =
f^{\sum_{\mathbf{m} \in \mathcal{T}}R_{\mathbf{m}}R}_{\fb}(M).
$$
We can now use Faltings' (Extended) Annihilator Theorem
\cite{Falti78} (see Reminder \ref{tm.14}) to complete the proof of
part (i) and to obtain the statement in part (ii).
\end{proof}

\begin{cor}\label{tm.16} Assume that $M$ is
finitely generated.
\begin{enumerate}
\item For each non-empty set $\mathcal{T} \subseteq \mathbf{1} +
\nn^r$, we have $$f^{\sum_{\mathbf{m} \in
\mathcal{T}}R_{\mathbf{m}}R}_{\fb}(M) = g^{\{1, \ldots,
r\}}_{\fb}(M).$$
\item For each set $\mathcal{T} \subseteq \nn^r \setminus \{\mathbf{0}\}$ such that
$\N \mathbf{e}_i \cap \mathcal{T} \neq \emptyset$ for all $i \in
\{1, \ldots, r\}$, we have \begin{align*}f^{\sum_{\mathbf{m} \in
\mathcal{T}}R_{\mathbf{m}}R}_{\fb}(M) &=  \sup\!\left\{ k \in \nn :
\mathcal{S}(H^i_{\frak{b}}(M)) \mbox{~is finite for all~}
i<k\right\}\\ &= \sup\!\left\{ k \in \nn : H^i_{\frak{b}}(M)
\mbox{~is finitely graded for all~} i<k\right\}.
\end{align*}
\item If $M \neq \fb M$, then $f^R_{\fb}(M) = g^{\emptyset}_{\fb}(M)
= \grade_M\fb$.
\end{enumerate}
\end{cor}

\begin{note} If, in the case where $r=1$, we take $\mathcal{T} =
\N$, so that $\sum_{m \in \mathcal{T}}R_{m}R = R_+$, then the
statement in part (ii) becomes
$$f^{R_+}_{\fb}(M) = \sup\!\left\{ k \in \nn : H^i_{\frak{b}}(M)
\mbox{~is finitely graded for all~} i<k\right\},$$ a result proved
by Marley in \cite[Proposition 2.3]{Marle95}.
\end{note}

\begin{proof} (i) By Theorem \ref{tm.15}(i), we have $f^{\sum_{\mathbf{m} \in
\mathcal{T}}R_{\mathbf{m}}R}_{\fb}(M) = \inf\!\left\{
g^{\mathcal{P}(\mathbf{m})}_{\fb}(M) : \mathbf{m} \in \mathcal{T}
\right\}$. But $\mathcal{P}(\mathbf{m}) = \{1,\ldots,r\}$ for all
$\mathbf{m} \in \mathbf{1} + \nn^r$.

(ii) By Theorem \ref{tm.15}(i), we have $$f^{\sum_{\mathbf{m} \in
\mathcal{T}}R_{\mathbf{m}}R}_{\fb}(M) = \inf\!\left\{
g^{\mathcal{P}(\mathbf{m})}_{\fb}(M) : \mathbf{m} \in \mathcal{T}
\right\}.$$ By the hypothesis, for each $i \in \{1, \ldots, r\}$,
there exists $\mathbf{m}_i \in \mathcal{T}$ with
$\mathcal{P}(\mathbf{m}_i) = \{i\}$. It therefore follows from
Remark \ref{tm.13}(ii) that $\inf\!\left\{
g^{\mathcal{P}(\mathbf{m})}_{\fb}(M) : \mathbf{m} \in \mathcal{T}
\right\} = \min\!\left\{ g^{\{1\}}_{\frak{b}}(M), \ldots ,
g^{\{r\}}_{\frak{b}}(M)\right\}$. However, we noted in Remark
\ref{tm.13}(v) that $$\min\!\left\{ g^{\{1\}}_{\frak{b}}(M), \ldots
, g^{\{r\}}_{\frak{b}}(M)\right\} = \sup\!\left\{ k \in \nn :
\mathcal{S}(H^i_{\frak{b}}(M)) \mbox{~is finite for all~}
i<k\right\}.$$

(iii) Since $R = R_{\mathbf{0}}R$, we can deduce from Theorem
\ref{tm.15}(i) and Remark \ref{tm.13}(i) that
$$
f^R_{\fb}(M) = f^{R_{\mathbf{0}}R}_{\fb}(M) =
g^{\mathcal{P}(\mathbf{0})}_{\fb}(M) = g^{\emptyset}_{\fb}(M) =
\sup\!\left\{ k \in \nn : H^i_{\frak{b}}(M) =0 \mbox{~for all~} i <
k\right\} = \grade_M\fb.
$$
\end{proof}

\bibliographystyle{amsplain}

\begin{thebibliography}{99}

\bibitem{Bass63}H. Bass, \textit{On the ubiquity of Gorenstein rings},
 Math.\ Z.\ \textbf{82} (1963), 8--28.

\bibitem{Brodm01}M. Brodmann, \textit{Cohomological invariants of coherent
 sheaves over projective schemes: a survey\/}, in \textit{Local
 cohomology and its applications (Guanajuato, 1999)\/}, Lecture Notes
 in Pure and Appl.\ Math.\ \textbf{226}, Dekker, New York, 2002,
 pp.\ 91--120.

\bibitem{BRS}M. P. Brodmann, Ch.\ Rotthaus and R. Y. Sharp,
 \textit{On annihilators and associated primes of local cohomology
 modules\/}, J. Pure and Applied Algebra \textbf{153} (2000),
 197--227.

\bibitem{LC}M. P. Brodmann and R. Y. Sharp,  \textit{Local cohomology: an algebraic
 introduction with geometric applications\/},
 Cambridge Studies in
 Advanced Mathematics \textbf{60}, Cambridge University Press, 1998.

\bibitem{Falti78}G. Faltings, \textit{\"{U}ber die Annulatoren lokaler
 Kohomologiegruppen\/}, Archiv der Math.\ \textbf{30} (1978), 473--476.

\bibitem{Fos-Fox74}R. Fossum and H.-B. Foxby,
 \textit{The category of graded modules},
 Math.\ Scand.\ \textbf{35} (1974), 288--300.

\bibitem{Fumas99}S. Fumasoli, \textit{Die K\"unnethrelation in
 abelschen Kategorien und ihre Anwendungen auf die
 Idealtransformation\/}, Diplomarbeit, Universit\"at Z\"urich, 1999.

\bibitem{GotWat78}S. Goto and K.-i. Watanabe, \textit{On graded rings\/} \rm{II}
 \textit{($\Z^n$-graded rings)\/}, Tokyo J. Math.\ \textbf{1} (1978), 237--261.

\bibitem{ha}Huy T\`ai H\`a, \textit{Multigraded regularity, $a\sp *$-invariant and
 the minimal free resolution\/}, J. Algebra \textbf{310} (2007), 156--179.

\bibitem{Hyry99}E. Hyry, \textit{The diagonal subring and the Cohen-Macaulay property of a
 multigraded ring\/}, Trans. Amer. Math. Soc. \textbf{351} (1999), 2213--2232.

\bibitem{KreRob00}M. Kreuzer and L. Robbiano, \textit{Computational
 commutative algebra\/} 1, Springer, Berlin, 2000.

\bibitem{MacSmi05}D. Maclagan and G. G. Smith, \textit{Uniform bounds on multigraded
 regularity\/}, J. Algebraic Geom.\ \textbf{14} (2005), 137--164.

\bibitem{Macla63}S. Mac Lane, \textit{Homology\/}, Die Grundlehren der mathematischen
 Wissenschaften \textbf{114}, Springer, Berlin, 1963.

\bibitem{Marle95}T. Marley, \textit{Finitely graded local cohomology
 modules and the depths of graded algebras\/}, Proc.\ American Math.\
 Soc.\ \textbf{123} (1995), 3601--3607.

\bibitem{70}R. Y. Sharp, \textit{Bass numbers in the graded case, $a$-invariant
 formulas, and an analogue of Faltings' Annihilator Theorem}, J. Algebra
 \textbf{222} (1999), 246--270.

\bibitem{TruIke89}N. V. Trung and S. Ikeda, \textit{When is the Rees
 algebra Cohen--Macaulay?\/}, Communications in Algebra \textbf{17}
 (1989) 2893--2922.

\bibitem{ZS}O. Zariski and P. Samuel, \textit{Commutative algebra\/}, Volume
 I, The University Series in Higher Mathematics, D. Van Nostrand Company, Inc., Princeton, New Jersey, 1958.

\end{thebibliography}

\end{document}